\documentclass[11pt]{article}

\usepackage{amssymb,amsmath,amsfonts,amsthm,mathtools,enumerate,color}
\usepackage{mathrsfs}
\usepackage[toc,page,title,titletoc,header]{appendix}
\usepackage{graphicx,subfig}

\usepackage{paralist}

\usepackage{tikz}
\usetikzlibrary{arrows,positioning,shapes.geometric}

\usepackage{algorithm,algorithmic}

\graphicspath{ {./figures/} }

\usepackage{indentfirst}
\usepackage{multicol}
\usepackage{booktabs}
\usepackage{url}
\usepackage[outdir=./]{epstopdf} %to enable the use of eps

\usepackage{hyperref}
\hypersetup{
    colorlinks=true, %set true if you want colored links
    linktoc=all,     %set to all if you want both sections and subsections linked
    linkcolor=blue,  %choose some color if you want links to stand out
}
\numberwithin{equation}{section}

\newtheorem{Theorem}{Theorem}[section]
\newtheorem{Lemma}[Theorem]{Lemma}
\newtheorem{Proposition}[Theorem]{Proposition}
\newtheorem{Assumption}{H.\!\!}

\theoremstyle{definition}
\newtheorem{Definition}{Definition}[section]

\theoremstyle{remark}
\newtheorem{Remark}{Remark}[section]

 \def\p{\partial} 
\def\to{\rightarrow}
 \def\ol{\overline}    \def\ul{\underline}
\def\Om{\Omega}  \def\om{\omega} %\def\I{ {\rm (I) } }
 
\newcommand{\q}{\quad}   \newcommand{\qq}{\qquad}

\def\l{\label}  \def\f{\frac}  \def\fa{\forall}
\def\b{\beta}  \def\a{\alpha} \def\ga{\gamma}
 
\def\eps{\varepsilon}

 \def\t{\times}  \def\lam{\lambda}
\def\ms{\medskip}

\def\cA{\mathcal{A}}

\def\cC{\mathcal{C}}

\def\cE{\mathcal{E}}
\def\cF{\mathcal{F}}

\def\cK{\mathcal{K}}
\def\cL{\mathcal{L}}
\def\cM{\mathcal{M}}

\def\cO{\mathcal{O}}
\def\cP{\mathcal{P}}

\def\cS{\mathcal{S}}

\def\cV{\mathcal{V}}
\def\cW{\mathcal{W}}

\def\ba{{\textbf{a}}}
\def\bf{{\textbf{f}}}

\def\bA{{\textbf{A}}}

\def\bL{{\textbf{L}}}

\def\N{{\mathbb{N}}}
\def\bP{\mathbb{P}}

\def\R{{\mathbb R}}
\def\bS{\mathbb{S}}

\newcommand{\ex}{\mathbb{E}}

\DeclareMathOperator*{\argmax}{arg\,max}

\newcommand{\lc}
{\mathrel{\raise2pt\hbox{${\mathop<\limits_{\raise1pt\hbox
{\mbox{$\sim$}}}}$}}}

\newcommand{\gc}
{\mathrel{\raise2pt\hbox{${\mathop>\limits_{\raise1pt\hbox{\mbox{$\sim$}}}}$}}}

\newcommand{\ec}
{\mathrel{\raise2pt\hbox{${\mathop=\limits_{\raise1pt\hbox{\mbox{$\sim$}}}}$}}}

\def\bb{\begin{equation}} \def\ee{\end{equation}}
\def\bbn{\begin{equation*}} \def\een{\end{equation*}}

\def\beqn{\begin{eqnarray}}  \def\eqn{\end{eqnarray}}

\def\beqnx{\begin{eqnarray*}} \def\eqnx{\end{eqnarray*}}

\def\bn{\begin{enumerate}} \def\en{\end{enumerate}}

\def\bd{\begin{description}} \def\ed{\end{description}}



\begin{document}

\title{Regularity and stability  of feedback relaxed controls
}
\author{
Christoph Reisinger\thanks{Mathematical Institute, University of Oxford, United Kingdom ({\tt christoph.reisinger@maths.ox.ac.uk, yufei.zhang@maths.ox.ac.uk})}
\and
Yufei Zhang\footnotemark[1]
}
\date{}

\maketitle

%\tableofcontents

\noindent\textbf{Abstract.} 
This paper proposes a relaxed control regularization
with general exploration rewards
 to design robust feedback controls 
for multi-dimensional continuous-time stochastic exit time  problems.
We establish that  
the regularized control problem admits a H\"{o}lder continuous
optimal feedback control,
and 
 demonstrate that 
both the value function and the feedback control 
of the regularized control problem 
are Lipschitz stable  with respect to parameter perturbations.
Moreover, 
we   show that a pre-computed feedback relaxed control 
gives a robust performance  in a perturbed system,
 and derive a first-order sensitivity equation for both the value function and optimal feedback relaxed control.
These stability results  provide a theoretical
justification for  recent reinforcement learning heuristics 
that  including an exploration reward in the optimization objective
leads to more robust decision making.
We finally prove 
 first-order monotone convergence of the value functions for relaxed control problems
with vanishing exploration parameters, 
which  subsequently enables us to construct 
the pure exploitation strategy of the original control problem 
based on  the feedback relaxed controls.


\ms
%%  68Q32: 	Computational learning theory
%% 93E20:      Optimal stochastic control
%%	93B52  	Feedback control
% % 93B35  	Sensitivity (robustness)

\medskip

\section{Introduction}

In this paper, we propose a relaxed control regularization with a class of exploration rewards 
to design robust feedback controls 
for multi-dimensional 
stochastic control problems
in a continuous setting.
In particular,
we shall rigorously    demonstrate that the constructed optimal feedback control  is 
Lipschitz stable with respect to 
  perturbations in 
the underlying model.

Since parameter uncertainty in a given model is practically inevitable,
it is essential but challenging to 
\textit{a priori} evaluate the performance of a pre-computed feedback control 
in a perturbed system, 
and to design   feedback policies
capable of handling model uncertainty.
For instance,
let us consider the following infinite-horizon stochastic control problem.
Suppose  $(\a_t)_{t\ge 0}$ is an admissible  control process taking values in a \textit{finite} action space $\bA$, 
and  the underlying state dynamics follows a controlled stochastic differential equation (SDE) defined as follows: $X^{\a,x}_0=x\in \R^n$, and
$$
dX^{\a,x}_t= b(X^{\a,x}_t,\a_t) \,dt+ \sigma(X^{\a,x}_t,\a_t)\,dW_t, \q t\ge 0,
$$
where  $b:\R^n\t \bA\to \R^n$ and $\sigma:\R^n\t \bA\to \R^{n\t n}$ are  given  coefficients.
The aim of the controller is to maximize the total expected discounted  reward
over all admissible strategies. 
It is well-known that  
%%\color{blue}
(see e.g.~\cite[Corollary 5.1  on p.~167]{fleming2006} and 
Theorem \ref{thm:verification} for more precise statements),
\color{black}
under  certain regularity assumptions,
the optimal control strategy can be represented as a deterministic function $\a^u:\R^n\to \bA$, called the optimal feedback control, which maps the current state space  into the action space.
Moreover, one can construct such an optimal feedback control $\a^u$ 
%%\color{blue} 
via a verification
argument,
which consists of 
  solving 
a nonlinear Hamilton--Jacobi--Bellman (HJB)
partial differential 
equation (PDE) arising from the dynamic programming principle
for the optimal reward function $u$,
and then
performing a pointwise maximization of the associated Hamiltonian
involving  the function $u$ and its derivatives $(\p_i u,\p_{ij}u)_{i,j=1}^n$ as follows:
 for any given $x\in\R^n$,
\bb\l{eq:ctrl_intro}
\a^u(x)\in \argmax_{\a\in \bA}\bigg[\sum_{i,j=1}^na^{ij}(x,\a)\p_{ij}u(x)+\sum_{i=1}^n b^i(x,\a)\p_iu(x)-c(x,\a) u(x)+f(x,\a)\bigg],
\ee
where $a(x,\a)=\sigma(x,\a)\sigma^T(x,\a)/2$, the functions $c$ and $f$ denote the discount rate and the instantaneous reward, respectively.
%and the function $u$ and its derivatives $(\p_i u,\p_{ij}u)_{i,j=1}^n$ satisfy a  Hamilton--Jacobi--Bellman (HJB)
%partial differential 
%equation (PDE)
%involving parameters $(b,\sigma,c,f)$.
We refer the reader to Theorems \ref{thm:verification} and \ref{thm:verification_relax}
for rigorous arguments of the above procedure
for  control problems of our interest,  
and to \cite[Theorem 5.1  on p.~166]{fleming2006} for a general statement.
\color{black}

We observe, however, that  the control strategy $\a^u$ satisfying  \eqref{eq:ctrl_intro} 
 in general is  difficult to implement and 
  unstable to parameter perturbations,
which in practice would result in numerical instability of learning algorithms.  
Due to the finiteness of the action space $\bA$  and 
the fact that  $\argmax$ is a set-valued mapping,
 a function $\a^u:\R^n\to \bA$ satisfying \eqref{eq:ctrl_intro} 
 in general is non-unique and merely measurable,
 and hence it is hard to  follow such an irregular strategy in practice.
More importantly, the discreteness of the set $\bA$  implies that the $\argmax$ mapping is not continuous (in the sup-norm), which makes the feedback control $\a^u$ very sensitive to 
perturbations of the coefficients $(b,\sigma,c,f)$. 
In other words, a slight change of the model parameters will result in a significant change of the feedback control, 
especially in the regions where two or more actions lead to 
similar performances based on the current model.
Since it is difficult to determine the occurance of such regions \textit{a priori},
it is unclear how well the control strategy $\a^u$ will perform in a real system with the perturbed coefficients $(\tilde{b},\tilde{\sigma},\tilde{c},\tilde{f})$,
even if  $(\tilde{b},\tilde{\sigma},\tilde{c},\tilde{f})$ is very close to $(b,\sigma,c,f)$.
%\color{blue}
See the last paragraph of Section \ref{sec:exit}
for  
more details on the instability  of  feedback controls
and its practical impact on learning algorithms.
\color{black}

A tremendous amount of effort has been made to overcome the above difficulties,  particularly in the 
(discrete-time) Reinforcement Learning (RL) setting (see e.g.~\cite{sutton1998}), where 
the agent seeks (nearly) optimal decisions in  a random environment with incomplete information.
Generally speaking,
 the controller must balance  between 
greedily exploiting the available information to choose actions
that maximize short-term  rewards,
and continuously exploring the environment to 
acquire more knowledge for  long-term benefits.
In particular, an entropy-regularized formulation 
has been proposed for solving 
(discrete-time) RL problems in \cite{ziebart2008,nachum2017,geist2019},
where the authors  
incorporate explorations by explicitly including 
 the entropy of the exploration strategy in the optimization objective as a reward function,
and balance exploitation and exploration  by adjusting a weight imposed on this regularization term.
Empirical studies (e.g. \cite{ziebart2008,haarnoja2017, nachum2017,geist2019})
show that 
such a  regularized formulation leads to more  robust decision making. 
Recently, the authors in \cite{wang2019,wang2019_mv} 
 extended  this entropy-regularized formulation to continuous-time RL problems by using 
the relaxed control framework,
and  study the exploration/exploitation trade-off for one-dimensional linear-quadratic (LQ) control problems
via explicit solutions.
The relaxed control approach has then been extended to (discrete-time) 
RL problems with mean-field controls in \cite{gu2019}.

In this work, 
we propose an exploratory framework with general exploration rewards to design 
robust feedback controls for continuous-time stochastic exit time problems with continuous state space and 
discrete action space. 
Our  formulation extends the relaxed control approach in \cite{wang2019,wang2019_mv}
to multi-dimensional state dynamics and general exploration rewards, 
including   Shannon's differential entropy 
and other commonly used regularization functions in the optimization literature 
(see e.g.~\cite{chen1995,zang1980}); see the remark at the end of Section \ref{sec:relax} for a detailed comparison 
among different exploration reward functions.

A major theoretical contribution of this work is a rigorous  stability  analysis 
of the regularized control problem and  its associated feedback  control strategy. 
Although the entropy-regularized RL formulation has demonstrated remarkable robustness in various empirical studies (e.g. \cite{ziebart2008,haarnoja2017, nachum2017,geist2019,gu2019, wang2019_mv}), to the best of our knowledge, there is no published theoretical work on 
 the Lipschitz stability of  \textit{feedback relaxed controls} with respect to   parameter uncertainty (even in a discrete-time setting)
nor on the Lipschitz stability of the value functions for regularized 
continuous-time stochastic control problems
with general multi-dimensional nonlinear state dynamics.
%\color{blue}
In fact, most existing results on the Lipschitz stability of feedback controls
 are for LQ control problems
 with linear state dynamics and quadratic cost functions
 (see e.g.~\cite{mania2019} for discrete-time LQ problems in an ergodic setting
 and 
\cite{basei2020} for finite-horizon continuous-time LQ problems).
The stability analysis  of  such problems relies heavily on the linearity of optimal feedback controls and 
the associated  Riccati equations,
and hence cannot be directly extended to general nonlinear control problems.
We refer the reader also to 
\cite{aldous1981,langen1981, backhoff2020,backhoff2020_estimate, bayraktar2020, {bayraktar2020_extended}, kara2020}
for the  continuity of 
various 
stochastic optimization problems, 
including stochastic control problems and optimal stopping problems,
in the underlying processes
 with respect to the
 (extended) weak topology.

\color{black}

In this work, we shall close the gap by  providing a theoretical
justification for  recent RL heuristics 
that  including an exploration reward in the optimization objective
leads to more robust decision making.
In particular, we shall demonstrate that
%
% that %by choosing a positive exploration parameter,  
 the change in value functions of the regularized  control problems (in the $C^{2,\b}$-norm)
depends Lipschitz-continuously  on 
the perturbations of the model parameters, 
including the coefficients of the state dynamics and reward functions in the optimization objective.
We shall also prove that the regularized control problem 
admits a H\"{o}lder continuous feedback control 
(cf.~the original control $\a^u$ in \eqref{eq:ctrl_intro} is merely measurable),
which is Lipschitz stable  (in the $C^{\b}$-norm) with respect to  parameter perturbations; see Theorem \ref{thm:u_hat-u}.

Moreover, this is the first paper which 
precisely quantifies the performance of a  feedback control pre-computed based on a given model 
in  a new multi-dimensional controlled dynamics with perturbed coefficients.
We will prove that 
the gap between the suboptimal reward function achieved by 
the pre-computed feedback relaxed control 
 and the optimal reward function of the perturbed relaxed control problem 
 depends Lipschitz-continuously on the magnitude of perturbations in the coefficients (see Theorem \ref{thm:u_sub-u_op}).
% We shall also establish a first-order adjustment for the  value function and optimal feedback relaxed control of the perturbed relaxed control problem,
% which can be efficiently computed by solving a \textit{linear} equation involving only the pre-computed strategy
% and the changes in parameters
%  (see Theorem \ref{thm:sensitivity} and Remark \ref{rmk:sensitivity_ctrl}).
 We  also establish a first-order sensitivity equation for the  value function and  feedback  control of the perturbed relaxed control problem
(see Theorem \ref{thm:sensitivity} and Remark \ref{rmk:sensitivity_ctrl}),
% which can be efficiently computed by solving a \textit{linear} equation involving only the pre-computed strategy
% 
which enables us to 
 quantify the explicit dependence 
of the Lipschitz stability of feedback controls
on the exploration parameter $\eps$
(see Theorem \ref{thm:deltau_eps}).

Let us briefly comment on the two main difficulties encountered in 
 the  stability analysis of feedback relaxed controls
 beyond those encountered in
the finite-dimensional RL setting 
(see e.g.~\cite{forsyth2007,bokanowski2009,geist2019})
and the LQ setting
  (see  e.g.~\cite{mania2019,basei2020}).
As we shall see in \eqref{eq:ctrl_relax},
the feedback relaxed control (in the present continuous setting) is  defined as the pointwise maximizer
of the associated Hamiltonian,
which in general involves 
not only the value function
of the regularized control problem,
but also its  first and second order derivatives.
Hence, 
besides estimating the sup-norm of the value functions as in the finite-dimensional RL setting,
we also need to quantify the impact of parameter uncertainty
on the (first and second order) derivatives of the value functions,
which are  solutions to a fully nonlinear HJB PDEs.
%\color{blue}
For continuous-time LQ problems, 
such an analysis can be greatly 
 simplified by taking advantage of
the quadratic structure of 
 the value function,
 which reduces the study of HJB PDEs to 
 that of  Riccati
 ordinary differential equations.
 Such  a simplification is not possible for 
general nonlinear 
 control problems,
\color{black}
which  requires us to
derive a precise \textit{a priori} estimate for the derivatives of  solutions 
to the associated fully nonlinear HJB equations.

Moreover, the Lipschitz stability and the first-order sensitivity analysis of 
the feedback relaxed controls 
also require us to establish the regularity  of 
the HJB operator
and the $\argmax$-mapping 
between suitable function spaces
for  regularized control problems.
As already pointed out in \cite{smears2014,ito2019},
the fact that the HJB operator is fully nonlinear (since we allow the diffusion coefficients to be controlled)
  poses a significant challenge for
choosing proper function spaces to 
simultaneously ensure the differentiability of the fully nonlinear HJB operator
and the bounded invertibility of its (Fr\'{e}chet) derivative,
which are essential for deriving the sensitivity equations 
of the value functions and feedback controls
(see Theorem \ref{thm:sensitivity} and Remark \ref{rmk:sensitivity_ctrl}).
Here, 
by taking advantage of the  exploration reward functions,
we demonstrate that  the HJB operator and the $\argmax$-mapping for the regularized control problem
are sufficiently smooth between suitable H\"{o}lder spaces,
which together with an elliptic regularity estimate  leads us to the desired 
sensitivity results for the feedback relaxed controls;
see Remark \ref{rmk:Holder_Sobolev} for more details.

Finally, we  establish that, as the exploration parameter tends to zero, 
the value function  of the relaxed control problem  converges monotonically   to that of the classical stochastic control problem with a first-order accuracy (see Theorem \ref{thm:value_converge}).
The convergence of value functions (in  the $C^{2,\b}$-norm)
subsequently enables us to deduce a novel uniform result (on compact sets) for
the feedback relaxed control to a pure exploitation strategy of the original control problem.
We further prove an exact regularization property for a class of  reward functions,
which allows us to recover the pure exploitation  strategy based on the feedback relaxed control \textit{without} sending the exploration parameter  to 0
(see Theorem \ref{thm:ctrl_converge}).

We organize this paper as follows. Section \ref{sec:exit} introduces the stochastic exit control problem, and establishes its connection to  HJB equations. In Section \ref{sec:relax}, we  propose a relaxed  control regularization involving general exploration reward functions for  the stochastic control problem,
and  establish the  H\"{o}lder regularity of the feedback relaxed control strategy.
Then, for a fixed positive exploration parameter,
we prove the Lipschitz stability of the value function and feedback relaxed control 
with respect to parameter perturbations 
in Section \ref{sec:lipschitz},
and derive their first-order sensitivity equations  in Section \ref{sec:sensitivity}.
We establish the convergence of value functions and relaxed control strategies for 
 vanishing exploration parameters  in Section \ref{sec:convergence}.
Appendix \ref{appendix:lemmas} 
is devoted to the proofs of some technical results.

\section{Stochastic exit time problem and HJB equation}\l{sec:exit}
In this section, we introduce the stochastic exit time problem  of our interest, state the main assumptions on its coefficients, and recall its connection with HJB equations. We start with some useful notation which is needed frequently throughout this work.

For any given multi-index $\b=(\b_1,\ldots, \b_n)$ with $\b_i\in \N\cup\{0\}$, $i=1,\ldots, n$, we define $|\b|=\sum_{i=1}^n \b_i$ and 
$
D^\b\phi =\tfrac{\p^{|\b|}\phi}{\p x_1^{\b_1}\ldots \p x_n^{\b_n}}.
$
For any given  open  subset $\cO\subset \R^n$, $k\in \N\cup \{0\}$, $\theta\in (0,1]$, and   function $\phi:\ol{\cO}\to \R$, we define the following semi-norms:
\begin{equation*}
\begin{alignedat}{2}
[\phi]_{0;\ol{\cO}}&=\sup_{x\in \ol{\cO} }|\phi(x)|,
\q
&
[\phi]_{\theta;\ol{\cO}}&=
\sup_{x,y\in \ol{\cO},x\not=y}
\f{|\phi(x)-\phi(y)|}{|x-y|^\theta},
\\
[\phi]_{k,0;\ol{\cO}}&=\sum_{|\b|=k}[D^\b\phi]_{0;\ol{\cO}}, 
\q
&
[\phi]_{k,\theta;\ol{\cO}}&=\sum_{|\b|=k}[D^\b \phi]_{\theta;\ol{\cO}}.
\end{alignedat}
\end{equation*}
Then 
%for any given bounded open subset $\cO\subset \R^n$, and $k\in \N\cup \{0\}$, 
we shall denote by $C^{k}(\ol{\cO})$ the space of  $k$-times continuously differentiable functions in $\ol{\cO}$ equipped with the norm $|\phi|_{k;\ol{\cO}}=\sum_{m=0}^k[\phi]_{m,0;\ol{\cO}}$,
and by $C^{k,\theta}(\ol{\cO})$ the space consisting of all functions in $C^{k}(\ol{\cO})$ satisfying $[\phi]_{k,\theta;\ol{\cO}}<\infty$, equipped with the norm $|\phi|_{k,\theta;\ol{\cO}}=|\phi|_{k;\ol{\cO}}+[\phi]_{k,\theta;\ol{\cO}}$. When $k=0$, we use $C^\theta(\ol{\cO})$ to denote $C^{0,\theta}(\ol{\cO})$, and use $|\cdot|_{\theta;\ol{\cO}}$ to denote $|\cdot|_{0,\theta;\ol{\cO}}$.
We shall omit the subscript 
$\ol{\cO}$ in the (semi-)norms if no confusion appears.

Finally, we shall 
denote by $[a^{ij}]$ the $n\t n$ matrix whose $ij$th-entries are given by $a^{ij}$, 
 by $\bS^{n}$, $\bS^n_0$ and $\bS^n_{>}$, respectively, the set of $n\t n$ symmetric, symmetric  positive semi-definite
 and
 symmetric positive definite matrices, 
 by $X\ge Y$ in $\bS^n$ the fact that $X-Y$ is positive semi-definite.
For any given $K\in \N$, we denote 
  by $\Delta_K$ the probability simplex in $\R^K$, i.e.,
\bb\l{eq:delta_K}
\textstyle \Delta_K=\left\{\lambda\in \R^K  \biggm\vert  \sum_{k=1}^K \lambda_k=1, \lambda_k\ge 0,\, \fa k=1,\ldots, K\right\}.
\ee

Now we are ready to  introduce the control problem of interest. 
In order to allow irregular feedback control strategies, we consider  the following weak formulation of a control problem, which includes the underlying probability space as part of control strategies (see e.g.~\cite{yong1999,fleming2006}).
See Remark \ref{rmk:weak_formulation} for possible extensions to stochastic control problems under strong formulation, for which the underlying probability reference system is fixed.

\begin{Definition}
A $5$-tuple 
$\pi=(\Om, \cF, \{\cF_t\}_{t\ge 0}, \bP, W)$
is said to be a reference probability system 
if  $(\Om, \cF, \{\cF_t\}_{t\ge 0}, \bP)$ is 
a filtered probability space satisfying the usual condition\footnotemark,
\footnotetext{We say  $(\Om, \cF, \{\cF_t\}_{t\ge 0}, \bP)$ satisfies the usual condition if 
$(\Om, \cF, \bP)$ is complete, $\cF_0$ contains all the $\bP$-null sets in $\cF$, 
and $\{\cF_t\}_{t\ge 0}$ is right continuous.}
and $W=(W_t)_{t\ge 0}$ is an  $\{\cF_t\}_{t\ge 0}$-adapted 
$n$-dimensional  Brownian motion.
We denote by $\Pi_{\textnormal{ref}}$
 the set of all reference probability systems.
\end{Definition}

Now let $\cO$ be a given bounded domain in $\R^n$, 
i.e.,  a bounded connected open  subset of $\R^n$.
The aim of the controller is to maximize the expected discounted reward up to the first exit time of
a controlled dynamics from the domain $\cO$. 
More precisely,  let
$\pi=(\Om, \cF,  \{\cF_t\}_{t\ge 0}, \bP, W )\in\Pi_{\textnormal{ref}}$ be a given reference probability system,
and 
$\cA_\pi$ be the set of $\{\cF_t\}_{t\ge 0}$-progressively measurable processes $\a$ taking values in a finite set  $\bA$.
For any given initial
state $x \in \R^n$, and control $\a \in  \cA_\pi$, we consider the controlled  dynamics $X^{\a,x}$
satisfying the following SDE: $X_0^{\a,x}=x$ and 
\bb\l{eq:sde}
dX^{\a,x}_t= b(X^{\a,x}_t,\a_t) \,dt+ \sigma(X^{\a,x}_t,\a_t)\,dW_t, \q t \in ( 0,\infty),
\ee
where 
$b:\R^n\t \bA\to \R^n$ and $\sigma:\R^n\t \bA\to \R^{n\t n}$ 
are given Lipschitz continuous functions (see (H.\ref{assum:D}) for precise conditions),
and denote by 
$\tau^{\a,x}\coloneqq\inf\{t\ge 0\mid X^{\a,x}_t\not \in \cO\}$
the first exit time   of the  dynamics $X^{\a,x}$ from the  domain $\cO$,\footnotemark 
\footnotetext{Note that, 
if $(\Om, \cF, \{\cF_t\}_{t\ge 0}, \bP)$ is a filtered probability space satisfying the usual condition,
$(X_t)_{t\ge 0}$ is an $\{\cF_t\}_{t\ge 0}$-progressively measurable continuous process,
and  $\cO$ is an open subset of $\R^n$, then the first exit time $\tau=\inf\{t\ge 0\mid X^{\a,x}_t\not \in \cO\}$ is an $\{\cF_t\}_{t\ge 0}$-stopping time;
see \cite[Example 3.3 on p.~24]{yong1999}.
}
and by
$(\Gamma^{\a,x}_t)_{t\in [0,\tau^{\a,x}]}$  the controlled discount factor:  
$%\l{eq:gamma}
\Gamma^{\a,x}_t\coloneqq \exp\left(-\int_0^t c(X^{\a,x}_s,\a_s)\,ds\right)$
for all $t\in [0,\tau^{\a,x}]$.
Then, 
for each given $x\in \ol{\cO}$,
we shall  consider the following value function:
%\bb\l{eq:value}
%v(x)=\sup_{\pi\in \Pi_{\textnormal{ref}}}\sup_{\a\in \cA_\pi}J(x,\a),
%\q 
%\textnormal{with }
%J(x,\a)\coloneqq 
%\ex^{\bP}\bigg[\int_0^{\tau^{\a,x}} \Gamma^{\a,x}_s f(X^{\a,x}_s,\a_s)\,ds+ g(X^{\a,x}_{\tau^{\a,x}})\Gamma^{\a,x}_{\tau^{\a,x}}\bigg], 
%\ee
\bb\l{eq:value}
v(x)=\sup_{\pi\in \Pi_{\textnormal{ref}}}\sup_{\a\in \cA_\pi}
\ex^{\bP}\bigg[\int_0^{\tau^{\a,x}} \Gamma^{\a,x}_s f(X^{\a,x}_s,\a_s)\,ds+ g(X^{\a,x}_{\tau^{\a,x}})\Gamma^{\a,x}_{\tau^{\a,x}}\bigg], 
\ee
where the functions $f$ and $g$ denote, respectively, the running reward and the exit reward.

Throughout this work, we shall perform the analysis under the following  assumptions on the coefficients:

\begin{Assumption}\l{assum:D}
Let $n, K\in \N$, $\cK=\{1,\ldots,K\}$,   $\bA$ is a set of cardinality $K$, i.e., $\bA=\{\ba_k\}_{k\in \cK}$,
and $\cO$ be a bounded domain in $\R^n$.
There exist constants $\nu, \Lambda>0$, $\theta\in (0,1]$ such that 
the boundary $\p\cO$ of $\cO$ is of class  $C^{2,\theta}$, $g\in  C^{2,\theta}(\ol{\cO})$, 
and the  functions $b:\R^n\t\bA\to \R^n$, $\sigma:\R^n\t \bA\to \R^{n\t n}$, $c:\ol{\cO}\t \bA\to [0,\infty)$ and $f:\ol{\cO}\t \bA\to \R$ satisfy the following conditions: for each $k\in \cK$, 
\begin{align}
\sigma(x,\ba_k) \sigma^T(x,\ba_k)\ge \nu I_n,
\q \textnormal{for all  $x\in \R^n$,}\l{eq:elliptic}\\
\sum_{i,j} |\sigma^{ij}(\cdot,\ba_k)|_{0,1;\R^n}+\sum_{i} |b^{i}(\cdot,\ba_k)|_{0,1;\R^n}+
|c(\cdot,\ba_k)|_{\theta;\ol{\cO}}+|f(\cdot,\ba_k)|_{\theta;\ol{\cO}}\le \Lambda. \l{eq:holder}
\end{align}

\end{Assumption}

\begin{Remark}\l{rmk:trace}
The Lipschitz continuity of  $b$ and $\sigma$ on $\R^n$ ensures  that,
for any given 
$\pi\in \Pi_{\textnormal{ref}}$,
$\a\in \cA_\pi$
and $x\in \R^n$, the controlled SDE \eqref{eq:sde}  admits a unique strong solution.
Moreover,
the non-degeneracy  of $\sigma$ on $\R^n$ 
 ensures that 
SDEs with non-Lipschitz feedback controls admit a weak solution 
(cf.~Theorems \ref{thm:verification} and \ref{thm:verification_relax}); see also Lemma \ref{lemma:sqrt_A}.

%
%Moreover, the non-degeneracy condition \eqref{eq:elliptic} and the fact that $\p\cO$ is of class $C^{2,\theta}$ assure that 
%there exists a positive $\mu$ such that $\sup_{x\in \ol{\cO},\a\in \cA}\ex[\exp(\mu\tau^{\a,x})]<\infty$ (see e.g.~\cite[Lemma 3.1]{buckdahn2016}),
%which subsequently implies the value function \eqref{eq:value} is finite for all $x\in \ol{\cO}$.

As shown in \cite[Lemma 6.38]{gilbarg1985},  the fact that $\p\cO$ is of class $C^{2,\theta}$ ensures that a function in $C^{2,\theta}(\ol{\cO})$ has boundary values in $C^{2,\theta}(\p{\cO})$,
and  conversely, any function $\phi\in C^{2,\theta}(\p {\cO})$ can be extended to a function in $C^{2,\theta}(\ol{\cO})$.
Hence, one can introduce a boundary norm $ |\cdot|_{2,\theta;\p\cO}$ for the space $C^{2,\theta}(\p{\cO})$, such that for any given $\phi\in C^{2,\theta}(\p{\cO})$, $|\phi|_{2,\theta,\p\cO}=\inf_{\Phi}|\Phi|_{2,\theta;\ol{\cO}}$, where
$\Phi\in C^{2,\theta}(\ol{\cO})$ is a global extension of $\phi$ to $\ol{\cO}$. 
The space $C^{2,\theta}(\p{\cO})$ equipped with  the norm $|\cdot|_{2,\theta;\p\cO}$ is a Banach space (see e.g.~the discussions on page 94 in \cite{gilbarg1985}). 
%In what follows we shall consider 
%without loss of generality that 
%the boundary values $g$ belonging to $C^{2,\theta}(\ol{\cO})$ instead of introducing the boundary trace space $C^{2,\theta}(\p{\cO})$.

%For notational simiplicity, we will denote by $\phi_k$ a generic function $\phi(\cdot,\ba_k)$,
%for all $k\in \cK$.
% for its control-dependence.

%\color{blue}
To simplify the presentation, 
we  study exit time control problems with H\"{o}lder continuous  coefficients 
in this work 
and analyze classical solutions  of 
 associated elliptic HJB equations. 
Similar results, including the characterization and Lipchitz stability of feedback 
relaxed controls in Sections \ref{sec:relax} and 
\ref{sec:lipschitz},
 can be obtained for
 finite horizon
  control problems with measurable coefficients,
whose corresponding  parabolic HJB equations admit 
weak solutions in 
suitable  Sobolev spaces (see \cite{smears2015} for the well-posedness of weak solutions to parabolic HJB equations
and \cite[Theorem 1 on p.~122]{krylov1980} for a generalized It\^{o}'s formula).
The first-order sensitivity analysis in Section 
\ref{sec:sensitivity} in general can only be performed 
for classical solutions in H\"{o}lder spaces; see Remark \ref{rmk:Holder_Sobolev} for details.

\color{black}
\end{Remark}

The rest of this section is devoted to the connection between the stochastic exit time problem and a Hamilton-Jacobi-Bellman (HJB) boundary value problem, which plays an essential role in the construction of feedback control strategies.
More precisely, we now consider the following HJB equation  with inhomogeneous Dirichlet boundary data:
\begin{align}\l{eq:hjb}
F_0[u]&\coloneqq 
H_0(\bL u+\bf)=0 \q\textnormal{in $\cO$}, 
%\q\q
\qq
u=g \q \textnormal{on $\p\cO$,}
\end{align}
where $H_0:\R^K\to \R$ is the pointwise maximum function, i.e., $H_0(x)=\max_{k\in \cK} x_k$ for all $x=(x_1,\ldots, x_K)^T\in \R^K$, 
$\bf:\ol{\cO}\to \R^K$ is the function satisfying
$\bf(x)=(f(x,\ba_k))_{k\in \cK}$
for all $x\in \ol{\cO}$, 
and $\bL=(\cL_k)_{k\in \cK}$ is a family of elliptic operators satisfying for all $k\in \cK$, $\phi\in C^2({\cO})$, $x\in \cO$ that 
\bb\l{eq:L_k}
\cL_k \phi(x)\coloneqq a^{ij}_k(x)\p_{ij}\phi(x)+b^i_k(x)\p_i\phi(x)-c_k(x) \phi(x), \q \textnormal{with $a_k\coloneqq\tfrac{1}{2}\sigma_k\sigma^T_k$.}
\ee
Above and hereafter, when there is no ambiguity, we shall 
denote by $\phi_k(\cdot)$ a generic function $\phi(\cdot,\ba_k)$ for all $k\in \cK$,
and 
adopt the summation convention as in \cite{gilbarg1985,chen1998}, i.e., repeated equal dummy indices indicate summation from 1 to $n$.

%It is well-known that under suitable conditions, the value function \eqref{eq:value} can be characterized by the unique continuous viscosity solution of \eqref{eq:hjb} (see e.g.~\cite{fleming2006, buckdahn2016} and references within).
%However, 
Throughout this paper, 
we shall  focus on the classical solution $u\in
 C(\ol{\cO}) \cap C^2(\cO)$ to \eqref{eq:hjb}
 established in the following theorem,
which subsequently enables us to  characterize 
 optimal feedback controls  for \eqref{eq:value}. 
\begin{Theorem}\l{thm:wp}
Suppose (H.\ref{assum:D}) holds,
and let $M=\sup_{i,j,k}|\sigma^{ij}_k|_{0;\ol{\cO}}$.
Then
the Dirichlet problem \eqref{eq:hjb} admits a unique solution $u\in C(\ol{\cO}) \cap C^2(\cO)$. 
Moreover, there exists a constant $\b_0=\b_0(n,\nu,M)\in (0,1)$
and a Borel measurable function $\a^u:\ol{\cO}\to \bA$
 such that $u\in  C^{2,\min(\b_0,\theta)}(\ol{\cO})$ and 
\bb\l{eq:ctrl}
\a^u(x)\in \argmax_{\ba_k\in \bA}
%\bigg(
%\tfrac{1}{2}
%\big(\sigma_k(x)\sigma^T_k(x)\big)^{ij}\p_{ij}u(x)+b^i_k(x)\p_iu(x)-c_k (x)u(x)+f_k(x)
%\bigg)
\big(
\cL_ku(x)+f_k(x)
\big)
\q \fa x\in \ol{\cO}.
\ee
\end{Theorem}

\begin{proof}
We shall only prove the uniqueness of solutions in $C(\ol{\cO}) \cap C^2(\cO)$, 
since
the existence of classical solutions in $C^{2,\min(\b_0,\theta)}(\ol{\cO})$ 
will be established constructively based on 
the relaxed control approximation in Theorem \ref{thm:value_converge}
(see also \cite[Theorem 7.5]{chen1998}
for a proof of existence based on the method of continuity),
and the existence of a Borel measurable function satisfying \eqref{eq:ctrl} follows directly from the measurable 
selection theorem (see \cite[Theorem 18.19]{aliprantis2006}).

Let $u_1,u_2\in C(\ol{\cO}) \cap C^2(\cO)$ be solutions to \eqref{eq:hjb}.
Then 
 for all $x\in \cO$, we can deduce from the fundamental theorem of calculus that
\begin{align*}
0&=H_0(\bL u_1(x)+\bf(x))-H_0(\bL u_2(x)+\bf(x))
=
\int_0^1 h(s,x)^T\bL (u_1-u_2)(x)\,ds
=\tilde{L}(u_1-u_2)(x),
\end{align*}
where
$h:[0,T]\t \cO \to \Delta_K$ is a measurable function, and 
$\tilde{L}$ denotes the elliptic operator  satisfying for all  $\phi\in C^2({\cO})$ and $x\in \cO$ that 
$\tilde{L}\phi(x)=\ \eta^T(x)\bL\phi(x) $
with $\eta(x)=\int_0^1 h(s,x)\,ds$.
In particular, the function $h$ can be chosen as 
the weak limit of the functions 
$([0,T]\t\cO\ni (s,x)\mapsto (\nabla H^\eps_0)(\bL u_2(x)+\bf(x)+s\bL (u_1-u_2)(x))\in \Delta_K)_{\eps>0}$
in $L^2([0,T]\t\cO)$, where $(H^\eps_0)_{\eps>0}$ is a sequence of smooth approximations of $H_0$
obtained by using the standard mollification argument.
Then we can easily show that 
$\eta(x)\in \Delta_K$ for all $x\in \cO$,
$\tilde{L}$ is a uniform elliptic operator, 
and $\sum_{k=1}^K\eta_k(x)c_k(x)\ge 0$
for all $x\in \cO$.
Hence the classical maximum principle (see e.g.~\cite[Theorem 3.7]{gilbarg1985}) 
and $u_1=u_2$ on $\p\cO$ 
imply that $u_1=u_2$ on $\ol{\cO}$, 
which shows that 
 the Dirichlet problem \eqref{eq:hjb} admits at most one solution 
in $C(\ol{\cO}) \cap C^2(\cO) $.
\end{proof}

%Now we proceed to show that the solution to \eqref{eq:hjb} is the value function \eqref{eq:value}, and the measurable function 
%$\a^u$ defined as in \eqref{eq:ctrl} is an optimal feedback control of \eqref{eq:value}.

We now  present a  verification result, i.e., Theorem \ref{thm:verification},
which shows that 
the classical solution to the HJB equation \eqref{eq:hjb} is the value function \eqref{eq:value},
and
the Borel measurable function 
$\a^u$ defined as in \eqref{eq:ctrl} is a feedback control of \eqref{eq:value}.
%which extends the result in \cite[Corollary 5.1 on p.~166]{fleming2006}
%to control problems with a controlled state-dependence discount factor.
%
%
%
The proof will be postponed to Appendix \ref{appendix:lemmas},
which essentially follows from  It\^{o}'s formula and
the existence result of weak solutions to SDEs with non-degenerate diffusion coefficients 
(see  \cite[Theorem 1]{mishura2016}).

We first recall the definition of optimal feedback control (see e.g.~\cite[Definition 6.1]{yong1999}).
%Note that  the reward  functional $J$ only involves the trajectory of the controlled dynamics  up to the first exit time from the open set $\cO$.
%Hence, for any given Borel measurable function $h$ defined on $\ol{\cO}$, we shall 
%extend   it (by setting $h$ to be an arbitrary constant strategy on $\ol{\cO}^c$) to a measurable function on the whole space $\R^n$, without changing the corresponding reward functional $J$.

\begin{Definition}\l{def:feedback}
A Borel measurable function 
$h:\ol{\cO}\to \bA$ is said to be a feedback control of \eqref{eq:value}
if for all $x\in \ol{\cO}$, 
there exists 
$\pi^x=(\Om^{x}, \cF^{x}, \{\cF^{x}_t\}_{t\ge 0}, \bP^{x}, W)\in \Pi_{\textnormal{ref}}$,  
and an $\{\cF^{x}_t\}_{t\ge 0}$-progressively measurable continuous process $(X^{x}_t)_{t\ge 0}$,
 such that
$X^{x}_0=x$,
and
for  $\bP^{x}$-a.s.~that
\bb\l{eq:X^*}
dX^{h,x}_t= b(X^{h,x}_t,{h}(X^{h,x}_t)) \,dt+ \sigma(X^{h,x}_t,{h}(X^{h,x}_t))\,dW_t
\q
\textnormal{ $\fa t\in [0,\tau^{x}]$,}
\ee
and $\int_0^{\tau^{h,x}} \left(| b(X^{h,x}_s,{h}(X^{h,x}_s))|+|\sigma(X^{h,x}_s,{h}(X^{h,x}_s))|^2\right)\,ds<\infty$,
where 
$\tau^{h,x}\coloneqq\inf\{t\ge 0\mid X^{h,x}_t\not \in \cO\}$.
A feedback control $h$ is said to be optimal if 
we have 
for all $x\in \ol{\cO}$ that 
$v(x)=J(x,h)$,
%=J(x,({h}(X^{x}_t))_{t\ge 0})
where 
\bb\l{eq:J_h}
J(x,h)\coloneqq \ex^{\bP^x}\bigg[\int_0^{\tau^{h,x}} \Gamma^{h,x}_s f(X^{h,x}_s,h(X^{h,x}_s))\,ds+ g(X^{h,x}_{\tau^{h,x}})\Gamma^{h,x}_{\tau^{h,x}}\bigg],
\ee
and $\Gamma^{h,x}_t=\exp\big(-\int_0^t c(X^{h,x}_s,h(X^{h,x}_s))\,ds\big)
$
for all $t\in [0,\tau^{h,x}]$.
\end{Definition}

\begin{Theorem}\l{thm:verification}
Suppose (H.\ref{assum:D}) holds. 
%and 
% $\sigma(x,\ba_k)\in \bS^n_0$ 
%for all $(x,\ba_k)\in\R^n\t \bA$.
Let  
$v:\ol{\cO}\to \R$ be the value function defined as in \eqref{eq:value},
 $u\in C(\ol{\cO}) \cap C^2(\cO)$ be the solution to the Dirichlet problem \eqref{eq:hjb},
 and $\a^u:\ol{\cO}\to \bA$ be a Borel measurable function 
 satisfying \eqref{eq:ctrl}.
Then we have 
$u(x)=v(x)$ 
for all $x\in \ol{\cO}$, 
and 
$\a^u$ is an optimal feedback control of \eqref{eq:value}.
\end{Theorem}

\begin{Remark}\l{rmk:weak_formulation}
As shown in Theorem \ref{thm:verification},
by considering a weak formulation of the stochastic control problem \eqref{eq:value} with   reference probability systems varying in  $\Pi_{\textnormal{ref}}$, we can   rigorously demonstrate that a measurable function $\a^u$ satisfying \eqref{eq:ctrl} is indeed  an optimal feedback control strategy.

One can also consider stochastic exit time problems under a strong formulation, for which we first fix a reference probability system $\pi=(\Om, \cF,  \{\cF_t\}_{t\ge 0}, \bP ,W)$, and the agent only maximizes the reward functional over all admissible control processes in $\cA_\pi$.
%It has been shown in \cite{lions1982, buckdahn2016} that, if the discount rate $c$ is sufficiently large, then the value function of the stochastic control problem (under the strong formulation) is a viscosity solution to \eqref{eq:hjb}.
%We can further deduce from the strong comparison result in \cite[Section V.1]{ishii1990} that the Dirichlet problem \eqref{eq:hjb} admits at most one viscosity solution under (H.\ref{assum:D}). 
%\color{blue}
It has been shown in \cite[Theorem 2.1]{chaumont2004} that, 
if we assume 
(H.\ref{assum:D}) and $c>0$ on $\bar{\cO}\t \bA$, 
then \eqref{eq:hjb} satisfies the strong comparison principle
i.e., a comparison result for semicontinuous viscosity solutions.
In particular, 
(H4) in  \cite{chaumont2004} is satisfied since 
$\p\cO\in C^{2,\theta}$ enjoys the exterior ball condition, 
and  (H5) in \cite{chaumont2004} is satisfied with $\Gamma_{\textrm{out}}=\p\cO$
due to the uniform ellipticity condition \eqref{eq:elliptic}.
The strong comparison principle further enables us to show that 
the value function of the stochastic control problem (under the strong formulation) is 
the unique continuous viscosity solution to \eqref{eq:hjb};
see \cite[Theorem 3.1]{barles1998}.
\color{black}
Since  the classical solution $u$ is a viscosity solution of \eqref{eq:hjb},
we see it is the value function  of the stochastic control problem (under the strong formulation),
and the strategy $\a^u$ defined in \eqref{eq:ctrl} will lead to the optimal reward.
Hence, we can still view the function $\a^u$ as an optimal feedback control.
\end{Remark}

We reiterate that, due to the fact that $\argmax$ is a set-valued mapping,
the feedback control strategy \eqref{eq:ctrl} 
in general is non-unique, discontinuous, 
and  sensitive to the perturbation of the coefficients.
For instance, let $K=2$, and  consider the set 
$G=\{x\in \cO\mid (\cL_1-\cL_2)u(x)+(f_1-f_2)(x)=0\}$
at whose boundary the optimal control $\a^u$ in \eqref{eq:ctrl} could have a jump discontinuity. 
Except for the trivial case where $\a^u$ is a constant on $\cO$, 
one can easily deduce from the connectedness of $\cO$, the fact that $u\in C^2(\cO)$, and the continuity of the coefficients that the set $G$ is non-empty. 
Since the boundary of the level set $G$ can have poor regularity,
we see the feedback control $\a^u$ in general is merely Borel measurable, 
which introduces a substantial difficulty to follow  the optimal control  in practice.
Moreover, the discontinuity of  $\a^u$ 
also implies that 
 a small perturbation of the coefficients could lead to a significant difference of  $\a^u$ in the sup-norm, especially near the boundary of the set $G$.
%\color{blue}
It is well-known
(see e.g.~\cite[Section 6.4.2]{bertsekas1996} and \cite[Figure 4]{guo2020})
 that such an instability of feedback controls would 
result in a numerical instability of the learning process, 
i.e., the approximate policies generated by an iterative learning algorithm may
change subsequently  from one iteration to the next, 
and eventually oscillate among several far-from-optimal policies. 
\color{black}

\section{Relaxation of stochastic exit time problem }\l{sec:relax}
In this section, we propose a relaxation of the stochastic exit time problem \eqref{eq:value},
which extends the  ideas used in \cite{wang2019} to control problems with multi-dimensional controlled dynamics and general exploration reward functions.
As we shall see shortly, the relaxed control problem has a  H\"{o}lder continuous feedback control strategy, 
and enjoys better stability  with respect to  perturbation of the coefficients.

The following technical lemma is essential for the formulation
of relaxed control problems
with multi-dimensional dynamics,
whose proof is included in Appendix \ref{appendix:lemmas}. 
%In the sequel, we shall denote by 
%$\Delta_K=\{\lambda\in \R^K\mid \sum_{k=1}^K \lambda_k=1, \lambda_k\ge 0,\, \fa k\in \cK\}$.

\begin{Lemma}\l{lemma:sqrt_A}
Suppose $(H.\ref{assum:D})$ holds. Then there exist  unique functions  $\tilde{b}:\R^n\t \Delta_K\to \R^n$ and  $\tilde{\sigma}:\R^n\t \Delta_K\to \bS^{n}_{>}$
such that it holds for all $x\in \R^n$, $\lambda\in \Delta_K$ that 
$$
\tilde{b}(x,\lambda)=\sum_{k=1}^K{b}(x,\ba_k)\lambda_k,
\q 
\tilde{\sigma}(x,\lambda)\tilde{\sigma}(x,\lambda)^T=
\sum_{k=1}^K{\sigma}(x,\ba_k){\sigma}(x,\ba_k)^T\lambda_k.
$$
Moreover, 
%there exists a constant  $\Lambda'>0$
%%depending only on $\nu, \Lambda$ and $n$, 
%such that 
it holds for all $x\in \R^n$, $\lambda\in \Delta_K$ that 
$ \tilde{\sigma}(x,\lambda)\ge \sqrt{\nu} I_n$
and 
$\sum_{i,j} |\tilde{\sigma}^{ij}(\cdot,\lambda)|_{0,1}+\sum_{i} |\tilde{b}^{i}(\cdot,\lambda)|_{0,1}<\infty$.
\end{Lemma}

We now proceed to  introduce the relaxation of the exit time problem \eqref{eq:value}.
Roughly speaking, 
instead of seeking the optimal feedback action, which maps the current state to \textit{a specific action} in the space $\bA$, 
we seek the optimal feedback control distribution, which is a deterministic mapping from the current state to  \textit{a probability measure} over the space $\bA$, i.e., $\lambda^*:\cO\to \cP(\bA)$.
Once such a mapping is determined, at each given state,  the agent will execute the control  by sampling a control action based on the distribution $\lambda^*(x)$. 
We refer the reader  to \cite{wang2019} for a more  detailed derivation of the following regularized control problem \eqref{eq:ctrl_relax} in a one-dimensional setting.
Note that  the fact that $\bA$ has  cardinality $K<\infty$  enables us to identify the space of 
probability measures over  $\bA$
as the probability simplex $\Delta _K$.

More precisely, 
let 
$\pi=(\Om, \cF,  \{\cF_t\}_{t\ge 0}, \bP ,W)\in\Pi_{\textnormal{ref}}$ be a given reference probability system,
and $\cM_\pi$ be the set of $\{\cF_t\}_{t\ge 0}$-progressively measurable processes $\lambda$ taking values in the  set  $\Delta_K$.
Suppose that (H.\ref{assum:D}) holds, 
for any given initial
state $x \in \R^n$, and control $\lambda \in  \cM_\pi$, we consider the controlled diffusion process $X^{\lam,x}$
satisfying the following SDE: $X_0^{\lam,x}=x$ and 
\bb\l{eq:sde_relax}
dX^{\lam,x}_t= \tilde{b}(X^{\lam,x}_t,\lam_t) \,dt+ \tilde{\sigma}(X^{\lam,x}_t,\lam_t)\,dW_t, \q t \in ( 0,\infty),
\ee
where $\tilde{b}:\R^n\t \Delta_K\to \R^n$ and $\tilde{\sigma}:\R^n\t \Delta_K\to \bS^n_>$ are the   functions defined in Lemma \ref{lemma:sqrt_A}.
We further introduce
the first exit time of $X^{\lam,x}$ from the domain $\cO$ defined as  $\tau^{\lam,x}\coloneqq\inf\{t\ge 0\mid X^{\lam,x}_t\not \in \cO\}$,
and the controlled discount factor   $\Gamma^{\lam,x}_t\coloneqq \exp\left(-\int_0^t \sum_{k=1}^K c(X^{\lam,x}_s,\ba_k)\lam^k_s\,ds\right)$
for all $t\in [0,\tau^{\lam,x}]$.

Now let  $\rho:\R^K\to \R\cup\{\infty\}$ be a  given exploration reward function 
satisfying $\rho<\infty$ on $\Delta_K$
(precise conditions will be specified in (H.\ref{assum:rho})).
For any given relaxation parameter $\eps>0$, 
 we consider the following value function:
for each $x\in \ol{\cO}$,
\bb\l{eq:value_relax}
v^\eps(x)=\sup_{\pi\in \Pi_{\textnormal{ref}}}\sup_{\lam\in \cM_\pi}\ex^{\bP}\bigg[\int_0^{\tau^{\lam,x}} \Gamma^{\lam,x}_s 
\left(\sum_{k=1}^K  f(X^{\lam,x}_s,\ba_k)\lam^k_s-\eps\rho(\lambda_s)\right)\,ds+ g(X^{\lam,x}_{\tau^{\lam,x}})\Gamma^{\lam,x}_{\tau^{\lam,x}}\bigg].
\ee

Note that the exploration reward function $\rho$ plays a crucial role in the above relaxed control regularization.
If we set the exploration reward function $\rho\equiv 0$ or the relaxation parameter $\eps=0$, then one can  show that 
Dirac measures supported on the optimal  strategies of the original control problem \eqref{eq:ctrl} 
(see $\a^u$ defined as in \eqref{eq:ctrl}) are optimal control distributions of 
the relaxed control problem \eqref{eq:value_relax}, 
and the value function $v$ in \eqref{eq:value} 
will be equal to  the value function $v^\eps$ in \eqref{eq:value_relax}
(see Theorems \ref{thm:value_converge} and \ref{thm:ctrl_converge}).
Hence, to achieve the stability of the optimal control strategy for the relaxed control problem \eqref{eq:value_relax}, 
we shall impose the following condition on   the reward function $\rho$:

\begin{Assumption}\l{assum:rho}
There exists a  convex function $H\in C^2(\R^K)$ and a constant $c_0>0$, depending on $K$, 
such that for all $x,y\in\R^K$, we have
$H(x)-c_0\le \max_{k\in \cK}x_k\le H(x)$
%$|\nabla H(x)-\nabla H(y)|\le L|x-y|$,
and 
$\rho(y)=\sup_{z\in \R^K}\big(z^Ty-H(z)\big)$.
\end{Assumption}

We remark that (H.\ref{assum:rho}) is satisfied by most commonly used reward functions, 
including  Shannon's differential entropy proposed in \cite{ziebart2008,haarnoja2017, nachum2017,geist2019,wang2019}.
We refer the reader to the discussion at the end of this section 
for a detailed comparison of different reward functions.

Given a function $H:\R^K\to \R$, 
we define for each   $\eps\ge 0$ the function $H_\eps:\R^K\to \R$ such that 
for all $x=(x_1,\ldots, x_K)^T\in \R^K$,
\bb\l{eq:H_eps}
H_\eps(x)=
\begin{cases}
\eps H(\eps^{-1}x), & \eps>0,\\
\max\{x_1,\ldots, x_K\}, &\eps=0.
\end{cases}
\ee
Note that $(H_\eps)_{\eps\ge 0}$ are convex functions  if $H$ is a convex function.
The next lemma follows directly from (H.\ref{assum:rho}) and standard arguments in convex analysis, 
whose proof will be given  in Appendix \ref{appendix:lemmas} for completeness.

\begin{Lemma}\label{lemma:rho_convex}
Suppose (H.\ref{assum:rho}) holds, and let $(H_\eps)_{\eps\ge 0}$ be defined as in \eqref{eq:H_eps}.
Then we have that
\bn[(1)]
\item\l{item:rho} 
the function $\rho:\R^K\to \R\cup \{\infty\}$ is convex on $\R^K$, continuous relative to  $\Delta_K$, and 
satisfies that
$\rho(y)\in [-c_0,0]$ for all $y\in \Delta_K$ and $\rho(y)=\infty$ for all $y\in  (\Delta_K)^c$,
\item \l{item:H_eps}
it holds for all $x\in \R^K$ and $\eps> 0$ that 
$H_\eps(x)- \eps c_0\le H_0(x)\le H_\eps (x)$,
$H_\eps(x)=\max_{y\in \Delta_K}\big(y^Tx-\eps\rho(y)\big)$,
and 
$(\nabla H_\eps)(x)=\argmax_{y\in \Delta_K}\big(y^Tx-\eps\rho(y)\big)$.
Consequently, we have  for all $x,y\in \R^K$ and $\eps>0$ that
$|H_\eps(x)-H_\eps(y)|\le |x-y|$.
\en
\end{Lemma}

We proceed to study the corresponding HJB equation of the relaxed control problem \eqref{eq:value_relax}, 
which plays a crucial role in our subsequent analysis.
For each $\lambda=(\lambda^1,\ldots,\lambda^K)^T\in \Delta_K$, 
let $f^\lambda:\cO\to \R$ 
be the function satisfying
for all $x\in \cO$ that 
$f^\lambda(x)=\sum_{k=1}^Kf(x,\ba_k)\lambda^k=\lambda^T\bf(x)$ (with $\bf$ defined as  in \eqref{eq:hjb}),
and 
$\cL^\lambda$ be the elliptic operator satisfying for all $\phi\in C^2(\cO)$ and $x\in \cO$ that
\begin{align}\l{eq:L_lam}
\begin{split}
\cL^\lambda\phi(x)
&=
\frac{1}{2}\big(\tilde{\sigma}(x,\lambda)\tilde{\sigma}^T(x,\lambda)\big)^{ij}\p_{ij}\phi(x)+\tilde{b}^i(x,\lambda)\p_i\phi(x)-
\bigg(\sum_{k=1}^Kc(x,\ba_k)\lambda^k\bigg)\phi(x)\\
&=
\sum_{k=1}^K
\bigg(
\frac{1}{2}\big({\sigma}(x,\ba_k){\sigma}^T(x,\ba_k)\big)^{ij}\p_{ij}\phi(x)+{b}^i(x,\ba_k)\p_i\phi(x)-
c (x,\ba_k)\phi(x)
\bigg)\lambda^k
=\lambda^T\bL\phi(x),
\end{split}
\end{align}
where 
we have used the definition of the elliptic operators
 $\bL=(\cL_k)_{k\in\cK}$ (cf.~\eqref{eq:L_k}),
 and the definition of the functions  $\tilde{b}$ and $\tilde{\sigma}$ (cf.~Lemma \ref{lemma:sqrt_A}).

Since
 the diffusion coefficient of SDE \eqref{eq:sde_relax}  is  non-degenerate (see Lemma \ref{lemma:sqrt_A})
 and all coefficients of the relaxed control problem \eqref{eq:value_relax} are continuous on $\ol{\cO}\t \Delta_K$,
 a formal application of 
the dynamic programming principle 
(see e.g.~\cite{fleming2006, buckdahn2016} and references within)
enables us to  associate the relaxed  control problem \eqref{eq:value_relax} with the following HJB equation:
$$
\max_{\lambda\in \Delta_K}
\big[
\cL^\lambda u^\eps+f^\lambda-\eps\rho(\lambda)
\big]=0
\q\textnormal{in $\cO$}, 
\q\q
u^\eps=g \q \textnormal{on $\p\cO$.}
$$
Moreover, \eqref{eq:L_lam} and Lemma \ref{lemma:rho_convex}\eqref{item:H_eps} imply that 
the above Dirichlet problem is equivalent to 
\bb\l{eq:hjb_relax}
F_\eps[u^\eps]
\coloneqq 
H_\eps
(
\bL u^\eps+\bf)=0
\q\textnormal{in $\cO$}, 
\q\q
u^\eps=g \q \textnormal{on $\p\cO$,}
\ee
where the function $H_\eps$ is defined as in \eqref{eq:H_eps},
and $\bL$, $\bf$ are defined as those in \eqref{eq:hjb}.

In order to rigorously justify the connection between \eqref{eq:value_relax} and \eqref{eq:hjb_relax},
we establish the well-posedness of classical solutions to \eqref{eq:hjb_relax} in Theorem \ref{thm:wp_relax},
and then prove a verification result in Theorem \ref{thm:verification_relax}.

We need the following proposition, which gives an \textit{a priori} estimate of classical solutions to 
\eqref{eq:hjb_relax}.
We postpone the proof to Appendix \ref{appendix:lemmas}, 
which adapts the  technique in \cite[Theorem 7.5 on p.~127]{chen1998} to HJB equations with compact control sets, and reduces the problem to 
an \textit{a priori} estimate for
 HJB equations involving only  principal  terms.

\begin{Proposition}\l{prop:Holder_bdd}
Suppose  (H.\ref{assum:D}) and (H.\ref{assum:rho}) hold,
and let $M=\sup_{i,j,k}|\sigma^{ij}_k|_{0;\ol{\cO}}$.
Then 
there exists a constant $\b_0=\b_0(n,\nu,M)\in (0,1)$,
%{independent of  $K$, $\eps$, $c_0$},
 such that 
 it holds  for all $\b\in (0,\min(\b_0,\theta)]$ that,
if 
$u^\eps\in C^{2,\b}(\ol{\cO})$ is a solution to 
the Dirichlet problem \eqref{eq:hjb_relax}
with   parameter $\eps>0$,
then $u^\eps$ satisfies the  estimate that 
 $ |u^\eps|_{2,\b}\le C(|g|_{2,\b}+\eps c_0+1)$,
 where 
the constant $C$ depends only on  $n$, $\nu$, $\Lambda$, $\b$ and  $\cO$.
\end{Proposition}

\begin{Theorem}\l{thm:wp_relax}
Suppose (H.\ref{assum:D}) and (H.\ref{assum:rho}) hold,
 let $\eps>0$ and  $M=\sup_{i,j,k}|\sigma^{ij}_k|_{0;\ol{\cO}}$.
Then
the Dirichlet problem \eqref{eq:hjb_relax} admits a unique solution $u^\eps\in C(\ol{\cO}) \cap C^2(\cO)$. 
Moreover, there exists a constant $\b_0=\b_0(n,\nu,M)\in (0,1)$
%{independent of  $K$, $\eps$, $c_0$},
and a unique  function $\lambda^{u^\eps}:\ol{\cO}\to \Delta_K$
 such that 
$u^\eps\in  C^{2,\min(\b_0,\theta)}(\ol{\cO})$,
$\lambda^{u^\eps}\in C^{\min(\b_0,\theta)}(\ol{\cO},\R^K)$
and
\bb\l{eq:ctrl_relax}
\lambda^{u^\eps}(x)=
\argmax_{\lambda\in \Delta_K}
\big(
\cL^\lambda u^\eps(x)+f^\lambda(x)-\eps\rho(\lambda)
\big)
=(\nabla H_\eps)(\bL u^\eps(x)+\bf(x))
\q \fa x\in \ol{\cO}.
\ee
%and it holds  for all $\b\in (0,\min(\b_0,\theta)]$ that
%$u^\eps\in  C^{2,\min(\b_0,\theta)}(\ol{\cO})$ and 
%$\lambda^{u^\eps}\in C^{\min(\b_0,\theta)}(\ol{\cO},\R^K)$.
In particular, one can take the same constant $\b_0$ as  in Proposition  \ref{prop:Holder_bdd}.
\end{Theorem}
\begin{proof}
One can deduce by 
similar arguments as those for Theorem \ref{thm:wp}
and the classical maximum principle that 
\eqref{eq:hjb_relax} admits a unique classical solution in $C(\ol{\cO}) \cap C^2(\cO)$.
%while the existence and global regularity of classical solutions shall be established constructively via policy iteration; see \todo{Theorem}.
Moreover, %let $\b_0$ be the same constant as  in Proposition  \ref{prop:Holder_bdd},
 by using the \textit{a priori} bound of classical solutions in Proposition \ref{prop:Holder_bdd},
we can 
%directly extend the arguments in 
%\cite[Theorem 7.7.5 on p.~127]
% and 
establish the existence and  regularity of the classical solution $u^\eps$ to \eqref{eq:hjb_relax} based on the method of continuity; see \cite[Theorem 5.1 on p.~116]{chen1998}.

Now let $u^\eps\in  C^{2,\b}(\ol{\cO})$ be the solution to \eqref{eq:hjb_relax}
with some $\b\in (0,\theta]$.
The continuity of $\cL^\lambda, f^\lambda$ and $\rho$ on $\Delta_K$,
and Lemma \ref{lemma:rho_convex}\eqref{item:H_eps}
%\ref{eq:ctrl_relax}\eqref{item:H_eps}
ensure that
the function  $\lambda^{u^\eps}$ is well-defined on $\ol{\cO}$, 
and has the expression 
$\lambda^{u^\eps}=(\nabla H_\eps)(\bL u^\eps+\bf)$.
Note that, it holds for any given $\phi_1,\phi_2\in C^{\b}(\ol{\cO})$ that
$\phi_1\phi_2\in C^{\b}(\ol{\cO})$. 
%satisfying 
%$|\phi_1\phi_2|_\b\le |\phi_1|_\b|\phi_2|_\b$. 
Hence the H\"older continuity of the coefficients (see (H.\ref{assum:D})) 
implies that $\bL u^\eps+\bf\in C^\b(\ol{\cO},\R^K)$.
We can then easily deduce from the local Lipschitz continuity of $\nabla H_\eps:\R^K\to \R^K$ that 
$\lambda^{u^\eps}\in C^\b(\ol{\cO},\R^K)$.
\end{proof}

 The next theorem shows that the function \eqref{eq:ctrl_relax} is an optimal feedback  control of \eqref{eq:value_relax}, which is defined similarly to  Definition \ref{def:feedback}.
The proof of this statement is similar to that of Theorem \ref{thm:verification} and hence omitted.  

%Note that
%Lemma \ref{lemma:sqrt_A} ensures the diffusion coefficient $\tilde{\sigma}$ of \eqref{eq:sde_relax} is always in $\bS^n_>$,
%which subsequently implies
%the verification theorem always holds for the relaxed control problem, even though
%the original diffusion coefficient $\sigma\in \R^{n\t n}$ is not symmetric positive semidefinite (cf.~Theorem \ref{thm:verification}).

\begin{Theorem}\l{thm:verification_relax}
Suppose (H.\ref{assum:D}) and (H.\ref{assum:rho}) hold.
Let $\eps>0$, 
$v^\eps:\ol{\cO}\to \R$ be the value function defined as in \eqref{eq:value_relax},
$u^\eps\in C(\ol{\cO}) \cap C^2(\cO)$ be the solution to the Dirichlet problem \eqref{eq:hjb_relax},
and $\lambda^{u^\eps}:\ol{\cO}\to \Delta_K$ be the function defined as in \eqref{eq:ctrl_relax}.
Then %we have 
$u^\eps(x)=v^\eps(x)$ 
for all $x\in \ol{\cO}$, 
and 
$\lambda^{u^\eps}$ is an optimal feedback control of \eqref{eq:value_relax}.
\end{Theorem}

\begin{Remark}
Theorem \ref{thm:wp_relax} shows that 
the feedback control  $\lambda^{u^\eps}$ is uniquely defined and H\"older continuous.
This improved regularity makes it  easier to implement the relaxed control 
$\lambda^{u^\eps}$ in practice, compared to the original  (merely measurable)
feedback control $\a^u$ (cf.~Theorem \ref{thm:wp}).

\end{Remark}

We end this section with a remark about 
possible choices of reward functions. 
Generally speaking,
we shall choose a reward function $\rho$ whose generating function $H$ and its gradient $\nabla H$ can be efficiently evaluated,
such that one can design an efficient algorithm to solve the relaxed control problem \eqref{eq:value_relax}
(see e.g.~\cite{ziebart2008,haarnoja2017, nachum2017,geist2019,ito2019}).
A common choice of   reward functions  in the literature is the following entropy-type reward function (see e.g.~\cite{kort1972,peng1998,peng1999,wang2019}):
$$
\rho_{\textnormal{en}}(y)=
\begin{cases}
\sum_{k=1}^K y_k\ln (y_k), & y\in \Delta_K,\\
\infty, & y\in  (\Delta_K)^c,
\end{cases}
$$
whose generating function is $H_{\textnormal{en}}(x)=\ln \sum_{k=1}^K \exp({x_k})$, $x\in \R^K$.
One can  show that $H_{\textnormal{en}}\in C^\infty(\R^K)\cap C^{2,1}(\R^K)$, and 
it satisfies 
(H.\ref{assum:rho}) with $c_0=\ln K$ 
% $H_{\textnormal{en}}(x)-\ln K\le H_0(x)\le H_{\textnormal{en}}(x)$, $x\in \R^K$ 
(see e.g.~\cite{peng1999}).

The advantage of the  entropy reward function is that 
both $H_{\textnormal{en}}$ and $\nabla H_{\textnormal{en}}$ are given in  closed form,
and they can be naturally extended to continuous action spaces $\bA$ (see e.g.~\cite{wang2019}). 
However,
it is important to notice that 
the evaluation of $H_{\textnormal{en}}$ and $\nabla H_{\textnormal{en}}$
involves exponentials.  
%\color{blue}
Hence, 
when the relaxation parameter $\eps$ is small, 
a naive implementation of iterative algorithms  for solving \eqref{eq:hjb_relax},
which in general involves  evaluating 
the value and inverse of
$H_{\textnormal{en}}$
and $\nabla H_{\textnormal{en}}$
% and 
%  $(\nabla H_{\textnormal{en}})^{-1}(z/{\eps})$
  at a large argument
$z=(\bL u^\eps(x)+\bf(x))/\eps\in \R^K$ 
with $ x\in \cO$, 
may lead to 
unreliable results
due to unstable floating-point arithmetic;
see 
\cite[Example 4.2]{birbil2005}
and 
\cite{blanchard2019}
for more details.
\color{black}
Moreover, since $\nabla H_{\textnormal{en}}(x)\in (0,1)^K$ for all $x\in \R^K$,
the  optimal relaxed control of \eqref{eq:value_relax} 
may converge to 
 the optimal control  of \eqref{eq:value} with a very slow rate as the relaxation parameter $\eps$ tends to zero.

Alternatively, 
by virtue of the fact that 
only the   generating function $H$  and its gradient are involved in the HJB equation \eqref{eq:hjb_relax} and the feedback control \eqref{eq:ctrl_relax},
we can also obtain a  reward function $\rho$ by 
directly constructing  a $K$-dimensional  function $H$ 
based on a recursive application of smoothing functions for the two-dimensional max function.
 For instance, we can start with the following two-dimensional smoothing functions (see e.g.~\cite{chen1995,zang1980}): for $x=(x_1,x_2)^T\in \R^2$,
\begin{align}
{H}_{\textnormal{chks}}(x)&=\f{\sqrt{(x_1-x_2)^2+1}+x_1+x_2}{2},
\l{eq:H_chks}\\
{H}_{\textnormal{zang}}(x)&=
\begin{cases}
x_1,  & x_2-x_1<-1/2,\\
-\f{1}{2}(x_1-x_2)^4+\f{3}{4}(x_1-x_2)^2+\f{x_1+x_2}{2}+\f{3}{32}, &|x_1-x_2|\le 1/2,\\
x_2, &x_2-x_1>1/2. 
\end{cases}
\l{eq:H_zang}
\end{align}
Then, for any given $K\ge 3$,  by using  the fact that 
$\max_{k\in \cK}x_k=\max(\max_{i\in \cK_1}x_i,\max_{j\in \cK_2}x_j)$,
with $\cK_1=\{1,\ldots, K_0\}$, $\cK_2=\{  K_0+1,\ldots, K\}$ and $K_0=\lfloor (K+1)/2\rfloor$,
we can express the $K$-dimensional max function as a nested application of
the two-dimensional max function and one-dimensional identity function.
Hence, by replacing the two-dimensional max function with 
the two-dimensional smoothing function \eqref{eq:H_chks} (resp.~\eqref{eq:H_zang})
in the recursive expression,
we can obtain the $K$-dimensional smoothing function $H_{\textnormal{chks}}\in C^{\infty}(\R^K)\cap C^{2,1}(\R^K)$ 
(resp.~$H_{\textnormal{zang}}\in C^{2,1}(\R^K)$).
It has been shown in \cite[Lemma 3.3]{birbil2005} that for any given $K\ge 2$, 
both functions $H_{\textnormal{chks}}$ and $H_{\textnormal{zang}}$ satisfy (H.\ref{assum:rho}) with 
$c_0=(\log_2(K-1)+1)/2$ for ${H}_{\textnormal{chks}}$,  
and 
$c_0=3(\log_2(K-1)+1)/32$ for ${H}_{\textnormal{zang}}$.

Note that, the evaluation of $H_{\textnormal{chks}}$, $H_{\textnormal{zang}}$ and their gradients 
only
involves   square-roots and multiplications, 
hence
they are numerically more stable than  
the entropy-type smoothing $H_{\textnormal{en}}$ (see \cite{birbil2005}).
More importantly, 
since $H_{\textnormal{zang}}$ only modifies the function $H_0$ locally near the non-differentiable points,
we can  determine the optimal control of \eqref{eq:value}  precisely 
from the optimal control of \eqref{eq:value_relax}
without sending the relaxation parameter $\eps$ to zero (see Theorem \ref{thm:ctrl_converge} and Remark \ref{rmk:exact_regularization} for details).

Figure \ref{fig:comparision_reward} compares the functions $H_{\textnormal{en}},H_{\textnormal{zang}}:\R^3\to \R$ and 
the reward functions generated by them. 
One can clearly see from Figure \ref{fig:comparision_reward}  (left) that $H_{\textnormal{en}}$ 
substantially modifies the pointwise maximum function $H_0$ everywhere, while $H_{\textnormal{zang}}$  only performs a  modification of  $H_0$ locally near the kinks.
For both functions, the difference from $H_0$ peaks around the 
 the points where
$ \argmax_{k\in \cK} x_k $ is not a singleton. 
 Such points correspond to the regions where 
 the agent of the  control problem \eqref{eq:value}
 cannot make a clear decision based on the current model,
since  two or more different actions would result in a very similar reward.

 %\color{blue}

 Figure \ref{fig:comparision_reward} (right) depicts the reward functions $\rho_{\textnormal{en}}(y_1,y_2,y_3)$
 and $\rho_{\textnormal{zang}}(y_1,y_2,y_3)$
 with $y_3=1-y_1-y_2$, for all $(y_1,y_2)\in\cC\coloneqq \{(y_1,y_2)\in \R^2\mid 0\le y_1,y_2\le 1, y_1+y_2\le 1\}$. 
 The point $(1/3,1/3,1/3)$ corresponds to the pure exploration strategy,
 i.e., the uniform distribution on the action space $\bA=\{\ba_1,\ba_2,\ba_3\}$,
 while the vertices of $\cC$ 
 corresponds to the pure exploitation strategy,
 i.e., the Dirac measures supported on some $\ba_i\in \bA$.
   Both functions achieve their minimum around the point 
 $(1/3,1/3,1/3)$, which indicates that the exploration reward functions encourage 
 the controller of the relaxed control problem to explore further,  
especially when it is difficult to choose a unique optimal  action based on the current model.

Note that, 
by comparing the values of the reward functions 
near the point  $(1/3,1/3,1/3)$
and 
near the vertices of $\cC$, 
we see that
$\rho_{\textnormal{en}}$ in general gives more   rewards for exploration than $\rho_{\textnormal{zang}}$.
%even at the points where one strategy significantly outperforms the remaining ones (i.e., ).
Consequently, to recover the value function and optima control of  \eqref{eq:value},
we have to take a smaller 
 relaxation parameter for \eqref{eq:value} with $\rho_{\textnormal{en}}$
than that for \eqref{eq:value} with $\rho_{\textnormal{zang}}$,
which could cause a numerical instability issue due to 
the exponentials in  $H_{\textnormal{en}}$ and $\nabla H_{\textnormal{en}}$  (see e.g.~\cite{birbil2005}).

\color{black}

\begin{figure}[!htb]
    \centering
    \includegraphics[width=0.41\columnwidth,height=5.7cm]{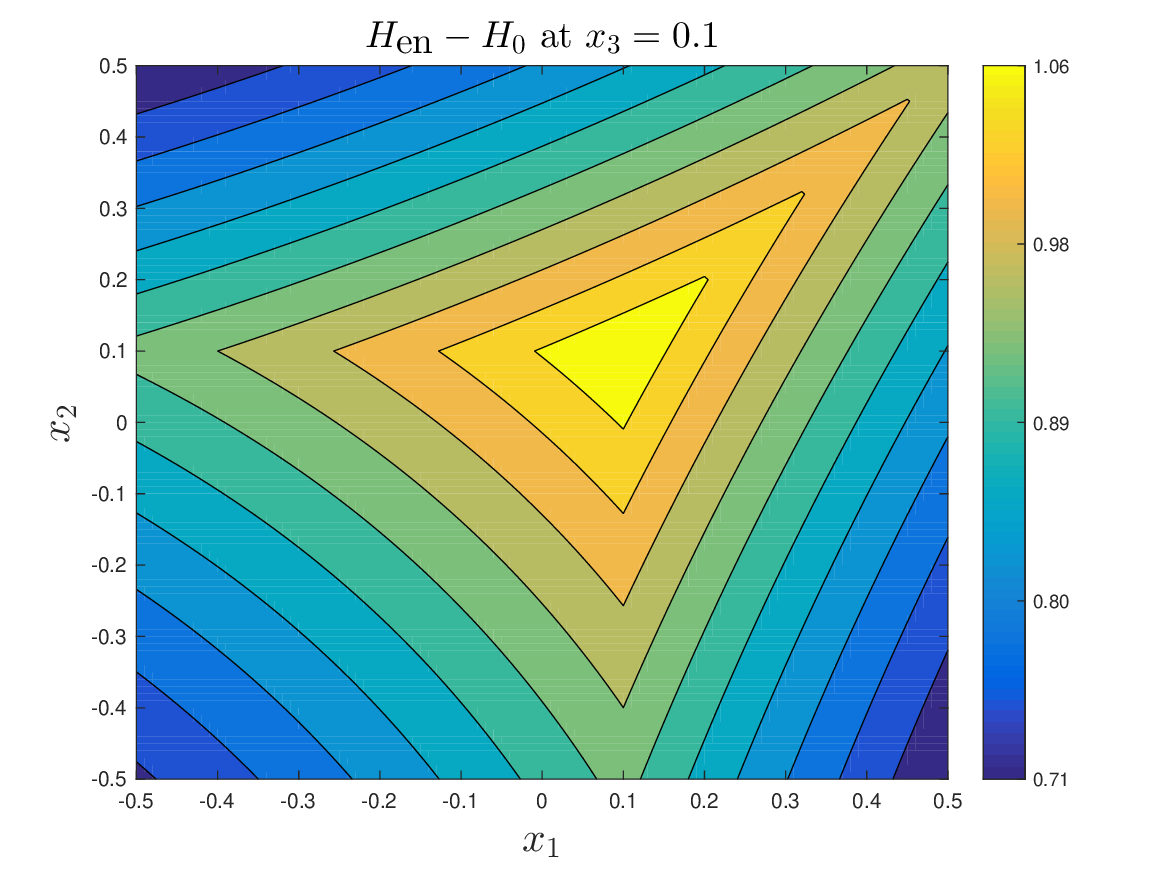}\!
    \includegraphics[width=0.41\columnwidth,height=5.7cm]{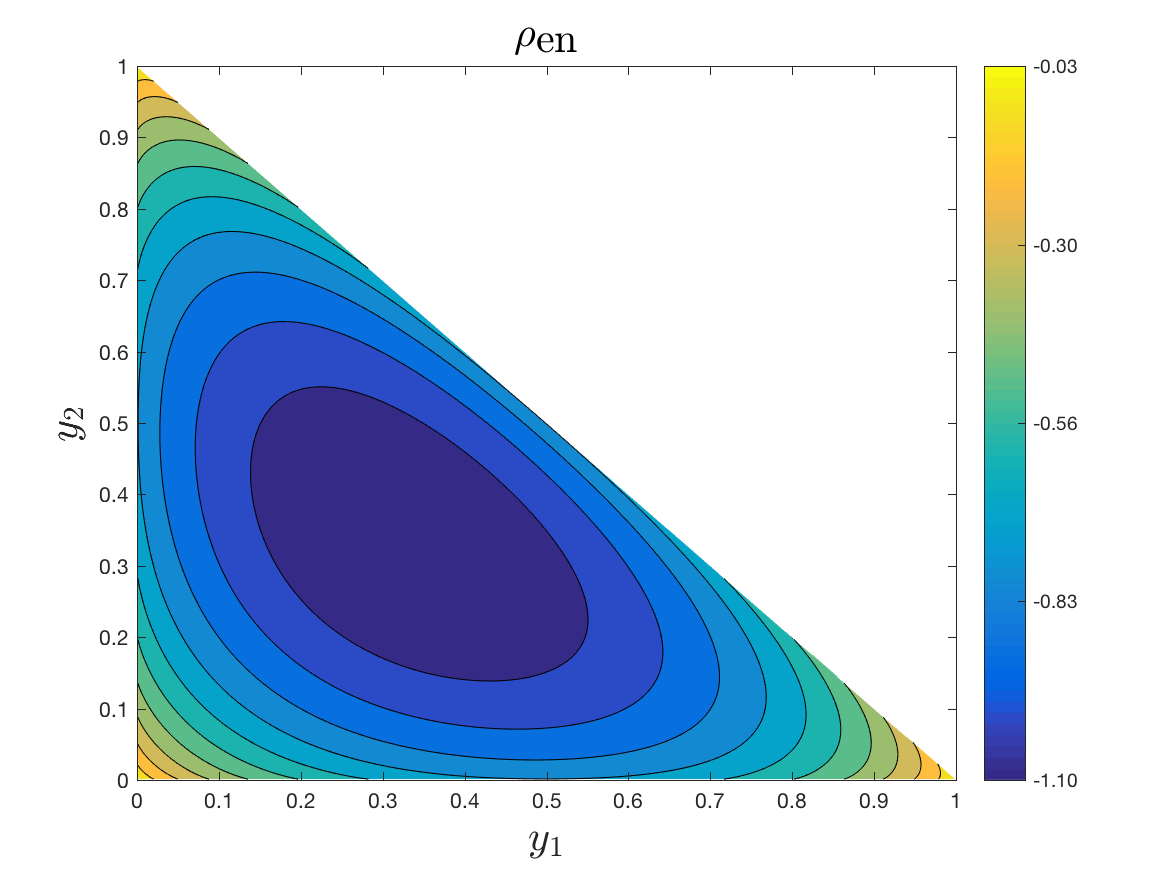}\\
        \includegraphics[width=0.41\columnwidth,height=5.7cm]{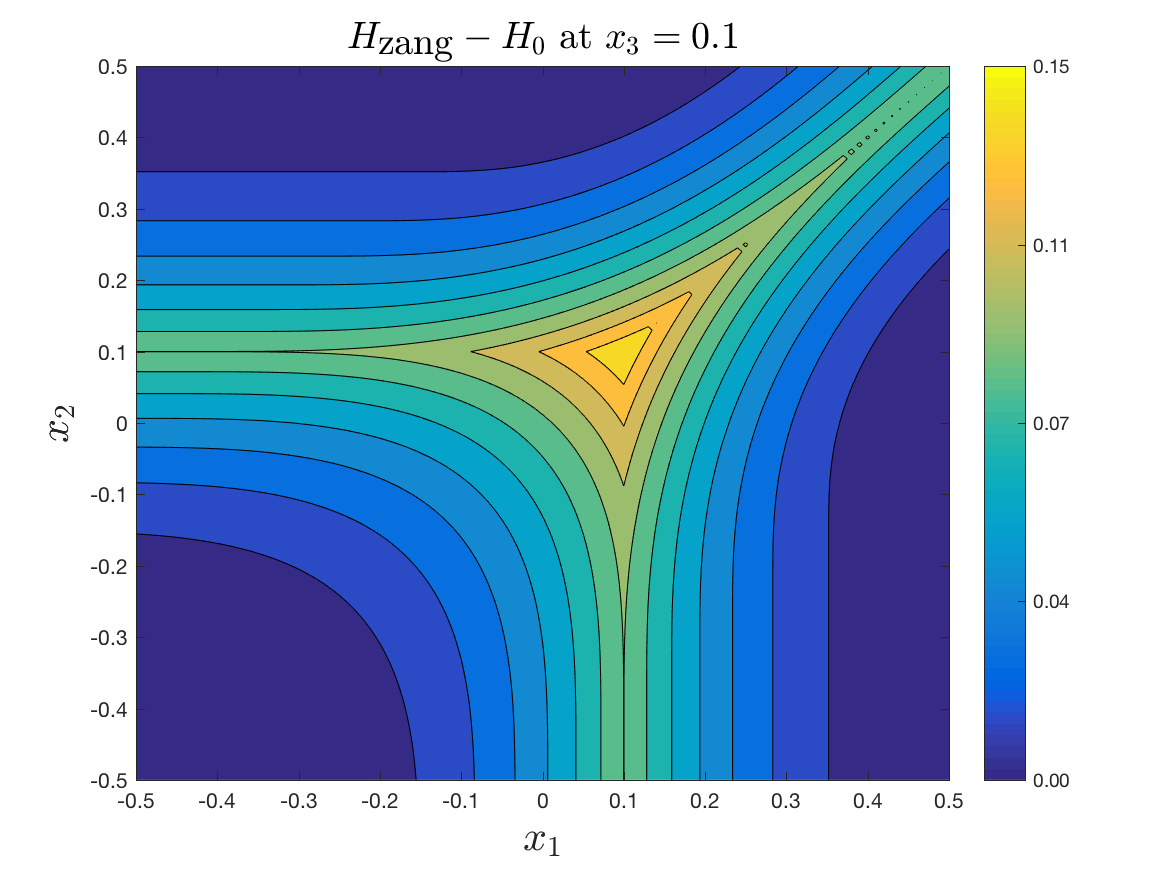}\!
    \includegraphics[width=0.41\columnwidth,height=5.7cm]{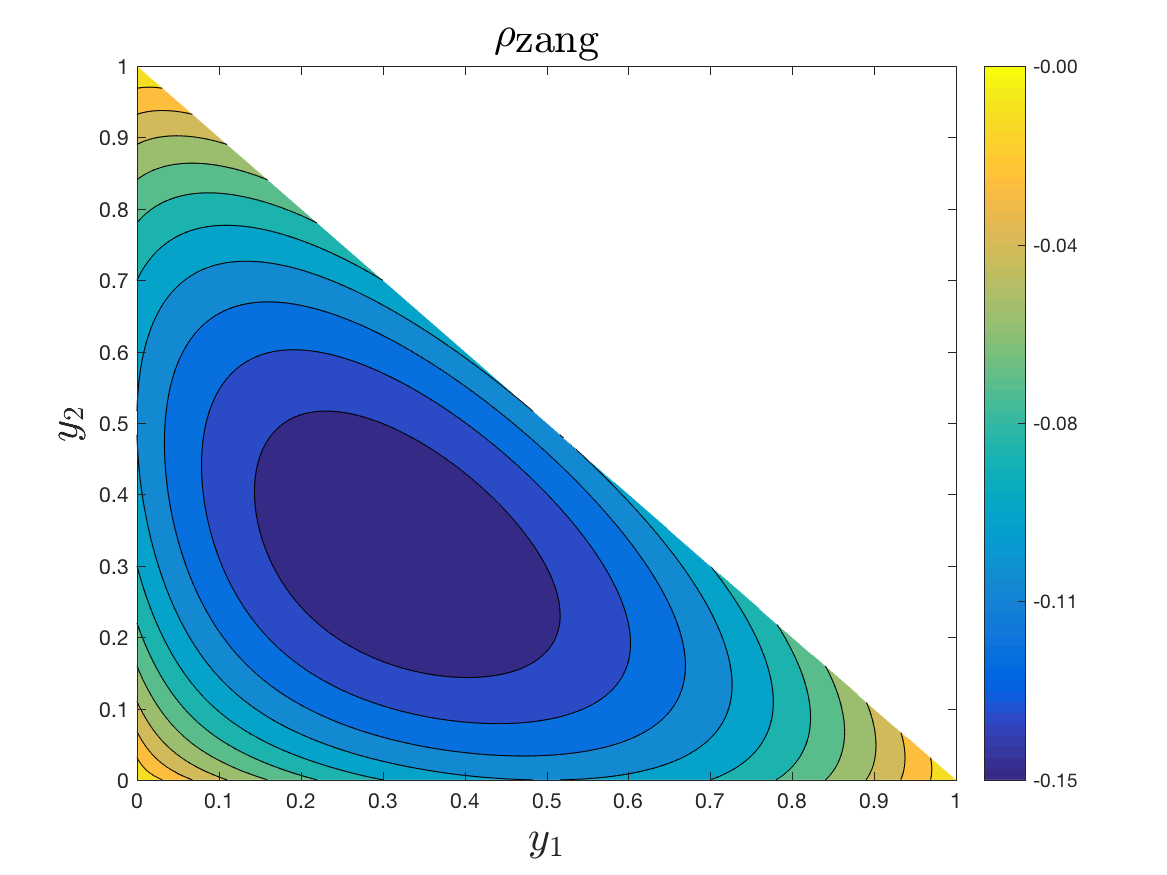}\\
    \caption{Comparison of $H_{\textnormal{en}}$ and $H_{\textnormal{zang}}$ and their corresponding reward functions for $K=3$.}
    \label{fig:comparision_reward}
 \end{figure}

\section{Lipschitz stability of optimal feedback relaxed control}\l{sec:lipschitz}
In this section, we 
shall fix a relaxation parameter $\eps>0$
and  study the robustness of the feedback control strategy \eqref{eq:ctrl_relax} for a relaxed control problem
associated with a perturbed model.
In particular,
we shall show that the control strategy \eqref{eq:ctrl_relax} 
 admits a (locally) Lipschitz continuous dependence 
on the perturbation of the coefficients, if the reward function is  generated by a   function $H$ with locally Lipschitz continuous Hessian.

We start by presenting two technical results, which are essential for our subsequent analysis.
The first one
is due to Nugari \cite{nugari1993},
which establishes the regularity of Nemytskij operators
in H\"{o}lder spaces.

\begin{Lemma}\l{lemma:nemytskij}
Let $n,K\in \N$, $\a\in (0,1]$, $\cO\subset \R^n$ be an  open bounded set, 
 $\phi:\R^K\to \R$ be a continuously differentiable function, 
and $\Phi:u\in C^{\a}(\ol{\cO},\R^K)\mapsto \Phi [u]\in C^{\a}(\ol{\cO})$
be the Nemytskij operator satisfying for all $u=(u_1,\ldots u_K)$ that
$\Phi[u](x)=\phi(u(x))$,  $x\in \ol{\cO}$.
 Then $\Phi$ is well-defined,  continuous and bounded.
 Moreover, 
if we further suppose $\nabla \phi$ is locally Lipschitz  continuous 
(resp.~$\phi$ is twice continuously differentiable),
then $\Phi$ is locally Lipschitz continuous 
(resp.~continuously differentiable with the Fr\'{e}chet derivative $\Phi'[u]=(\nabla\phi)^T(u)$ for all $u\in 
C^{\a}(\ol{\cO},\R^K)$).
\end{Lemma}

%\color{blue}
\begin{Remark}\l{rmk:Holder_Sobolev}
Lemma \ref{lemma:nemytskij} enables us  to view 
the fully nonlinear HJB operator $F_\eps$ in \eqref{eq:hjb_relax} 
and the value-to-action map $u^\eps\mapsto \lambda^{u^\eps}$ defined  in \eqref{eq:ctrl_relax}
as differentiable maps between suitable H\"{o}lder spaces,
which is  essential for the sensitivity analysis on the value functions and feedback relaxed controls in Section \ref{sec:sensitivity}.

Note that in general it is not possible to 
perform the same first-order sensitivity analysis 
by interpreting the HJB operator $F_\eps$ 
as  a map between the Sobolev space $W^{2,p}(\cO)$ 
and the Lebesgue space $L^q(\cO)$. 
In fact, since the operator $F_\eps:W^{2,p}(\cO)\to L^q(\cO)$ in general is only differentiable with $p>q$ 
(see  \cite[Theorem 13]{smears2014}), 
we see the derivative of $F_\eps$, 
which is a second-order linear elliptic  operator,
is not bijective between $W^{2,p}(\cO)$ and $L^q(\cO)$.
Consequently, we cannot apply 
the implicit function theorem to derive the sensitivity equation for the value function \eqref{eq:value_relax} as in Theorem \ref{thm:sensitivity}. 

If the operator $F_\eps$ is only semilinear, i.e., the diffusion coefficient of \eqref{eq:sde} is  uncontrolled, then one can show that
$F_\eps$ is differentiable between $W^{2,p}(\cO)$ and  $L^p(\cO)$ for $1<p<\infty$,
 and its derivative is a bijection  between the same spaces (see \cite{ito2019} for the case with $p=2$). 
 In this case,  
 we can extend Theorem \ref{thm:sensitivity}
 and study $L^p$-perturbation of the coefficients in \eqref{eq:value_relax}. 
\end{Remark}
\color{black}

Now we   proceed to introduce a relaxed control problem with a set of perturbed coefficients satisfying the following conditions:

\begin{Assumption}\l{assum:D_p}
Let $\nu>0$, $\theta\in (0,1]$ be the constants in (H.\ref{assum:D}),
and  $\Lambda'>0$ be a constant.
The  functions $\hat{b}:\R^n\t\bA\to \R^n$, $\hat{\sigma}:\R^n\t \bA\to \R^{n\t n}$, $\hat{c}:\ol{\cO}\t \bA\to [0,\infty)$,  $\hat{f}:\ol{\cO}\t \bA\to \R$, and $\hat{g}:\ol{\cO}\to \R$
satisfy that  $\hat{g}\in  C^{2,\theta}(\ol{\cO})$,
$\hat{\sigma}(x,\ba_k) \hat{\sigma}^T(x,\ba_k)\ge \nu I_n$ for all $(x,\ba_k)\in \R^n\t \bA$, 
and for all $\ba_k\in \bA$ that
$$
\sum_{i,j} |\hat{\sigma}^{ij}(\cdot,\ba_k)|_{0,1;\R^n}+\sum_{i} |\hat{b}^{i}(\cdot,\ba_k)|_{0,1;\R^n}+
|\hat{c}(\cdot,\ba_k)|_{\theta;\ol{\cO}}+|\hat{f}(\cdot,\ba_k)|_{\theta;\ol{\cO}}\le \Lambda'.
$$

\end{Assumption}

Let $\eps>0$ be a  fixed relaxation parameter. 
We shall consider a perturbed   control problem \eqref{eq:value} with the coefficients 
$(\hat{b},\hat{\sigma},\hat{c},\hat{f},\hat{g})$,
and its relaxation (see \eqref{eq:value_relax}) with parameter $\eps$, 
whose value function is denoted as $\hat{v}^\eps$.
Then, 
by using Lemma \ref{lemma:rho_convex},
Theorems  \ref{thm:wp_relax} and \ref{thm:verification_relax},
 one can verify that, under
(H.\ref{assum:rho}) and 
 (H.\ref{assum:D_p}), the value function $\hat{v}^\eps$ is the classical  solution 
$\hat{u}^\eps\in C(\ol{\cO}) \cap C^2(\cO)$ of the following Dirichlet problem:
\bb\l{eq:hjb_relax_p}
\max_{\lambda\in \Delta_K}
\big(\lambda^T(\hat{\bL} \hat{u}^\eps+\hat{\bf})
-\eps\rho(\lambda)
\big)
=
H_\eps
\big(
\hat{\bL} \hat{u}^\eps+\hat{\bf}
\big)=0
\q\textnormal{in $\cO$}, 
\q\q
\hat{u}^\eps=\hat{g} \q \textnormal{on $\p\cO$,}
\ee
where the function $H_\eps$ is defined as in \eqref{eq:H_eps},
$\hat{\bf}:\ol{\cO}\to \R^K$ is the function satisfying
$\hat{\bf}(x)=(\hat{f}(x,\ba_k))_{k\in \cK}$
for all $x\in \ol{\cO}$, 
and $\hat{\bL}=(\hat{\cL}_k)_{k\in \cK}$ is a family of elliptic operators satisfying for all $k\in \cK$, $\phi\in C^2({\cO})$, $x\in \cO$ that 
$$
\hat{\cL}_k \phi(x)\coloneqq \hat{a}^{ij}_k(x)\p_{ij}\phi(x)+\hat{b}^i_k(x)\p_i\phi(x)-\hat{c}_k(x) \phi(x), \q \textnormal{with $\hat{a}_k=\tfrac{1}{2}\hat{\sigma}_k\hat{\sigma}^T_k$.}
$$
Moreover, 
we can deduce from
 \eqref{eq:ctrl_relax} that,
the optimal feedback  control of the perturbed relaxed control problem is given by
\bb\l{eq:ctrl_relax_p}
\hat{\lambda}^{\hat{u}^\eps}(x)
=\argmax_{\lambda\in \Delta_K}
\big(\lambda^T(\hat{\bL} \hat{u}^\eps(x)+\hat{\bf}(x))
-\eps\rho(\lambda)
\big)
=(\nabla H_\eps)(\hat{\bL} \hat{u}^\eps(x)+\hat{\bf}(x))
\q \fa x\in \ol{\cO}.
\ee
Note that Theorem \ref{thm:wp_relax} shows that the classical solution $\hat{u}^\eps$ of 
\eqref{eq:hjb_relax_p} is  in $C^{2,\b}(\ol{\cO})$ for some $\b>0$, so the above function $\hat{\lambda}^{\hat{u}^\eps}$ is well-defined on $\p\cO$.

The following result 
shows the (local) Lipschitz dependence of 
 $\hat{u}^\eps-u^\eps$
 and $\hat{\lambda}^{\hat{u}^\eps}-{\lambda}^{{u}^\eps}$ 
on
    perturbation of the coefficients,
 which demonstrates the robustness of the relaxed control problem.
For notational simplicity, given the functions 
$({b},{\sigma},{c},{f},{g})$ and $(\hat{b},\hat{\sigma},\hat{c},\hat{f},\hat{g})$
satisfying (H.\ref{assum:D}) and (H.\ref{assum:D_p}) respectively,
we shall introduce for each $\b\in (0,\theta]$  the following measurement of perturbations:
\bb\l{eq:perturbation}
\cE_{\textnormal{per},\b}\coloneqq\sup_{i,j,k} (|a^{ij}_k-\hat{a}^{ij}_k|_{\b;\ol{\cO}}+|b^i_k-\hat{b}^i_k|_{\b;\ol{\cO}}+|c_k-\hat{c}_k|_{\b;\ol{\cO}}+|f_k-\hat{f}_k|_{\b;\ol{\cO}})+|g-\hat{g}|_{2,\b;\ol{\cO}},
\ee
where $a_k=\sigma_k\sigma_k^T$, 
$\hat{a}_k=\hat{\sigma}_k\hat{\sigma}_k^T$ for each $k\in \cK$.

\begin{Theorem}\l{thm:u_hat-u}
Suppose (H.\ref{assum:D}), (H.\ref{assum:rho}) and (H.\ref{assum:D_p}) hold.
Let $\eps>0$, 
$M=\sup_{i,j,k}\max(|\sigma^{ij}_k|_{0;\ol{\cO}},|\hat{\sigma}^{ij}_k|_{0;\ol{\cO}})$,
$M_g=\max(|g|_{2,\theta},|\hat{g}|_{2,\theta})$,
$u^\eps\in C(\ol{\cO}) \cap C^2(\cO)$ 
(resp.~$\hat{u}^\eps\in C(\ol{\cO}) \cap C^2(\cO)$) be the solution to the Dirichlet problem \eqref{eq:hjb_relax} (resp.~\eqref{eq:hjb_relax_p}),
and
$\lambda^{u^\eps}:\ol{\cO}\to \Delta_K$ 
(reps.~$\hat{\lambda}^{\hat{u}^\eps}:\ol{\cO}\to \Delta_K$)
be the function defined as in \eqref{eq:ctrl_relax}
(resp.~\eqref{eq:ctrl_relax_p}).
Then 
there exists $\b_0=\b_0(n,\nu,M)\in (0,1)$,
 such that 
 it holds
for all $\b\in (0,\min(\b_0,\theta)]$ that,
\bb\l{eq:u_hat-u}
| \hat{u}^\eps-u^\eps|_{2,\b}\le 
C\cE_{\textnormal{per},\b}
\ee
with the constant  $\cE_{\textnormal{per},\b}$ defined as in \eqref{eq:perturbation},
and  
a constant 
$C=C(\eps, n, K, \nu, \Lambda, \Lambda', \b, c_0, M_g, \cO)$.

%$C$ is a constant 
%depending only  on $\eps$, $n$, $K$, $\nu$, $\Lambda$, $\Lambda'$, $\b$, $c_0$, $M_g$ and  $\cO$.

If we further suppose the function $H:\R^K\to \R$ in (H.\ref{assum:rho})
has a locally Lipschitz continuous Hessian, 
then it also holds that 
$|\hat{\lambda}^{\hat{u}^\eps}-\lambda^{u^\eps}|_{\b}\le C\cE_{\textnormal{per},\b}$.

\end{Theorem}

\begin{proof}
Throughout this proof, 
we shall denote by $C$ a generic constant, 
which  
depends only  on $\eps$, $n$, $K$, $\nu$, $\Lambda$,  $\Lambda'$, $\b$, $c_0$, $M_g$ and  $\cO$,
and may take a different value at each occurrence.

The \textit{a priori} estimate in Proposition \ref{prop:Holder_bdd} shows that 
there exists a constant $\b_0=\b_0(n,\nu,M)\in (0,1)$,
 such that 
we have  for all $\b\in (0,\min(\b_0,\theta)]$ the estimates
$ |u^\eps|_{2,\b},|\hat{u}^\eps|_{2,\b}\le C$.
Moreover, we have by  the fundamental theorem of calculus that 
\begin{align}\l{eq:u-u_hat}
\begin{split}
0&=
H_\eps\big(
{\bL} {u}^\eps+{\bf}
\big)
-H_\eps\big(
\hat{\bL} \hat{u}^\eps+\hat{\bf}
\big)={\eta}^T\big({\bL} {u}^\eps+{\bf}-\hat{\bL} \hat{u}^\eps-\hat{\bf}\big)
%\\
%&
={\eta}^T\big({\bL} ({u}^\eps-\hat{u}^\eps)+({\bL} -\hat{\bL}) \hat{u}^\eps+{\bf}-\hat{\bf}\big)
%\q \textnormal{in $\cO$,}
\end{split}
\end{align}
in $\cO$, where ${\eta}:\ol{\cO}\to \Delta_K$ is the function defined as 
${\eta} \coloneqq \int_0^1 (\nabla H_{\eps})\big(s({\bL} {u}^\eps+{\bf})+(1-s)(\hat{\bL} \hat{u}^\eps+\hat{\bf})\big)\,ds$.

Now let  $\b\in (0,\min(\b_0,\theta)]$ be a fixed constant. 
The fact that  $\nabla H_\eps\in C^1(\R^K,\Delta_K)$ (see  (H.\ref{assum:rho})), the H\"{o}lder continuity of  coefficients 
(see (H.\ref{assum:D}) and (H.\ref{assum:D_p})), and the   \textit{a priori} estimates of $ |u^\eps|_{2,\b}$ and $|\hat{u}^\eps|_{2,\b}$
yield  the estimate that
  $|\eta|_{\b}\le C$ (see Lemma \ref{lemma:nemytskij}).
Then, by setting $w={u}^\eps-\hat{u}^\eps\in C^{2,\b}(\ol{\cO})$, we can deduce from \eqref{eq:u-u_hat} that 
$w$ is the classical solution to the following Dirichlet problem:
\begin{align*}
{\eta}^T{\bL} w=-{\eta}^T\big(({\bL} -\hat{\bL}) \hat{u}^\eps+{\bf}-\hat{\bf}\big)
\q\textnormal{in $\cO$}, 
\qq
w=g-\hat{g} \q \textnormal{on $\p\cO$.}
\end{align*}
Hence the fact that $\eta\in C^\b(\ol{\cO},\Delta_K)$
and  the  global Schauder estimate in \cite[Theorem 6.6]{gilbarg1985} 
lead us to the estimate that
\begin{align*}
|w|_{2,\b}
&\le C(|w|_{0}+|g-\hat{g}|_{2,\b}+|{\eta}^T\big(({\bL} -\hat{\bL}) \hat{u}^\eps+{\bf}-\hat{\bf}\big)|_\b),
\end{align*}
which, together with the maximum principle (see \cite[Theorem 3.7]{gilbarg1985}) and 
the    \textit{a priori} estimate of  $|\hat{u}^\eps|_{2,\b}$, enables us to conclude that:
\begin{align*}
|{u}^\eps-\hat{u}^\eps|_{2,\b}=|w|_{2,\b} &\le C(|g-\hat{g}|_{2,\b}+|{\eta}^T\big(({\bL} -\hat{\bL}) \hat{u}^\eps+{\bf}-\hat{\bf}\big)|_\b)
\le C\cE_{\textnormal{per},\b},
%&\le C( |a_k-\hat{a}_k|_\b+|b_k-\hat{b}_k|_\b+|c_k-\hat{c}_k|_\b+|f_k-\hat{f}_k|_\b+|g-\hat{g}|_{2,\b}),
\end{align*}
with the constant $\cE_{\textnormal{per},\b}$ defined as in \eqref{eq:perturbation}.

Now we show the stability of feedback controls. Note that  \eqref{eq:u_hat-u} implies that 
\begin{align*}
&|\big(
{\bL} {u}^\eps+{\bf}
\big)
-\big(
\hat{\bL} \hat{u}^\eps+\hat{\bf}
\big)|_\b
=
|{\bL} ({u}^\eps-\hat{u}^\eps)+({\bL} -\hat{\bL}) \hat{u}^\eps+{\bf}-\hat{\bf}|_\b
\le C\cE_{\textnormal{per},\b}.
\end{align*}
The additional assumption that 
$H:\R^K\to \R$ in (H.\ref{assum:rho})
has a locally Lipschitz continuous Hessian implies that 
$\nabla H_\eps$ is differentiable with locally Lipschitz continuous derivatives,
which along with Lemma \ref{lemma:nemytskij} shows that 
the Nemytskij operator $\nabla H_\eps:C^\b(\ol{\cO},\R^K)\to C^\b(\ol{\cO},\R^K) $ is locally Lipschitz continuous.
Hence there exists a constant $C$, such that for all
perturbed coefficients 
$(\hat{b},\hat{\sigma},\hat{c},\hat{f},\hat{g})$ 
satisfying  (H.\ref{assum:D_p}), we have %the estimate
\begin{align*}
&|\hat{\lambda}^{\hat{u}^\eps}-\lambda^{u^\eps}|_{\b}
=|(\nabla H_\eps)(\hat{\bL} \hat{u}^\eps(x)+\hat{\bf}(x))-
(\nabla H_\eps)({\bL} {u}^\eps(x)+{\bf}(x))|\\
&\le 
C|\big(
\hat{\bL} \hat{u}^\eps+\hat{\bf}
\big)
-\big(
{\bL} {u}^\eps+{\bf}
\big)|_\b
\le C\cE_{\textnormal{per},\b},
\end{align*}
which finishes the desired (local) Lipschitz estimate.
\end{proof}

\begin{Remark}\l{rmk:unstable_eps0}
The assumption that 
 $H:\R^K\to \R$ in (H.\ref{assum:rho})
has a locally Lipschitz continuous Hessian is satisfied by 
most commonly used functions, including  
$H_{\textnormal{en}}$, $H_{\textnormal{chks}}$ and $H_{\textnormal{zang}}$
given in Section \ref{sec:relax}.
In general, if $H$ is merely twice continuously differentiable as in (H.\ref{assum:rho}), we can follow a similar argument and establish that 
the H\"{o}lder norm of the difference between two relaxed control strategies is continuously dependent on 
the H\"{o}lder norms of
the perturbations in the coefficients.

%\color{blue}
Note that
the Lipschitz stability
result \eqref{eq:u_hat-u} 
in general does not hold for 
the original control problem \eqref{eq:value} 
(or equivalently, $\eps=0$ in \eqref{eq:value_relax}).
In fact, for any given $\b\in (0,1)$,  \cite[Theorem 2]{drabek1975} shows that 
the Nemytskij operator $\bf\in (C^{\b}(\ol{\cO}))^K\mapsto H_0(\bf)\in C^{\b}(\ol{\cO})$
is not continuous, which implies that there exists $(\bf_m)_{m\in \N\cup\{\infty\}}\subset (C^{\b}(\ol{\cO}))^K$ such that
$\lim_{m\to\infty}
|\bf_m-\bf_\infty|_{\b}=0$ and $|H_0(\bf_m)-H_0(\bf_\infty)|_{\b}\ge 1$ for all $m\in\N$.
Now for each $m\in \N\cup\{\infty\}$, we consider  
the following simple HJB equation  \eqref{eq:hjb}:
$\Delta u_m+H_0(\bf_m)=0$ {in $\cO$} and  $u_m=0$ {on $\p\cO$}.
Hence we have $|\Delta (u_m-u_\infty)|_\b=|H_0(\bf_m)-H_0(\bf_\infty)|_\b\ge 1$ for all $m\in \N$,
which implies that the  $C^{2,\b}$-norm of the
value function   \eqref{eq:value} does not 
depend continuously on the $C^{\b}$-perturbation of the model parameters.
 See Theorem \ref{thm:deltau_eps}
  for a precise quantification of $\eps$-dependence in \eqref{eq:u_hat-u}.
\color{black}
\end{Remark}

The remaining part of this section is devoted to an 
important application of Theorem \ref{thm:u_hat-u},
where we shall examine 
the performance of the control strategy $\lambda^{u^\eps}$,
computed based on 
the relaxed control problem with the original coefficients $({b},{\sigma},{c},{f},{g})$  (see \eqref{eq:ctrl_relax}),
on a new relaxed control problem with  perturbed coefficients satisfying (H.\ref{assum:D_p}).

We first observe that,
if there exists a classical solution $\ul{u}^\eps\in C(\ol{\cO}) \cap C^2(\cO)$ to the following  problem: 
\bb\l{eq:hjb_p_sub}
(\lambda^{u^\eps})^T(\hat{\bL} \ul{u}^\eps+\hat{\bf})-\eps\rho(\lambda^{u^\eps})=0
\q\textnormal{in $\cO$}, 
\qq
\ul{u}^\eps=\hat{g} \q \textnormal{on $\p\cO$,}
\ee
with  $\hat{\bL}$ and $\hat{\bf}$  defined as in \eqref{eq:hjb_relax_p},
then by 
 using It\^{o}'s formula, one can easily show that 
the  reward function $\ul{v}^\eps$, resulting by implementing the H\"{o}lder continous feedback control $\lam^{u^\eps}$ to  
the relaxed control problem with the coefficients $(\hat{b},\hat{\sigma},\hat{c},\hat{f},\hat{g})$,
coincides with the function $\ul{u}^\eps$ (see e.g.~Theorems \ref{thm:verification} and \ref{thm:verification_relax}).
On the other hand, we have seen that 
the (optimal) value function $\hat{v}^\eps$ of the perturbed relaxed control problem 
is  the classical solution $\hat{u}^{\eps}$ to \eqref{eq:hjb_relax_p}.
Hence it suffices to compare the classical solutions to \eqref{eq:hjb_p_sub} and \eqref{eq:hjb_relax_p}.

The following proposition shows that \eqref{eq:hjb_p_sub} indeed admits an unique classical solution.
\begin{Proposition}
Suppose (H.\ref{assum:D}), (H.\ref{assum:rho}) and (H.\ref{assum:D_p}) hold.
Let $\eps>0$,  $M=\sup_{i,j,k}|\sigma^{ij}_k|_{0;\ol{\cO}}$,
$\lambda^{u^\eps}:\ol{\cO}\to \Delta_K$ 
be the function defined as in \eqref{eq:ctrl_relax},
and 
$\b_0=\b_0(n,\nu,M)\in (0,1)$ be the  constant  in Proposition \ref{prop:Holder_bdd}.
Then
the Dirichlet problem \eqref{eq:hjb_p_sub} admits a unique solution 
%$ \ul{u}^\eps\in C(\ol{\cO}) \cap C^2(\cO)$. 
%Moreover, 
%%there exists a constant $\b_0=\b_0(n,\nu,M)\in (0,1)$,
%%such that 
%we have for all $\b\in (0,\min(\b_0,\theta)]$ that
$\ul{u}^\eps\in  C^{2,\min(\b_0,\theta)}(\ol{\cO})$.

\end{Proposition}
\begin{proof}
Let us denote $\bar{\b}_0=\min(\b_0,\theta)$ throughout this proof.
The uniqueness of classical solutions to \eqref{eq:hjb_p_sub} follows directly from the 
classical maximum principle (see \cite[Theorem 3.7]{gilbarg1985}).
Hence we shall focus on establishing the existence and regularity of solutions to \eqref{eq:hjb_p_sub}.
Note that,
%it has been proved in Theorem \ref{thm:wp_relax} that, 
%there exists a constant $\b_0=\b_0(n,\nu,M)\in (0,1)$, such that
%such that for all $\b\in (0,\min(\b_0,\theta)]$,
%we have $\lambda^{u^\eps}\in  C^{\b}(\ol{\cO},\Delta_K)$ 
%and $u^\eps\in C^{2,\b}(\ol{\cO})$,
%where $u^\eps$ is the classical solution to \eqref{eq:hjb_relax}.
Theorem \ref{thm:wp_relax}  shows that
 $\lambda^{u^\eps}\in  C^{\bar{\b}_0}(\ol{\cO},\Delta_K)$ 
and $u^\eps\in C^{2,\bar{\b}_0}(\ol{\cO})$,
where $u^\eps$ is the classical solution to \eqref{eq:hjb_relax}.

We now study the  function
$\rho(\nabla H_\eps):x\in \R^K\mapsto \rho(\nabla  H_\eps(x)) \in\R$.
Note that,  (H.\ref{assum:rho}) and \eqref{eq:H_eps}  imply that 
$H_\eps:\R^K\to \R$ is convex and differentiable.   Moreover, Lemma \ref{lemma:rho_convex}\eqref{item:H_eps} shows that the  convex conjugate of $H_\eps$, denoted by $(H_\eps)^*$, is given by
$(H_\eps)^*(y)=\sup_{x\in \R^K}(x^Ty-H_\eps(x))=\eps\rho(y)$, for all $y\in \R^K$.
Hence, we can deduce from \cite[Theorem 23.5]{rockafellar1970} (by setting $f=H_\eps$ in the statement) that
\bb\l{eq:rho_nablaH}
(\eps \rho)((\nabla H_\eps)(x))=(H_\eps)^*((\nabla H_\eps)(x))=x^T (\nabla H_\eps)(x)-H_\eps(x), \q x\in \R^K,
\ee
which implies that $(\eps\rho)(\nabla H_\eps)\in C^1(\R^K)$.
Then, by using the representation $\lambda^{u^\eps}=(\nabla H_\eps)(\bL u^\eps+\bf)$, we can conclude that $\eps\rho(\lambda^{u^\eps})=(\eps\rho)\big((\nabla H_\eps)(\bL u^\eps+\bf)\big)\in C^{\bar{\b}_0}(\ol{\cO})$.
% for all $\b\in (0,\min(\b_0,\theta)]$.

Therefore, %for any given  $\b\in (0,\min(\b_0,\theta)]$, 
we can deduce from the fact that all coefficients of \eqref{eq:hjb_p_sub} are in $C^{\bar{\b}_0}(\ol{\cO})$,
$\sum_{k=1}^K \lambda^{u^\eps,k}c_k\ge 0$,
and \cite[Theorem 6.14]{gilbarg1985} that 
\eqref{eq:hjb_p_sub} admits a unique solution in $C^{2,\bar{\b}_0}(\ol{\cO})$.
\end{proof}

We are ready to show that, 
the difference between 
this suboptimal reward function $\ul{v}^\eps$
and the (optimal) value function $\hat{v}^\eps$ of the perturbed relaxed control problem 
depends Lipschitz-continuously on 
the magnitude of perturbations in the coefficients.

\begin{Theorem}\l{thm:u_sub-u_op}
Suppose (H.\ref{assum:D}), (H.\ref{assum:rho}) and (H.\ref{assum:D_p}) hold.
Let $\eps>0$, 
$M=\sup_{i,j,k}\max(|\sigma^{ij}_k|_{0;\ol{\cO}},|\hat{\sigma}^{ij}_k|_{0;\ol{\cO}})$,
$M_g=\max(|g|_{2,\theta},|\hat{g}|_{2,\theta})$,
and
$\ul{u}^\eps\in C(\ol{\cO}) \cap C^2(\cO)$ 
(resp.~$\hat{u}^\eps\in C(\ol{\cO}) \cap C^2(\cO)$) be the solution to the Dirichlet problem \eqref{eq:hjb_p_sub} (resp.~\eqref{eq:hjb_relax_p}).
%and
%$\lambda^{u^\eps}:\ol{\cO}\to \Delta_K$ 
%(reps.~$\hat{\lambda}^{\hat{u}^\eps}:\ol{\cO}\to \Delta_K$)
%be the function defined as in \eqref{eq:ctrl_relax}
%(resp.~\eqref{eq:ctrl_relax_p}).
Then 
we have $\hat{u}^\eps\ge \ul{u}^\eps$ on $\ol{\cO}$.

If we further suppose the function $H:\R^K\to \R$ in (H.\ref{assum:rho})
has a locally Lipschitz continuous Hessian,
then there exists $\b_0=\b_0(n,\nu,M)\in (0,1)$,
 such that 
 for all $\b\in (0,\min(\b_0,\theta)]$,
 we have the estimate
 $
| \hat{u}^\eps-\ul{u}^\eps|_{2,\b}
\le 
C\cE_{\textnormal{per},\b}$,
with the constant $\cE_{\textnormal{per},\b}$  defined as in \eqref{eq:perturbation},
and  
a constant $C=C(\eps, n, K, \nu, \Lambda, \Lambda', \b, c_0, M_g, \cO)$.
\end{Theorem}

\begin{proof}
Let $\lambda^{u^\eps}:\ol{\cO}\to \Delta_K$ 
(reps.~$\hat{\lambda}^{\hat{u}^\eps}:\ol{\cO}\to \Delta_K$)
be the function defined as in \eqref{eq:ctrl_relax}
(resp.~\eqref{eq:ctrl_relax_p}),
and $C$ be a generic constant, 
which  
depends only  on $\eps$, $n$, $K$, $\nu$, $\Lambda$, $\Lambda'$, $\b$, $c_0$, $M_g$ and  $\cO$,
and may take a different value at each occurrence.

The fact $\hat{u}^\eps$ solves \eqref{eq:hjb_relax_p} implies that, for all $x\in \cO$,
\begin{align*}
0&=
H_\eps
\big(
\hat{\bL} \hat{u}^\eps(x)+\hat{\bf}(x)
\big)
= 
\max_{\lambda\in \Delta_K}
\big(\lambda^T(\hat{\bL} \hat{u}^\eps(x)+\hat{\bf}(x))
-\eps\rho(\lambda)
\big)\\
&\ge 
(\lambda^{u^\eps}(x))^T(\hat{\bL} \hat{u}^\eps(x)+\hat{\bf}(x))-\eps\rho(\lambda^{u^\eps}(x)),
\end{align*}
which, together with the fact that $\hat{u}^\eps=\ul{u}^\eps=\hat{g}$
and the classical maximum principle (see \cite[Theorem 3.7]{gilbarg1985}),
shows that $\hat{u}^\eps\ge \ul{u}^\eps$ on $\ol{\cO}$.

We now estimate  $\hat{u}^\eps- \ul{u}^\eps$ by assuming 
the function $H:\R^K\to \R$ in (H.\ref{assum:rho}) has a locally Lipschitz continuous Hessian.
By using the definition of the optimal control $\hat{\lambda}^{\hat{u}^\eps}$, we have
that 
$$
(\hat{\lambda}^{\hat{u}^\eps})^T(\hat{\bL} \hat{u}^\eps+\hat{\bf})-\eps\rho(\hat{\lambda}^{\hat{u}^\eps})=0,
\q\textnormal{in $\cO$}.
$$
By subtracting \eqref{eq:hjb_p_sub} from the above equation, we have
\begin{align*}
0&=\big[(\hat{\lambda}^{\hat{u}^\eps})^T(\hat{\bL} \hat{u}^\eps+\hat{\bf})-\eps\rho(\hat{\lambda}^{\hat{u}^\eps})\big]
-\big[(\lambda^{u^\eps})^T(\hat{\bL} \ul{u}^\eps+\hat{\bf})-\eps\rho(\lambda^{u^\eps})\big]\\
&=(\hat{\lambda}^{\hat{u}^\eps}-\lambda^{u^\eps})^T(\hat{\bL} \hat{u}^\eps+\hat{\bf})
+({\lambda}^{{u}^\eps})^T\hat{\bL} (\hat{u}^\eps-\ul{u}^\eps)
-\big(\eps\rho(\hat{\lambda}^{\hat{u}^\eps})-\eps\rho(\lambda^{u^\eps})\big),
\q\textnormal{in $\cO$}.
\end{align*}
Note that,
the \textit{a priori} estimate in Proposition \ref{prop:Holder_bdd} 
shows that,
under (H.\ref{assum:D}),  (H.\ref{assum:rho}) and  (H.\ref{assum:D_p}),
there exists a constant $\b_0=\b_0(n,\nu,M)\in (0,1)$,
 such that 
we have  for all $\b\in (0,\min(\b_0,\theta)]$ the estimates
$ |u^\eps|_{2,\b},|\hat{u}^\eps|_{2,\b}\le C$,
which, along with the fact that $\nabla H_\eps\in C^1(\R^K)$ and 
Lemma \ref{lemma:nemytskij}, implies the \textit{a priori} bounds
$ |\hat{\lambda}^{\hat{u}^\eps}|_{\b},|{\lambda}^{{u}^\eps}|_{\b}\le C$.
Hence, 
from any given  $\b\in (0,\min(\b_0,\theta)]$,
we can deduce from the  Schauder theory in \cite[Theorem 6.6]{gilbarg1985}
and the maximum principle in \cite[Theorem 3.7]{gilbarg1985}  that
\begin{align}\l{eq:u_hat-u_ul}
\begin{split}
|\hat{u}^\eps-\ul{u}^\eps|_{2,\b}
&\le 
C\big(
|(\hat{\lambda}^{\hat{u}^\eps}-\lambda^{u^\eps})^T(\hat{\bL} \hat{u}^\eps+\hat{\bf})|_{\b}
+|\eps\rho(\hat{\lambda}^{\hat{u}^\eps})-\eps\rho(\lambda^{u^\eps})|_\b
\big)\\
&\le 
C\big(
|\hat{\lambda}^{\hat{u}^\eps}-\lambda^{u^\eps}|_{\b}
+|\eps\rho(\hat{\lambda}^{\hat{u}^\eps})-\eps\rho(\lambda^{u^\eps})|_\b
\big).
\end{split}
\end{align}

By using 
the additional assumption that $H$ has a locally Lipschitz continuous Hessian, 
and the identity \eqref{eq:rho_nablaH}, 
we can  deduce that 
$\rho(\nabla H_\eps):\R^K\to \R$ is continuously differentiable with 
a 
locally Lipschitz continuous gradient, 
from which, we can obtain from  Lemma \ref{lemma:nemytskij} that
for any $\a\in (0,1]$,
the corresponding Nemytskij operator 
$(\eps\rho)(\nabla H_\eps):C^\a(\ol{\cO},\R^K)\to C^\a(\ol{\cO},\R)$ is locally Lipschitz continuous.
Hence, we can obtain from \eqref{eq:u_hat-u_ul}
and the definitions of 
$\lambda^{u^\eps}$
and  $\hat{\lambda}^{\hat{u}^\eps}$ (see \eqref{eq:ctrl_relax} and \eqref{eq:ctrl_relax_p})
 that
\begin{align*}
|\hat{u}^\eps-\ul{u}^\eps|_{2,\b}
&\le
C\big(
|\hat{\lambda}^{\hat{u}^\eps}-\lambda^{u^\eps}|_{\b}
+|
(\eps \rho)((\nabla H_\eps)(\hat{\bL} \hat{u}^\eps+\hat{\bf}))
-(\eps \rho)((\nabla H_\eps)({\bL} {u}^\eps+{\bf}))|_\b
\big)\\
&\le C\big(
|\hat{\lambda}^{\hat{u}^\eps}-\lambda^{u^\eps}|_{\b}
+|
(\hat{\bL} \hat{u}^\eps+\hat{\bf})-({\bL} {u}^\eps+{\bf})|_\b
\big)\\
&\le C\big(
|\hat{\lambda}^{\hat{u}^\eps}-\lambda^{u^\eps}|_{\b}
+|
(\hat{\bL}-\bL) \hat{u}^\eps+{\bL}( \hat{u}^\eps-{u}^\eps)+\hat{\bf}-{\bf}|_\b
\big),
\end{align*}
from which, we can conclude from the 
\textit{a priori} bound of $|\hat{u}^\eps|_{2,\b}$ and Theorem \ref{thm:u_hat-u} 
the desired estimate $|\hat{u}^\eps-\ul{u}^\eps|_{2,\b}\le C\cE_{\textnormal{per},\b}$.
\end{proof}

\section{First-order sensitivity equations for relaxed control problems}\l{sec:sensitivity}
%In Section \ref{sec:lipschitz}, we have shown that,
%for any given threshold $\delta>0$,
% the optimal value function and the optimal feedback policy of 
%the relaxed control problems depend  Lipschitz continuously on the magnitude of perturbations in coefficients
%(with a Lipschitz constant depending on $\delta$), 
%provided that the  perturbed coefficients are within the given  threshold.
In this section, we  proceed to derive a first-order Taylor expansion for the   value function
and the optimal control
of the relaxed control problem \eqref{eq:value_relax} with perturbed coefficients, 
which 
subsequently leads us to 
a first-order approximation of the  optimal strategy
for the perturbed problem 
based on the pre-computed optimal control.
The sensitivity equation further enables us to 
quantify the explicit dependence 
of the Lipschitz stability result \eqref{eq:u_hat-u} 
on the relaxation parameter $\eps$.

The following proposition establishes the Fr\'{e}chet differentiability of 
the fully nonlinear HJB operator with inhomogeneous boundary conditions.
For notational simplicity, 
for any given $\b\in (0,1]$,
and bounded open subset $\cO\subset \R^n$ with $C^{2,\b}$ boundary,
we shall introduce the Banach space $\Theta^\b$ for the coefficients:
\bb\l{eq:X_b}
\Theta^\b= \big(C^{\b}(\ol{\cO},\R^{n\t n})\t C^{\b}(\ol{\cO},\R^{n})\t C^{\b}(\ol{\cO})\t C^{\b}(\ol{\cO})\big)^K\t C^{2,\b}(\ol{\cO})
\ee
equipped with the product norm $|\cdot|_{\Theta^\b}$, 
and denote by $\boldsymbol{\vartheta}=((a_k,b_k,c_k,f_k)_{k\in \cK},g)$ a generic element in $\Theta^\b$.
We also denote by $C^{2,\b}(\p\cO)$ 
the Banach space of $C^{2,\b}$ functions defined on $\p\cO$ (see Remark \ref{rmk:trace}),
and  by $\tau_D:  C^{2,\b}(\ol{\cO}) \to  C^{2,\b}(\p\cO)$ the restriction operator on $\p\cO$.
Furthermore, for any given Banach spaces $X$ and $Y$,
we denote by $B(X,Y)$  the Banach space containing all continuous linear mappings from $X$ into $Y$,
equipped with the operator norm.

 \begin{Proposition}\l{prop:differentiability}
Suppose  (H.\ref{assum:rho}) holds. 
Let $\eps>0$, 
$\b\in (0,1]$,
 $\cO$ be a bounded domain in $\R^n$ with $C^{2,\b}$ boundary,
$H_\eps:\R^K\to \R$ be the function defined as in \eqref{eq:H_eps},
$\Theta^\b$ be the Banach space defined as in \eqref{eq:X_b},
and
$F^\b:\Theta^\b \t C^{2,\b}(\ol{\cO})\to C^{\b}(\ol{\cO})\t C^{2,\b}(\p{\cO})$ be the following HJB operator:
$$
F^\b:(\boldsymbol{\vartheta},u)\in  \Theta^\b\t C^{2,\b}(\ol{\cO})\mapsto 
F^\b[\boldsymbol{\vartheta},u]\coloneqq (H_\eps(\bL^{\boldsymbol{\vartheta}} u+\bf^{\,\boldsymbol{\vartheta}}),\tau_D(u-g^{\boldsymbol{\vartheta}})) \in C^{\b}(\ol{\cO})\t C^{2,\b}(\p{\cO}),
$$
where for any given $\boldsymbol{\vartheta}=((a_k,b_k,c_k,f_k)_{k\in \cK},g)\in \Theta^\b$, 
$\bf^{\,\boldsymbol{\vartheta}}=(f_k)_{k\in \cK}\in C^\b(\ol{\cO})^K$,
$g^{\boldsymbol{\vartheta}}=g$
and 
$\bL^{\boldsymbol{\vartheta}}=(\cL^{\boldsymbol{\vartheta}}_k)_{k\in \cK}$ is the elliptic operators satisfying   
$\cL^{\boldsymbol{\vartheta}}_k \phi= a^{ij}_k\p_{ij}\phi+b^i_k\p_i\phi-c_k \phi$
for all $k\in \cK$, $\phi\in C^2({\cO})$.

Then $F^\b$ is continuously differentiable 
with the  derivative 
$F^\b: \Theta^\b\t C^{2,\b}(\ol{\cO})\to B(\Theta^\b\t C^{2,\b}(\ol{\cO}), C^{\b}(\ol{\cO})\t C^{2,\b}(\p{\cO}))$
satisfying for all  $(\boldsymbol{\vartheta},u)\in \Theta^\b\t C^{2,\b}(\ol{\cO})$,  $\tilde{\boldsymbol{\vartheta}}\in \Theta^\b$ and $v\in C^{2,\b}(\ol{\cO})$ that
$$
(F^\b)'[\boldsymbol{\vartheta},u](\tilde{\boldsymbol{\vartheta}},v)
=\big((\nabla H_\eps)^T(\bL^{\boldsymbol{\vartheta}} u+\bf^{\,\boldsymbol{\vartheta}} )(\bL^{{\boldsymbol{\vartheta}}}v+\bL^{\tilde{\boldsymbol{\vartheta}}}u+\bf^{\,\tilde{\boldsymbol{\vartheta}}}),
\tau_D(v-g^{\tilde{\boldsymbol{\vartheta}}})\big).
$$

\end{Proposition}

\begin{proof}
%Since $\cK$ is a finite set, we shall assume without loss of generality that, 
%the Banach space
%$C^{\b}(\ol{\cO})^K$  is endowed
%with the product norm $|\cdot|_{\b,K}$, i.e., 
%for each $u=(u_k)_{k\in \cK}\in C^{\b}(\ol{\cO})^K$, $|u|_{\b,K}=\sum_{k=1}^K|u_k|_\b$.
We first write the HJB operator as $F^\b=(F_1,F_2)$,
where $F_1:\Theta^\b \t C^{2,\b}(\ol{\cO})\to C^{\b}(\ol{\cO})$
is the composition of the Nemytskij operator
$H_\eps: C^{\b}(\ol{\cO})^K\to  C^{\b}(\ol{\cO})$
 and the  mapping
 $G:(\boldsymbol{\vartheta},u )\in \Theta^\b\t  C^{2,\b}(\ol{\cO}) \mapsto G[\boldsymbol{\vartheta},u]\coloneqq \bL^{\boldsymbol{\vartheta}} u+\bf^{\,\boldsymbol{\vartheta}} \in C^{\b}(\ol{\cO})^K$, %parameterized by the index $k\in \cK$,
and $F_2: (\boldsymbol{\vartheta},u )\in \Theta^\b \t C^{2,\b}(\ol{\cO})\mapsto F_2[\boldsymbol{\vartheta},u]\coloneqq \tau_D(u-g^{\boldsymbol{\vartheta}})\in C^{2,\b}({\p\cO})$
is the linear boundary operator.
 
 Since the function $H_\eps$ is in $C^2(\R^K)$, we can deduce from 
 Lemma \ref{lemma:nemytskij}
 that
 the Nemytskij operator $H_\eps: C^{\b}(\ol{\cO})^K\to  C^{\b}(\ol{\cO})$
 is well-defined and 
 continuously differentiable with the Fr\'{e}chet derivative 
 $(H_\eps)'[u]=(\nabla H_\eps)^T(u)\in B(C^{\b}(\ol{\cO})^K,C^{\b}(\ol{\cO}))$ for all $u\in 
C^{\b}(\ol{\cO})^K$.

Moreover, 
since 
for any given $(\boldsymbol{\vartheta},u)\in \Theta^\b\t C^{2,\b}(\ol{\cO})$,
$G[\cdot,u]:\Theta^\b\to C^{\b}(\ol{\cO})^K$
and 
$G[\boldsymbol{\vartheta},\cdot]:C^{2,\b}(\ol{\cO})\to C^{\b}(\ol{\cO})^K$
are affine mappings,
one can easily compute the  partial derivatives
$\p_u G: \Theta^\b\t C^{2,\b}(\ol{\cO}) \to B(C^{2,\b}(\ol{\cO}), C^{\b}(\ol{\cO})^K)$ 
and 
$\p_{\boldsymbol{\vartheta}} G: \Theta^\b\t C^{2,\b}(\ol{\cO}) \to B(\Theta^\b, C^{\b}(\ol{\cO})^K)$ of $G$ as follows:
$(\p_u G)[\boldsymbol{\vartheta},u](v)=\bL^{{\boldsymbol{\vartheta}}}v$
and 
$(\p_{\boldsymbol{\vartheta}} G)[\boldsymbol{\vartheta},u](\tilde{\boldsymbol{\vartheta}})=\bL^{\tilde{\boldsymbol{\vartheta}}}u+\bf^{\,\tilde{\boldsymbol{\vartheta}}}$
 for all  $(\boldsymbol{\vartheta},u)\in \Theta^\b\t C^{2,\b}(\ol{\cO})$, $\tilde{\boldsymbol{\vartheta}}\in \Theta^\b$ and $v\in C^{2,\b}(\ol{\cO})$.
Moreover, it is clear that $\p_u G$ and $\p_{\boldsymbol{\vartheta}}G$ are both continuous, which implies that 
$G:\Theta^\b\t C^{2,\b}(\ol{\cO}) \to  C^{\b}(\ol{\cO})^K$ is continuously differentiable with derivative
$$
G'[\boldsymbol{\vartheta},u](v,\tilde{\boldsymbol{\vartheta}})=(\p_u G)[\boldsymbol{\vartheta},u](v)+(\p_{\boldsymbol{\vartheta}} G)[\boldsymbol{\vartheta},u](\tilde{\boldsymbol{\vartheta}})=\bL^{{\boldsymbol{\vartheta}}}v+\bL^{\tilde{\boldsymbol{\vartheta}}}u+\bf^{\,\tilde{\boldsymbol{\vartheta}}}
$$
for all  $(\boldsymbol{\vartheta},u)\in \Theta^\b\t C^{2,\b}(\ol{\cO})$, $\tilde{\boldsymbol{\vartheta}}\in \Theta^\b$ 
and $v\in C^{2,\b}(\ol{\cO})$ 
(see \cite[Theorem 7.2-3]{ciarlet2013}).

Therefore, by using the chain rule (see \cite[Theorem 7.1-3]{ciarlet2013}), we see the composite mapping $F_1: \Theta^\b\t C^{2,\b}(\ol{\cO})\to C^{\b}(\ol{\cO})$ is also continuously differentiable with the derivative 
$F_1'[\boldsymbol{\vartheta},u]=(H_\eps)'[G[\boldsymbol{\vartheta},u]]G'[\boldsymbol{\vartheta},u]$ for all $(\boldsymbol{\vartheta},u)\in \Theta^\b\t C^{2,\b}(\ol{\cO})$.
This, along with the fact that $F_2:C^{2,\b}(\ol{\cO})\t \Theta^\b\to  C^{2,\b}({\p\cO})$  is a linear operator,
enables us to conclude the desired  differentiability of the operator $F^\b=(F_1,F_2)$.
\end{proof}

With the above proposition in hand, we are ready to derive the first-order sensitivity equation for the value function 
of the relaxed control problem with respect to the parameter perturbations.

\begin{Theorem}\l{thm:sensitivity}
Suppose (H.\ref{assum:D}) and (H.\ref{assum:rho}) hold. 
Let $\eps>0$,
%$M=\sup_{i,j,k}|\sigma^{ij}_k|_{0;\ol{\cO}}$,
$(\Theta^\b)_{\b\in (0,1]}$ be the Banach spaces defined as in \eqref{eq:X_b},
$\boldsymbol{\vartheta}_0=((\sigma_k\sigma^T_k/2,b_k,c_k,f_k)_{k\in \cK},g)$,
 ${u}^\eps\in C(\ol{\cO}) \cap C^2(\cO)$ 
 be the solution to the Dirichlet problem \eqref{eq:hjb_relax}
 (with the coefficients $\boldsymbol{\vartheta}_0$),
and $\b_0\in (0,1)$ be the  constant  in Proposition \ref{prop:Holder_bdd}.
 
Then 
it holds 
%for some constant $\b_0=\b_0(n,\nu,M)\in (0,1)$ that,
for each $\b\in (0,\min(\b_0,\theta)]$ that,
there exists a neighborhood $\cV$ of $\boldsymbol{\vartheta}_0$ in $\Theta^\b$,
 a neighborhood $\cW$ of $u^\eps$ in $C^{2,\b}(\ol{\cO})$,
 and 
 a  mapping $\cS:\cV\to \cW$ satisfying the following properties:
 \bn[(1)]
 \item 
 for each $\tilde{\boldsymbol{\vartheta}}\in \cV$, $\cS[\tilde{\boldsymbol{\vartheta}}]$ is the classical solution to 
 the following Dirichlet problem: %\eqref{eq:hjb_relax}: 
$$
H_\eps(\bL^{\tilde{\boldsymbol{\vartheta}}} u+\bf^{\,\tilde{\boldsymbol{\vartheta}}})=0
\q \textnormal{in $\cO$}, 
\q\q
{u}={g}^{\tilde{\boldsymbol{\vartheta}}} \q \textnormal{on $\p\cO$,}
$$
where $(\bL^{\,\boldsymbol{\vartheta}},\bf^{\,\boldsymbol{\vartheta}}, g^{\,\boldsymbol{\vartheta}})$ are defined as in Proposition \ref{prop:differentiability} for each $\boldsymbol{\vartheta}\in \Theta^\b$, 
 \item
$\cS:\cV\to \cW$ is continuously differentiable with 
 $\cS[{\boldsymbol{\vartheta}}_0+\delta\boldsymbol{\vartheta}]=u^\eps+\cS'[{\boldsymbol{\vartheta}}_0]\delta \boldsymbol{\vartheta}+o(|\delta\boldsymbol{\vartheta}|_{\Theta^\b})$ as $|\delta\boldsymbol{\vartheta}|_{\Theta^\b}\to 0$, and
 for each $\delta \boldsymbol{\vartheta}\in \Theta^\b$, $\delta u=\cS'[{\boldsymbol{\vartheta}}_0]\delta \boldsymbol{\vartheta}\in C^{2,\b}(\ol{\cO})$ is the solution to the following Dirichlet problem:
 \bb\l{eq:hjb_sensitivity}
(\lambda^{u^\eps})^T(
\bL^{\boldsymbol{\vartheta}_0} \delta u
+\bL^{\,\delta \boldsymbol{\vartheta}}u^\eps
+
\bf^{\,\delta \boldsymbol{\vartheta}})=0
\q\textnormal{in $\cO$}, 
\qq
\delta u=g^{\,\delta \boldsymbol{\vartheta}} \q \textnormal{on $\p\cO$,}
\ee
where 
%$(\bL^{\,\boldsymbol{\vartheta}},\bf^{\,\boldsymbol{\vartheta}}, g^{\,\boldsymbol{\vartheta}})$ are defined as in Proposition \ref{prop:differentiability} for each $\boldsymbol{\vartheta}\in \Theta^\b$, 
%and 
$\lambda^{u^\eps}:\ol{\cO}\to \Delta_K$ is the function defined as in \eqref{eq:ctrl_relax}.
 \en
\end{Theorem}
\begin{proof}
The desired result comes from a direct application of the implicit function theorem (see \cite[Theorem 7.13-1]{ciarlet2013}).
Theorem \ref{thm:wp_relax} shows that 
%there exists a constant $\b_0=\b_0(n,\nu,M)\in (0,1)$ such that,
the Dirichlet problem \eqref{eq:hjb_relax}
with the coefficients $\boldsymbol{\vartheta}_0$
admits a solution $u^\eps\in C^{2,\b}(\ol{\cO})$
for each $\b\in (0,\min(\b_0,\theta)]$.

Let $\b\in (0,\min(\b_0,\theta)]$ be a fixed constant. We shall consider the mapping 
$F^\b:\Theta^\b\t C^{2,\b}(\ol{\cO})\to C^{\b}(\ol{\cO})\t C^{2,\b}(\p{\cO})$ defined as follows:
$$
F^\b:(\boldsymbol{\vartheta},u)\in  \Theta^\b\t C^{2,\b}(\ol{\cO})\mapsto 
F^\b[\boldsymbol{\vartheta},u]\coloneqq (H_\eps(\bL^{\boldsymbol{\vartheta}} u+\bf^{\,\boldsymbol{\vartheta}}),\tau_D(u-g^{\boldsymbol{\vartheta}})) \in C^{\b}(\ol{\cO})\t C^{2,\b}(\p{\cO}).
$$
Due to the fact that 
$u^\eps\in C^{2,\b}(\ol{\cO})$ satisfies \eqref{eq:hjb_relax} with the coefficients $\boldsymbol{\vartheta}_0$,
we have $H_\eps(\bL^{\boldsymbol{\vartheta}_0} u^\eps+\bf^{\,\boldsymbol{\vartheta}_0})=0$ in $\cO$
 and 
$H_\eps(\bL^{\boldsymbol{\vartheta}_0} u^\eps+\bf^{\,\boldsymbol{\vartheta}_0})\in C^{\b}(\ol{\cO})$, 
which subsequently implies that
$H_\eps(\bL^{\boldsymbol{\vartheta}_0} u^\eps+\bf^{\,\boldsymbol{\vartheta}_0})=0$ on $\ol{\cO}$. 
The boundary condition of \eqref{eq:hjb_relax} implies that 
$\tau_D(u^{\eps}-g^{\boldsymbol{\vartheta}_0})=0$ in $C^{2,\b}(\p{\cO})$. 
Hence $F^\b[\boldsymbol{\vartheta}_0,u^\eps]=0$.

Proposition \ref{prop:differentiability} shows that $F^\b$ is continuously differentiable on $\Theta^\b\t C^{2,\b}(\ol{\cO})$, and 
for each $(\tilde{\boldsymbol{\vartheta}},v)\in \Theta^\b\t C^{2,\b}(\ol{\cO})$,
\begin{align*}
\p_u  F^\b[\boldsymbol{\vartheta}_0,u^\eps](v)
&=\big((\nabla H_\eps)^T(\bL^{\boldsymbol{\vartheta}_0} u^\eps+\bf^{\,\boldsymbol{\vartheta}_0} )\bL^{{\boldsymbol{\vartheta}_0}}v, \tau_Dv\big)
=\big((\lambda^{u^\eps})^T\bL^{{\boldsymbol{\vartheta}_0}}v, \tau_Dv\big),\\
\p_{\boldsymbol{\vartheta}}  F^\b[\boldsymbol{\vartheta}_0,u^\eps](\tilde{\boldsymbol{\vartheta}})
&=\big((\nabla H_\eps)^T(\bL^{\boldsymbol{\vartheta}_0} u^\eps+\bf^{\,\boldsymbol{\vartheta}_0} )(\bL^{\tilde{\boldsymbol{\vartheta}}}u^\eps+\bf^{\,\tilde{\boldsymbol{\vartheta}}}),
-\tau_Dg^{\tilde{\boldsymbol{\vartheta}}}\big)
=\big((\lambda^{u^\eps})^T(\bL^{\tilde{\boldsymbol{\vartheta}}}u^\eps+\bf^{\,\tilde{\boldsymbol{\vartheta}}}),
-\tau_Dg^{\tilde{\boldsymbol{\vartheta}}}\big),
\end{align*}
where we have used the definition  of $\lambda^{u^\eps}\in C^\b(\ol{\cO},\Delta_K)$ (see \eqref{eq:ctrl_relax}).
The classical maximum principle (see e.g.~\cite[Theorem 3.7]{gilbarg1985}) implies
 that the map $\p_u  F^\b[\boldsymbol{\vartheta}_0,u^\eps](\cdot):C^{2,\b}(\ol{\cO})\to C^{\b}(\ol{\cO})\t C^{2,\b}(\p\cO)$ is an injection.
We now show it is also a surjection.
Let $(\hat{f},\hat{g})\in C^{\b}(\ol{\cO})\t C^{2,\b}(\p\cO)$ be given. Then the assumption that $\p\cO\in C^{2,\b}$ enables us to apply \cite[Lemma 6.38]{gilbarg1985}
and
extend   $\hat{g}$ to a function in $C^{2,\b}(\ol{\cO})$, which is still denoted by  $\hat{g}$.
The fact that $\lambda^{u^\eps}\in C^\b(\ol{\cO},\Delta_K)$ (see Theorem \ref{thm:wp_relax}) 
and the  elliptic regularity theory (see \cite[Theorem 6.14]{gilbarg1985}) ensure that the Dirichlet problem 
$\p_u  F^\b[\boldsymbol{\vartheta}_0,u^\eps](v)=(\hat{f},\hat{g}) $ admits a unique solution  $v\in C^{2,\b}(\ol{\cO})$.
Hence we see  
$\p_u  F^\b[\boldsymbol{\vartheta}_0,u^\eps]:C^{2,\b}(\ol{\cO})\to C^{\b}(\ol{\cO})\t C^{2,\b}(\p\cO)$ is a bijection.

Therefore, the implicit function theorem (see \cite[Theorem 7.13-1]{ciarlet2013}) ensures the existence of $\cS\in C^1(\cV,\cW)$ with derivative 
$\cS'[\boldsymbol{\vartheta}_0]=-(\p_u  F^\b[\boldsymbol{\vartheta}_0,u^\eps])^{-1}\p_{\boldsymbol{\vartheta}}  F^\b[\boldsymbol{\vartheta}_0,u^\eps]\in B(\Theta^\b, C^{2,\b}(\ol{\cO}))$.
Hence we have 
$\cS[{\boldsymbol{\vartheta}}_0+\delta\boldsymbol{\vartheta}]=u^\eps+\cS'[\boldsymbol{\vartheta}_0]\delta \boldsymbol{\vartheta}+o(|\delta\boldsymbol{\vartheta}|_{\Theta^\b})$ as $|\delta\boldsymbol{\vartheta}|_{\Theta^\b}\to 0$. Let
$\delta \boldsymbol{\vartheta}\in \Theta^\b$ and $\delta u=\cS'[\boldsymbol{\vartheta}_0]\delta \boldsymbol{\vartheta}$,
the characterization of partial derivatives of $F^\b$ enables us to conclude that $\delta u$ satisfies \eqref{eq:hjb_sensitivity}.
\end{proof}
\begin{Remark}\l{rmk:sensitivity_ctrl}
We can further obtain a first-order expansion of the optimal control $\lambda^{u^\eps}$ in terms of
the  perturbations of the coefficients. 
If $\eps>0$ and  the function $H$ in (H.\ref{assum:rho}) is in $C^3(\R^K)$ (c.f.~${H}_{\textnormal{en}}$ and ${H}_{\textnormal{chks}}$ in Section \ref{sec:relax}), 
then Lemma \ref{lemma:nemytskij} shows that 
$\nabla H_\eps: C^\a(\ol{\cO},\R^K)\to C^\a(\ol{\cO},\R^K)$, $\a\in (0,1]$,
is continuously differentiable with derivative 
$(\nabla H_\eps)'[u]h=(\nabla^2 H_\eps)(u)h$ for all $u,h\in C^\a(\ol{\cO},\R^K)$,
where $\nabla^2 H_\eps$ is the Hessian of $H_\eps$.
Hence, by using the chain rule and Theorem \ref{thm:sensitivity}, we have for all $\b\in (0,\min(\b_0,\theta)]$ that
$$
\lambda^{\cS[{\boldsymbol{\vartheta}}_0+\delta\boldsymbol{\vartheta}]}=\lambda^{u^\eps}+
\big((\nabla ^2H_\eps)(\bL^{\boldsymbol{\vartheta}_0} u^\eps+\bf^{\,\boldsymbol{\vartheta}_0} )\big)
\big(\bL^{\boldsymbol{\vartheta}_0} \delta u+\bL^{\delta \boldsymbol{\vartheta}} u^\eps+\bf^{\,\delta \boldsymbol{\vartheta}} \big)+o(|\delta \boldsymbol{\vartheta}|_{\Theta^\b})
$$
as $|\delta\boldsymbol{\vartheta}|_{\Theta^\b}\to 0$,
where 
$\lambda^{\cS[{\boldsymbol{\vartheta}}_0+\delta\boldsymbol{\vartheta}]}$ is the optimal feedback control  
of the relaxed control problem with the perturbed coefficients ${\boldsymbol{\vartheta}}_0+\delta\boldsymbol{\vartheta}$,
and
$\delta u$ is the classical solution to  \eqref{eq:hjb_sensitivity}.
\end{Remark}

%\color{blue}

With the sensitivity equation \eqref{eq:hjb_sensitivity} in hand,
we now estimate the  precise dependence of $\delta u$ 
on the relaxation parameter $\eps$,
which strengthens the Lipschitz stability result \eqref{eq:u_hat-u} 
by quantifying 
the explicit $\eps$-dependence  of the (local) Lipschitz constant.
Note that Remark \ref{rmk:unstable_eps0} shows that 
 the value function   \eqref{eq:value} 
(in the $C^{2,\b}$-norm)  
 does not depend continuously on the $C^{\b}$-perturbation of the  parameters,
which suggests that for a fixed $\delta \boldsymbol{\vartheta}\in \Theta^\b$, 
the $|\cdot|_{2,\b}$-norm of $\delta u$ will blow up as 
the  parameter $\eps$ tends to $0$.

Since the H\"{o}lder norm of the function $\lambda^{u^\eps}$ in \eqref{eq:hjb_sensitivity}
tends to infinity as $\eps\to 0$,
we first present  a precise \textit{a priori} estimate for
 the classical solutions to linear elliptic equations with $\eps$-dependent coefficients.
The proof will be postponed to Appendix \ref{appendix:lemmas},
where we  first reduce the equation to a constant coefficient equation  involving only second-order terms,
and then apply the classical Schauder estimate.

\begin{Proposition}\l{prop:schauder_eps}
Let 
$\a \in [0,1]$, $\b\in (0,1)$, $\nu, \Lambda>0$, and 
$\cO$ be a bounded domain in $\R^n$
with  $C^{2,\b}$ boundary.
For every $\eps\in (0,1]$, let 
$a_\eps:\ol{\cO}\to \R^{n\t n}$,
$b_\eps:\ol{\cO}\to \R^{n}$
and $c_\eps:\ol{\cO}\to [0,\infty) $ be given functions
satisfying $a_\eps \ge \nu I_n$ on $\ol{\cO}$.
Suppose   that 
$[a^{ij}_\eps]_0,[b^i_\eps]_0, [c_\eps]_0\le \Lambda$
and 
$[a^{ij}_\eps]_\b,[b^i_\eps]_\b, [c_\eps]_\b\le \Lambda\eps^{-\a}$
for all $\eps\in (0,1]$ and $i,j=1,\ldots, n$.
Then for every 
$\eps\in (0,1]$, $f\in C^\b(\ol{\cO})$ and $g\in C^{2,\b}(\ol{\cO})$,
the Dirichlet problem
$$
 a^{ij}_\eps\p_{ij}w+b^i_\eps\p_iw-c_\eps w+f=0, 
\q\textnormal{in $\cO$}, 
\q\q
w=g \q \textnormal{on $\p\cO$}
$$
admits a unique  solution  $w^\eps\in C^{2,\b}(\ol{\cO})$
satisfying
the following estimate  with a constant $C=C(n,\b, \nu,\Lambda,\cO)$:
$$
|w^\eps|_{2,\b}\le C\big(
\eps^{-\a (\b+2)/\b}|f|_{0}
+ [f]_{\b}+\eps^{-\a (\b+2)/\b}|g|_{2,\b}
\big).
$$

\end{Proposition}

Now we present the \textit{a priori} estimate of $\delta u$ exhibiting 
its explicit $\eps$-dependence (cf. \eqref{eq:u_hat-u}),
which applies to
relaxed control problems with  reward functions generated by 
${H}_{\textnormal{en}}$, ${H}_{\textnormal{chks}}$
and ${H}_{\textnormal{zang}}$.

\begin{Theorem}\l{thm:deltau_eps}
Assume the setting of Theorem \ref{thm:sensitivity}
and in addition that 
 the function $H:\R^K\to \R$ in (H.\ref{assum:rho})
has a  Lipschitz continuous gradient. 
Let $\b_0\in (0,1)$ be the  constant  in Proposition \ref{prop:Holder_bdd}
and $\bar{\b}_0=\min(\b_0,\theta)$.
Then 
%there exists 
%a constant $\b_0=\b_0(n,\nu,M,\theta)\in (0,1)$ such that
it holds for all $\eps\in (0,1]$, $\b\in (0,\bar{\b}_0]$ and $\delta \boldsymbol{\vartheta}\in \Theta^\b$ that,
the classical solution $\delta u$ to the Dirichlet problem  \eqref{eq:hjb_sensitivity}
satisfies the estimate
$
|\delta u|_{2,\b}\le C
\eps^{- (\b+2)/\bar{\b}_0}
|\delta\boldsymbol{\vartheta}|_{\Theta^\b},
$
where  $C$ is a constant independent of $\eps$ and $\delta  \boldsymbol{\vartheta}$.
\end{Theorem}
\begin{proof}
Throughout this proof, let  $C$ be a generic constant 
depending  possibly on $\boldsymbol{\vartheta}_0$ and $\b$, but
independent of $\eps$ and $\delta  \boldsymbol{\vartheta}$.
Proposition \ref{prop:Holder_bdd}
shows that 
%there exists a constant $\b_0=\b_0(n,\nu,M,\theta)\in (0,1)$ such that
 $|u^\eps|_{2,\bar{\b}_0}\le C$ for all $\eps\in (0,1]$,
which together with \eqref{eq:ctrl_relax}, the fact that $\nabla H_\eps(x)=\nabla H(\eps^{-1}x)$ for all $x\in \R^K$ (see  \eqref{eq:H_eps}) and the Lipschitz continuity of $\nabla H$
implies that $|\lambda^{u^\eps}|_{0}\le C$ and $|\lambda^{u^\eps}|_{\bar{\b}_0}\le C\eps^{-1}$ for all $\eps\in (0,1]$.
Consequently, we have for all $\b\in (0,\bar{\b}_0]$ and $\eps\in (0,1]$ that 
$|\lambda^{u^\eps}|_{\b}\le C|\lambda^{u^\eps}|^{\b/\bar{\b}_0}_{\bar{\b}_0}|\lambda^{u^\eps}|^{(\bar{\b}_0-\b)/\bar{\b}_0}_{0}
\le C\eps^{-\b/\bar{\b}_0}$.

Now let us fix $\b\in (0,\bar{\b}_0]$ and $\delta \boldsymbol{\vartheta}\in \Theta^\b$.
Since $\lambda^{u^\eps}\in \Delta_K$ on $\ol{\cO}$, 
we can apply
Proposition \ref{prop:schauder_eps} (with $\a=\b/\bar{\b}_0$) to \eqref{eq:hjb_sensitivity}
and conclude the   desired estimate  from
the following inequality:
\begin{align*}
|\delta u|_{2,\b}
&\le C\big(
\eps^{- (\b+2)/\bar{\b}_0}
|(\lambda^{u^\eps})^T(\bL^{\,\delta \boldsymbol{\vartheta}}u^\eps
+
\bf^{\,\delta \boldsymbol{\vartheta}})|_{0}
+ 
|(\lambda^{u^\eps})^T(\bL^{\,\delta \boldsymbol{\vartheta}}u^\eps
+
\bf^{\,\delta \boldsymbol{\vartheta}})|_{\b}
+\eps^{- (\b+2)/\bar{\b}_0}|g^{\,\delta \boldsymbol{\vartheta}} |_{2,\b}
\big)\\
&\le C\big(
\eps^{- (\b+2)/\bar{\b}_0}
|\delta\boldsymbol{\vartheta}|_{\Theta^\b}
+ 
|\lambda^{u^\eps}|_\b
|\bL^{\,\delta \boldsymbol{\vartheta}}u^\eps
+
\bf^{\,\delta \boldsymbol{\vartheta}}|_{\b}
\big)
\le C\big(
\eps^{- (\b+2)/\bar{\b}_0}
|\delta\boldsymbol{\vartheta}|_{\Theta^\b}
+ 
\eps^{-\b/\bar{\b}_0}
|\delta\boldsymbol{\vartheta}|_{\Theta^\b}
\big). \qedhere
\end{align*}
%which leads to  
%due to the fact that $\b>0$.
\end{proof}

\color{black}

\section{Convergence analysis for vanishing relaxation parameter}\l{sec:convergence}
In this section, we analyze the convergence  of the relaxed control problem \eqref{eq:value_relax}
to the original control problem \eqref{eq:value} as the relaxation parameter tends to zero.
In particular, 
with the help of the HJB equations \eqref{eq:hjb} and \eqref{eq:hjb_relax},
we shall establish  first-order monotone convergence of the value functions, 
and also uniform convergence of the feedback controls (in  regions where a strict complementary condition is satisfied).

%
%The following proposition presents an \textit{a priori} estimate of classical solutions to 
%\eqref{eq:hjb_relax}, which is essential for our subsequent analysis.
%We postpone the proof to Appendix \ref{appendix:lemmas}, 
%which adapts the  technique in \cite[Theorem 7.5 on p.~127]{chen1998} to HJB equations with compact control sets, and reduces the problem to 
%an \textit{a priori} estimate for
% HJB equations involving only  principal  terms.
%
%\begin{Proposition}\l{prop:Holder_bdd}
%Suppose  (H.\ref{assum:D}) and (H.\ref{assum:rho}) hold,
%and let $M=\sup_{i,j,k}|\sigma^{ij}_k|_0<\infty$.
%Then 
%there exists a constant $\b_0=\b_0(n,\nu,M)\in (0,1)$,
%%{independent of  $K$, $\eps$, $c_0$},
% such that 
% it holds  for all $\b\in (0,\min(\b_0,\theta)]$ that,
%if 
%$u^\eps\in C^{2,\b}(\ol{\cO})$ is a solution to 
%the Dirichlet problem \eqref{eq:hjb_relax}
%with the relaxation parameter $\eps>0$,
%then $u^\eps$ satisfies the  estimate that 
% $ |u^\eps|_{2,\b}\le C(|g|_{2,\b}+\eps c_0+1)$,
% where 
%the constant $C$ depends only on  $n$, $\nu$, $\Lambda$, $\b$ and  $\cO$.
%\end{Proposition}

We first study the convergence of the value functions of the relaxed control problems.
The following theorem shows that,  as the relaxation parameter $\eps$ tends to zero,
the value function  \eqref{eq:value_relax} converges monotonically to the  value function \eqref{eq:value} 
in $C^{2,\b}(\ol{\cO})$ 
with  first order.

\begin{Theorem}\l{thm:value_converge}
Suppose  (H.\ref{assum:D}) and (H.\ref{assum:rho}) hold. 
Let 
$\b_0\in (0,1)$ be the  constant  in Proposition \ref{prop:Holder_bdd},
and $u\in C(\ol{\cO}) \cap C^2(\cO)$ (resp.~$u^\eps\in C(\ol{\cO}) \cap C^2(\cO)$) be the solution to 
\eqref{eq:hjb} (resp.~\eqref{eq:hjb_relax} with  parameter $\eps>0$). 
Then we have $u^{\eps_1}\ge u^{\eps_2}$ for all $\eps_1\ge \eps_2> 0$.
Moreover, 
%there exists a constant $\b_0\in (0,1)$,
%depending only on $n$, $\nu$ and $\sup_{i,j,k}|\sigma^{ij}_k|_{0;\ol{\cO}}$,
%such that
%we have
it holds for any  $\b\in (0,\min(\b_0,\theta))$ that 
$(u^\eps)_{\eps> 0}$ converges to $u$ in $C^{2,\b}(\ol{\cO})$ as $\eps\to 0$,
and satisfies the  estimate:
\bb\l{eq:error_u}
0\le u^\eps-u\le \bigg(\exp\left[\left(\tfrac{\max_{k\in \cK}\sum_{i=1}^n|b^i_k|_0}{\nu/2}+1\right)\textnormal{diam}(\cO)\right]-1\bigg)\f{2\eps c_0}{\nu}.
\ee
\end{Theorem}
\begin{proof}
Let $(F_\eps)_{\eps\ge 0}$ be defined as in \eqref{eq:hjb} and \eqref{eq:hjb_relax},
and $\eps_1\ge \eps_2>0$ be given constants.
Lemma \ref{lemma:rho_convex} shows that
 $\rho\le 0$ on $\Delta_K$, and
$H_\eps(x)=\max_{y\in \Delta_K}\big(y^Tx-\eps\rho(y)\big)$ for all $x\in \R^K$.
 Hence, we have  $H_{\eps_1}\ge H_{\eps_2}$, and 
\begin{align*}
0=F_{\eps_1}[u^{\eps_1}]-F_{\eps_2}[u^{\eps_2}]\ge F_{\eps_2}[u^{\eps_1}]-F_{\eps_2}[u^{\eps_2}]
=\eta^T\bL(u^{\eps_1}-u^{\eps_2}),
\end{align*}
where we write
$\eta \coloneqq \int_0^1 (\nabla H_{\eps_2})(\bL u^{\eps_2}+\bf+s\bL (u^{\eps_1}-u^{\eps_2}))\,ds$.
Since
$\eta(x)\in \Delta_K$ for all $x\in \cO$,
%$\sum_{k=1}^K\eta_k c_k\ge 0$,  
we can deduce from the 
classical maximum principle (see e.g.~\cite[Theorem 3.7]{gilbarg1985}) that
$\inf_{x\in \ol{\cO}}(u^{\eps_1}-u^{\eps_2})(x)\ge \inf_{x\in \p\cO}(u^{\eps_1}-u^{\eps_2})^-(x)=0$.

Similarly, for any given $\eps>0$, we can obtain from Lemma \ref{lemma:rho_convex}\eqref{item:H_eps} that 
\begin{align*}
0&=F_{\eps}[u^{\eps}]-F_{0}[u]\le  F_{\eps}[u^{\eps}]-(F_{\eps}[u]-\eps c_0)
=\tilde{\eta}^T\bL(u^{\eps}-u)+\eps c_0,\\
0&=F_{\eps}[u^{\eps}]-F_{0}[u]\ge  F_{\eps}[u^{\eps}]-F_{\eps}[u]
=\tilde{\eta}^T\bL(u^{\eps}-u),
\end{align*}
where we have
$\tilde{\eta} \coloneqq \int_0^1 (\nabla H_{\eps})(\bL u+\bf+s\bL (u^{\eps}-u))\,ds$.
By using 
$a_k=\sigma_k(\sigma_k)^T/2$,
\eqref{eq:elliptic} in (H.\ref{assum:D}),
 and {the fact that $\tilde{\eta}\in \Delta_K$ on $ \ol{\cO}$},
we deduce that
 $\sum_{k=1}^K\tilde{\eta}_k c_k\ge 0$
and $\sum_{k=1}^K \tilde{\eta}_k a_k\ge (\nu/2) I_n$.
Hence  the 
classical maximum principle (see e.g.~\cite[Theorem 3.7]{gilbarg1985}) 
and the fact that $u^\eps=u$ on $\p\cO$ give us
the estimate \eqref{eq:error_u}.

Finally, the \textit{a priori} bound in Proposition \ref{prop:Holder_bdd} 
and the
Arzel\`a--Ascoli  theorem
ensure that 
%there exists a constant $\b_0\in (0,1)$,
%depending only on $n$, $\nu$ and $\sup_{i,j,k}|\sigma^{ij}_k|_0$,
%such that
for any given $\b\in (0,\min(\b_0,\theta))$,
there exists a subsequence $(u^{\eps_m})_{m\in \N}$ with $\lim_{m\to \infty}\eps_m=0$, 
such that $(u^{\eps_m})_{m\in \N}$ converges in $C^{2,\b}(\ol{\cO})$ to some function $\bar{u}$
and $\bar{u}\in C^{2,\min(\b_0,\theta)}(\ol{\cO})$.
Since the entire sequence $(u^{\eps})_{\eps>0}$ converges monotonically to $u$, 
we have $u=\bar{u}$ and  $(u^{\eps})_{\eps>0}$ converges  to $u$ in $C^{2,\b}(\ol{\cO})$ for all
$\b\in (0,\min(\b_0,\theta))$.
\end{proof}
\begin{Remark}
The estimate  \eqref{eq:error_u} depends on $\eps,c_0, \nu, b^i_k$ and $\cO$ in a rather intuitive way. Note that, compared with the original control problem \eqref{eq:value},
the relaxed control problem
\eqref{eq:value_relax}  introduces additional randomness for  exploration  to achieve more robust decisions, especially at regions where two or more strategies lead to similar performances based on the given model
(the points at which $\argmax$ in \eqref{eq:ctrl} is not a singleton).
The relation \eqref{eq:ctrl} between feedback controls and the derivatives of value functions further suggests that such regions usually correspond to a sign change of derivatives of value functions.

The exploration surplus in the value functions clearly increases as $\eps$ or $c_0$ increase (see Lemma \ref{lemma:rho_convex}\eqref{item:rho} and Figure \ref{fig:comparision_reward}), since the same level of exploration will bring more rewards. 
It will also increase with $\textnormal{diam}(\cO)$ as the dynamics will stay in $\cO$ longer.
Furthermore, due to the lack of regularization from the Laplacian operator,
a small volatility  or a large drift-to-volalitly ratio of the underlying model usually leads to 
a more rapidly changing value function, which
increases the occurrence of the uncertain regions
and
 makes the relaxation approach more beneficial.
\end{Remark}

Now we turn to investigate the convergence of the feedback relaxed control \eqref{eq:ctrl_relax}.
To distinguish different convergence behaviours related to reward functions generated by $H_{\textnormal{en}}$ and $H_{\textnormal{zang}}$,
we first introduce the following concept for functions which only modify the pointwise maximum function locally near the kinks.

\begin{Definition}\l{def:S_loc}
Let $n\in \N$,  we say a function $\phi:\R^n\to \R$ satisfies ($S_{\textnormal{loc}}$) with constant $\vartheta\ge 0$, if 
it holds
for all $k= 1,\ldots, n$ and 
$x\in \R^n$ with $x_k\ge x_j+\vartheta$, $\fa j\not=k$,
that $\phi(x)=x_k$.
\end{Definition}

It is clear that the pointwise maximum function  on $\R^n$ satisfies ($S_{\textnormal{loc}}$) with $\vartheta=0$,
and the two-dimensional function $H_{\textnormal{zang}}$ defined in \eqref{eq:H_zang} 
satisfies ($S_{\textnormal{loc}}$) with $\vartheta=1/2$.
The following lemma shows that  property ($S_{\textnormal{loc}}$)  is preserved under function composition and scaling,
which consequently implies that the recursively constructed $K$-dimensional 
$H_{\textnormal{zang}}$ and its corresponding scaled function $(H_{\textnormal{zang}})_\eps$ (cf.~\eqref{eq:H_eps}) satisfy
($S_{\textnormal{loc}}$).
The proof follows directly from  Definition \ref{def:S_loc}, and is included in Appendix \ref{appendix:lemmas}.
% for   the reader's convenience.

\begin{Lemma}\l{lemma:S_loc} 
\bn[(1)]
\item \l{item:composition}
For each  $n\in \N$, let  $H^{(n)}_0:\R^n\to \R$ be the $n$-dimensional pointwise maximum function (see \eqref{eq:H_eps}). 
Let $n_1=2$, $n_2,n_3\in \N$, 
$(\vartheta_i,c_i)_{i=1}^3\subset [0,\infty)$,
 $\phi_i:\R^{n_i}\to \R$, $i=1,2,3$,   be given functions,
and 
$\phi:\R^{n_2+n_3}\to \R$ be the function satisfying 
for all $x=(x_1,\ldots, x_{n_2+n_3})^T\in \R^{n_2+n_3}$ that
$\phi(x)= \phi_1(\phi_2(x^{(1)}),\phi_3(x^{(2)}))$
with $x^{(1)}=(x_1,\ldots, x_{n_2})$ and $x^{(2)}=(x_{n_2+1},\ldots, x_{n_2+n_3})$.
Suppose that for each $i=1,2,3$, 
the function $\phi_i$ satisfies 
 ($S_{\textnormal{loc}}$) with constant $\vartheta_i$, and $\phi_i(x)\le H^{(n_i)}_0(x)+c_i$ for all $x\in \R^{n_i}$.
 Then 
  the function $\phi$ satisfies ($S_{\textnormal{loc}}$) with  constant $\max(\vartheta_2,\vartheta_3,c_2+\vartheta_1,c_3+\vartheta_1)$, and it holds for all $x\in \R^{n_2+n_3}$ that 
  $\phi(x)\le H^{(n_2+n_3)}_0(x)+c_1+\max(c_2,c_3)$.
\item\l{item:scaling}
 If $\phi:\R^n\to \R $ satisfies ($S_{\textnormal{loc}}$) with constant $\vartheta\ge 0$, then for each $\eps> 0$, the scaled function  $\phi_\eps:x\in \R^n\mapsto \eps \phi(\eps^{-1}x)\in \R$ satisfies ($S_{\textnormal{loc}}$) with constant $\eps\vartheta$.
  \en
\end{Lemma}

The following proposition presents several important convergence properties of the functions $(\nabla H_{\eps})_{\eps> 0}$.
In the sequel, 
we shall denote by
$e_k\in \R^K$, $k\in \cK$,    the unit vector from the
$k$-th column of the identify matrix $I_K$,
and by $\textnormal{conv}(S)$ the convex hull of a given set $S\subset \R^K$.
%and by $V\subset \subset U$ if $V\subset \ol{V}\subset U$ and $\ol{V}$ is compact, i.e., $V$ is strictly contained in $U$. 

\begin{Proposition}\l{prop:grad_H_conv}
Suppose (H.\ref{assum:rho}) holds. Let $(H_\eps)_{\eps\ge 0}$ be defined as in \eqref{eq:H_eps},
$(\p H_0)(x)=\textnormal{conv}(\{e_k\in \R^K\mid x_k=H_0(x), k\in \cK\})$ for all $x\in \R^K$,
and $U=\{x\in \R^K\mid \textnormal{$(\p H_0)(x)$ is a singleton}\}$. Then it holds for all 
$x\in \R^K$  and  compact subset $\cC \subset U$ that 
\bn[(1)]
\item \l{item:grad_H_ptr}
$\lim_{k\to \infty}\textnormal{dist}((\nabla H_{\eps_k})(x_k), (\p H_0)(x))=0$ provided that 
$\lim_{k\to \infty}x_k=x$ and $\lim_{k\to \infty}\eps_k=0^+$, 
\item \l{item:grad_H_uniform}
$(\nabla H_\eps)_{\eps>0}$ converges uniformly to $\p H_0$ on $\cC$ as $\eps\to 0$. 
If we further suppose the function $H:\R^K\to \R$ in (H.\ref{assum:rho}) satisfies ($S_{\textnormal{loc}}$) with constant $\vartheta\ge 0$, 
then there exists $\eps_0>0$ such that $(\nabla H_\eps)(x)=(\p H_0)(x)$ for all $x\in \cC$ and $\eps\in (0,\eps_0]$.
\en
\end{Proposition}
\begin{proof}
We first establish Property \eqref{item:grad_H_ptr} by considering the following function: 
$$
\phi:(x, \eps, y)\in \R^K\t [0,1]\t \Delta_K\mapsto y^Tx-\eps\rho(y)\in \R.
$$
Note that Lemma \ref{lemma:rho_convex}\eqref{item:rho}  shows that the restriction of  $\rho$ on $\Delta_K$ is continuous,
which subsequently implies that  $\phi$ is a continuous function. Then we can deduce from \cite[Theorem 17.31]{aliprantis2006} 
that the set-valued mapping 
$
\Xi:(x,\eps)\in  \R^K\t [0,1]\rightrightarrows \argmax_{y\in \Delta_K}\phi(x,\eps,y)\subset \Delta_K
$ is upper hemicontinuous,  which along with the fact that $\Xi(x,\eps)=(\nabla H_{\eps})(x)$  for all $(x,\eps)\in  \R^K\t (0,1]$ (see  Lemma \ref{lemma:rho_convex}\eqref{item:H_eps}) 
enables us to deduce 
$\lim_{k\to \infty}\textnormal{dist}((\nabla H_{\eps_k})(x_k), \Xi(x,0))=0$
for any given $\lim_{k\to \infty}x_k=x$ and $\lim_{k\to \infty}\eps_k=0^+$.
Property \eqref{item:grad_H_ptr} now follows from the fact that
$\Xi(x,0)=(\p H_0)(x)$
(see e.g.~\cite[Theorem 2]{poliquin1994}).

Now we shall prove Property \eqref{item:grad_H_uniform}. We first define the set 
$U_k=\{x\in \R^K\mid x_k>x_j, \fa j\not=k\}$ for  each $k\in \cK$. It is clear that 
$(U_k)_{k\in \cK}$ are disjoint open convex sets, $U=\cup_{k\in \cK}U_k$, and 
it holds for all  $k\in \cK$ and $x\in U_k$ that $H_0$ is differentiable at $x$ with $(\nabla H_0)(x)=e_k=(\p H_0)(x)$.

Let $\cC\subset U$ be a compact set, then 
 we have
$\cC=\cup_{k\in \cK}(\cC\cap U_k)$
due to $U=\cup_{k\in \cK}U_k$. 
Let us fix an arbitrary index  $k\in \cK$. By using the fact that $(U_k)_{k\in \cK}$ are disjoint open sets, we can deduce that 
$\cC\cap U_k$ is also compact. Since  $(H_\eps)_{\eps\ge 0}$ are convex and differentiable on $U_k$ and $\lim_{\eps\to 0} H_\eps(x)=H_0(x)$ for all $x\in U_k$, we can deduce from the convexity of $U_k$ and \cite[Theorem 25.7]{rockafellar1970} that $(\nabla H_\eps)_{\eps>0}$ converges uniformly to $\nabla H_0=\p H_0$ on $\cC\cap U_k$. Since $\cK$ is a finite set, we have shown the desired uniform convergence on $\cC$.

Moreover, for each $k\in \cK$, the compactness of $\cC\cap U_k$ implies that there exists $\eps_{0,k}>0$ such that 
$\cC\cap U_k\subset \{x\in \R^K\mid x_k>x_j+\eps_{0,k}, \fa j\not=k\}$. Then, if $H$ satisfies ($S_{\textnormal{loc}}$) with constant $\vartheta\ge 0$, then Lemma \ref{lemma:S_loc}\eqref{item:scaling} shows that 
for all $\eps>0$ satisfying $\eps \vartheta\le \eps_{0,k}$, we have $H_\eps= H_0$ (and hence $\nabla H_\eps= \nabla H_0$) on $\cC\cap U_k$. 
Hence, by setting $\eps_0>0$ to be a constant satisfying $\eps_0\vartheta\le \min_{k\in \cK}\eps_{0,k}$,  we can conclude  for all $\eps\in (0,\eps_0]$ that $\nabla H_\eps= \nabla H_0=\p H_0$ on $\cC$.
\end{proof}

Now we are ready to present the convergence  of the feedback relaxed control \eqref{eq:ctrl_relax}. 
Note that 
the H\"{o}lder continuity of the relaxed controls \eqref{eq:ctrl_relax}
and  the possible discontinuity of the feedback control  \eqref{eq:ctrl}  suggest that
the sequence $(\lambda^{u^\eps})_{\eps>0}$ in general  does not converge uniformly to $\a^u$ on ${\cO}$ as $\eps\to 0$.
Thus 
we shall show that the relaxed controls converge  in terms of 
the Hausdorff metric everywhere in $\cO$, and converge uniformly on compact subsets of the following region:
\bb\l{eq:O_st}
\cO_{{\textnormal{st}}}=\bigg\{x\in \cO
\biggm\vert  %\mid 
\textnormal{
$
\argmax_{k\in \cK}
\big(
\cL_ku(x)+f_k(x) 
\big)$ is a singleton}\bigg\},
\ee 
where $u\in C(\ol{\cO}) \cap C^2(\cO)$ is the solution to \eqref{eq:hjb} (or equivalently the value function \eqref{eq:value}
if the function $\sigma\in \bS^n_0$;
see Theorem \ref{thm:verification}), and $(\cL_k)_{k\in \cK}$ are the elliptic operators defined as in \eqref{eq:L_k}.
Note that $\cO_{{\textnormal{st}}}$ contains the points at which a strict complementary condition is satisfied, i.e., the optimal feedback control strategy of \eqref{eq:value} is uniquely determined.

\begin{Theorem}\l{thm:ctrl_converge}
Suppose  (H.\ref{assum:D}) and (H.\ref{assum:rho}) hold. 
Let
%  (resp.~$u^\eps\in C(\ol{\cO}) \cap C^2(\cO)$)  (resp.~\eqref{eq:hjb_relax} with a given parameter $\eps>0$), 
$(\lambda^{u^\eps})_{\eps>0}$ be the functions defined as in \eqref{eq:ctrl_relax} for each $\eps>0$, 
 $u\in C(\ol{\cO}) \cap C^2(\cO)$ be the solution to \eqref{eq:hjb},
and $\cO_{{\textnormal{st}}}$ be the set  defined as in \eqref{eq:O_st}.
Then we have for all $x\in \cO$
and $(x_\eps)_{\eps>0}\subset \cO$ with 
$\lim_{\eps\to 0}x_\eps=x$
 that 
\bb\l{eq:ctrl_conv_ptr}
\lim_{\eps\to 0}\textnormal{dist}\bigg(\lambda^{u^\eps}(x_\eps), 
\textnormal{conv}\bigg(
\bigg\{e_k\in \R^K\biggm\vert 
k\in \argmax_{k\in \cK}
\big(
\cL_ku(x)+f_k(x) 
\big)\bigg\}
\bigg)\bigg)=0.
\ee
Moreover, it holds for all  compact subset $\cC \subset \cO_{\textnormal{st}}$ that
$(\lambda^{u^\eps})_{\eps>0}$ converges uniformly to 
the function $\lambda^*:x\in\cO_{{\textnormal{st}}} \to e_{\kappa^u(x)}\in \Delta_K$
on $\cC$ as $\eps\to 0$,
where $\kappa^u(x)= \argmax_{k\in \cK}
\big(
\cL_ku(x)+f_k(x) 
\big)$ for all $x\in \cO_{{\textnormal{st}}}$.
If we further suppose the function $H:\R^K\to \R$ in (H.\ref{assum:rho})  satisfies ($S_{\textnormal{loc}}$) with constant $\vartheta> 0$, 
then there exists $\eps_0>0$ such that 
it holds for all $\eps\in (0,\eps_0]$ that 
$\lambda^{u^\eps}\equiv \lambda^*$ on $\cC$.

\end{Theorem}

\begin{proof}
For any give $\eps>0$, let 
$u^\eps\in C(\ol{\cO}) \cap C^2(\cO)$ be the solution to 
\eqref{eq:hjb_relax}.
We first prove \eqref{eq:ctrl_conv_ptr} by fixing an arbitrary point $x\in \cO$. 
By using %\eqref{eq:ctrl},
 \eqref{eq:ctrl_relax} and 
Proposition \ref{prop:grad_H_conv}\eqref{item:grad_H_ptr}, we see it suffices to show 
$\lim_{\eps\to 0} (\bL u^\eps(x_\eps)+\bf(x_\eps))=\bL u(x)+\bf(x)$, where $\bL$, $\bf$ are defined as  those in \eqref{eq:hjb}.
Then the fact that 
$(u^\eps)_{\eps>0}$ converges to $u$ uniformly in $C^2(\ol{\cO})$ (see Theorem \ref{thm:value_converge})
and the continuity of coefficients
enable us to conclude \eqref{eq:ctrl_conv_ptr}.

We now proceed to demonstrate the uniform convergence of $(\lambda^{u^\eps})_{\eps>0}$ in $\cO_{{\textnormal{st}}}$. 
Note that for all $x\in \cO_{{\textnormal{st}}}$, we have $e_{\kappa^u(x)}=(\p H_0)\big(\bL u(x)+\bf(x)\big)$,
where the set-valued mapping $\p H_0:\R^K\rightrightarrows \Delta_K$ is defined as in Proposition \ref{prop:grad_H_conv}.
We further define for any given $k\in \cK$ the set 
$$\cO_{\textnormal{st},k}=\{x\in \cO\mid \cL_ku(x)+f_k(x)>\cL_ju(x)+f_j(x), \,\fa j\not=k\},$$
where   $u\in C(\ol{\cO}) \cap C^2(\cO)$ is the solution  to \eqref{eq:hjb},
and $(\cL_k)_{k\in \cK}$ are the elliptic operators defined as in \eqref{eq:L_k}.
The continuity of the coefficients in $(\cL_k)_{k\in \cK}$ (see (H.\ref{assum:D})) implies that $(\cO_{\textnormal{st},k})_{k\in \cK}$ are disjoint open sets satisfying $\cO_{\textnormal{st}}=\cup_{k\in \cK}\cO_{\textnormal{st},k}$.

Now let $\cC\subset \cO_{{\textnormal{st}}}\subseteq \cO$ be a given compact subset. Then we have $\cC=\cup_{k\in \cK}( \cC\cap \cO_{\textnormal{st},k})$, and $\cC\cap \cO_{\textnormal{st},k}$ is a compact set for each $k\in \cK$.
Let $k\in \cK$ be a fixed index. 
Then the continuity of the coefficients in $(\cL_k)_{k\in \cK}$, the fact that $u\in C^2(\ol{\cO})$,
and  the compactness of $\cC\cap \cO_{\textnormal{st},k}$
imply that, 
there exist constants $C_1,C_2\in (0,\infty)$ such that 
we have for all $x\in  \cC\cap \cO_{\textnormal{st},k}$ and $j\in \cK$ that,
$|\cL_ju(x)+f_j(x)|\le C_2$ and 
$$
0<C_1\le \cL_ku(x)+f_k(x)-\max_{j\not=k}\big(\cL_ju(x)+f_j(x)\big)\le C_2<\infty.
$$
Now by using the fact that 
$(u^\eps)_{\eps>0}$ converges to $u$ uniformly in $C^2(\ol{\cO})$,
we can deduce that there exist $\eps_0,C_1,C_2>0$ such that the same estimates hold for
all  $(u^\eps)_{\eps\in(0,\eps_0]}$.
In other words, 
let $U$ be the set    defined as in Proposition  \ref{prop:grad_H_conv},  
we can introduce the compact set
$$
G_k\coloneqq \{x\in \R^K\mid 0<C_1\le x_k-\max_{j\not =k} x_j\le C_2,\, |x_j|\le C_2, \fa j\in \cK \}\subset U,
$$
and
conclude  for all $\eps\in (0,\eps_0]$, $x\in \cC\cap \cO_{\textnormal{st},k}$ that 
$\bL u^\eps(x)+\bf(x)\in G_k$
and $\bL u(x)+\bf(x)\in G_k$.

For any given $\eta>0$, the uniform convergence of $(\nabla H_\eps)_{\eps>0}$ to $\p H_0$ on $G_k$ 
(see Proposition  \ref{prop:grad_H_conv}\eqref{item:grad_H_uniform})
ensures that 
there exists $\delta_k>0$, such that we have for all $y\in G_k$ and $\eps<\delta_k$ that 
$|(\nabla H_\eps)(y)-(\p H_0)(y)|\le \eta$. Hence,
by using the fact that $\p H_0=\{e_k\}$ on $G_k$,
 we have
 for all $\eps<\min(\delta_k,\eps_0)$ and $x\in \cC\cap \cO_{\textnormal{st},k}$ that 
\begin{align*}
&|\lambda^{u^\eps}(x)-\lam^*(x)|=|(\nabla H_\eps)\big(\bL u^\eps(x)+\bf(x)\big)-(\p H_0)\big(\bL u(x)+\bf(x)\big)|\\
&= |(\nabla H_\eps)\big(\bL u^\eps(x)+\bf(x)\big)-(\p H_0)\big(\bL u^\eps(x)+\bf(x)\big)|\le \eta,
\end{align*}
which shows the uniform convergence of $(\lambda^{u^\eps})_{\eps>0}$ to  $\lambda^*$ on $\cC\cap \cO_{\textnormal{st},k}$. Since $\cC=\cup_{k\in \cK}( \cC\cap \cO_{\textnormal{st},k})$ and $\cK$ is a finite set,
we can conclude the desired uniform convergence on $\cC$.

Finally, if we further suppose $H$ satisfies ($S_{\textnormal{loc}}$) with constant $\vartheta\ge 0$, 
Proposition  \ref{prop:grad_H_conv}\eqref{item:grad_H_uniform} ensures that 
$\nabla H_\eps\equiv \p H_0$ on $G_k$ for all small enough $\eps>0$,
which leads to the fact that 
$\lambda^{u^\eps}\equiv \lam^*$ for all small enough $\eps>0$ on $\cC$ and finishes our proof.
\end{proof}

\begin{Remark}\l{rmk:exact_regularization}
One can identify the unit vector $e_k\in \Delta_K$, $k\in \cK$, as the Dirac measure supported on $\{\ba_k\}$,
which shows that, as the relaxation parameter tends to zero, the agent of the relaxed control problem will emphasize more
on exploitation, and the  relaxed control distribution  will collapse to a pure exploitation strategy for the classical control problem. 

Note that Theorem \ref{thm:ctrl_converge} demonstrates an exact regularization feature of the reward function $\rho_{{\textnormal{zang}}}$ generated by $H_{\textnormal{zang}}$,
which means that we can recover the original control strategy in the region $ \cO_{\textnormal{st}}$  based on the feedback relaxed control \textit{without} sending the relaxation parameter $\eps$ to $0$.
%\color{blue}
The main intuition of the proof  is that 
 the region $\cO_{{\textnormal{st}}}$ can be mapped into a finite number of
  convex sets
(i.e., the sets $(U_k)_{k\in \cK}$ in the proof of Proposition \ref{prop:grad_H_conv}).
Hence, if a reward function only modifies  the pointwise maximum function locally near the kinks,
then one can 
employ 
the local compactness and local convexity structure of $\cO_{{\textnormal{st}}}$ and
 the finiteness of the action set $\bA$,
 and deduce 
the local exact regularization property in the region $\cO_{{\textnormal{st}}}$.
 
\color{black}

The exact regularization feature of $\rho_{{\textnormal{zang}}}$ helps avoid the possible numerical instability for solving the relaxed control problem \eqref{eq:value_relax} with an extremely small relaxation parameter. 
In contrast, the feedback relaxed control $\lambda^{u^\eps}$ based on the entropy reward function $\rho_{\textnormal{en}}$ is always in $(0,1)^K$, and the convergence rate to the original control strategy can be  arbitrarily slow. 
\end{Remark}

\section{Conclusions}
%This paper develops a relaxed control approach with 
%a class of  exploration reward functions 
%to design robust feedback control strategies 
%for multi-dimensional continuous-time stochastic control problems.
%We establish that the regularized control problem has a H\"{o}lder continuous feedback control strategy, which is Lipschitz stable with respect to perturbations in the model parameters.
%We further  construct a first-order adjustment for the optimal strategy of the perturbed system based on the pre-computed feedback relaxed control.
%Finally, we  prove 
%a first-order monotone convergence 
%of value functions 
%and a uniform convergence 
%of feedback relaxed controls 
%for relaxed control problems 
%with vanishing exploration parameters.

%\color{blue}
To the best of our knowledge, this is the first paper which  constructs Lipschitz stable feedback control strategies for 
general multi-dimensional continuous-time stochastic control problems,
and rigorously analyzes the performance of a pre-computed feedback control 
for a perturbed problem in a continuous setting.
We also perform a novel first-order sensitivity analysis for 
the value function and feedback relaxed control 
with respect to perturbations in the model parameters,
and 
 quantify  the explicit dependence 
of the Lipschitz stability of feedback controls
on the exploration parameter.
%which is novel for a control problem in such a generality.
These stability results  provide a theoretical
justification for  recent reinforcement learning heuristics 
that  including an exploration reward in the optimization objective
leads to more robust decision making.

%\color{blue}
A natural next step would be to extend the stability analysis to finite horizon stochastic control problems
and mean-field control problems 
 with continuous action spaces (see e.g.~\cite{gu2019,wang2019}).
The infinite cardinality of action spaces
 implies that  
the corresponding relaxed controls  take values in  an infinite-dimensional space of probability measures, which poses additional challenges for the analysis of the regularized control problems.
 For example, 
 infinite-dimensional convex analysis on 
spaces of measures must be employed to analyze 
the regularity of the modified Hamiltonians 
and the well-posedness of the associated HJB equations. 
Moreover, one must endow the action space of relaxed controls with a suitable metric structure
(such as the Wasserstein metric)
in order to study the spatial regularity and Lipschitz stability of feedback relaxed controls.

\color{black}
Another interesting direction is to design efficient numerical algorithms for solving the regularized control problems in a continuous setting.

\color{black}

\appendix

%\section{Some background results}
%Here, we collect some fundamental results which are used frequently in the paper.
%
%
%\begin{Lemma}\l{lemma:a_priori}{\cite[Theorem 7.2 on p.~125]{chen1998}}
%Let $\cO$ be a bounded connected open  subset of $\R^n$,
%and
%$F:\cO\t \bS^n\to \R$ be a given function.
%Suppose  
%the function $F$ is  differentiable and convex in $r$,
%and 
%there exist constants 
% $\lambda,\Lambda>0$ such that 
%$\lambda I_n\le \big[\f{\p F}{\p r_{ij}}(x,r)\big]\le \Lambda I_n$ for all
%$(x,r)\in \cO\t \bS^n$.
%Then there exists a constant $\a=\a(n,\Lambda/\lambda)\in(0,1)$ such that 
%for any $\b\in (0,\a)$,
%if  we have in addition that
%$\p\cO\in C^{2,\b}$, $g\in C^{2,\b}(\ol{\cO})$,
%and 
%there exist constants $\gamma, {\mu}>0$ such that 
%it holds for all $x,y\in \cO$, $r\in \bS^n$ that 
%$|F(x,r)-F(y,r)|\le {\gamma}({\mu}+|r|)|x-y|^{\b}$,
%then 
%the  Dirichlet problem 
%\bb\l{eq:hjb_principal}
%F(x,D^2u)=0,\q\textnormal{in $\cO$}, 
%\q
%u=g, \q \textnormal{on $\p\cO$,}
%\ee
% admits a unique solution $u\in C^{2,\b}(\ol{\cO})$ satisfying the  estimate
%$[u]_{2,\b}\le C\big[|u|_0+|g|_{2,\b}+\mu\big]$,
%where the constant $C$ depends only on $n$, $\Lambda/\lambda$, $\gamma$, $(\a-\b)^{-1}$ and the $C^{2,\b}$-norm of $\p\cO$.
%\end{Lemma}
%

\section{Proofs of technical results}\l{appendix:lemmas}

\begin{proof}[Proof of Theorem \ref{thm:verification}]

Let 
$\pi=(\Om, \cF, \{\cF_t\}_{t\ge 0}, \bP,W)\in \Pi_{\textnormal{ref}}$, 
$\a\in \cA_\pi$, 
$x\in \ol{\cO}$,
let $X^{\a,x}=(X^{\a,x}_t)_{t\ge 0}$ be the strong solution to \eqref{eq:sde} with control $\a$,
and 
 for all $t\ge 0$, let 
 $Z^{\a,x}_t=\int_0^t c(X^{\a,x}_s,\a_s)\,ds$.
It is shown in  
\cite[Lemma 3.1]{buckdahn2016}  that  
 $\ex[\exp(\mu\tau^{\a,x})]<\infty$ 
 for some  constant  $\mu>0$,
 which implies that 
 $\tau^{\a,x}<\infty$ with probability 1.
Applying It\^{o}'s formula to the function $\phi(y,z)=u(y)\exp(-z)$, $(y,z)\in \R^n\t \R$, gives us that
\begin{align}\l{eq:ito}
\begin{split}
\ex^{\bP}[\phi(X^{\a,x}_{\tau^{\a,x}},Z^{\a,x}_{\tau^{\a,x}})]
&=
\phi(x,0)+\ex^{\bP}
\left[
\int_0^{\tau^{\a,x}}(\cL_{X^{\a,x}} \phi+c_\a \p_z\phi )(X^{\a,x}_t,Z^{\a,x}_t)\,dt
\right]\\
&=
u(x)+\ex^{\bP}
\left[
\int_0^{\tau^{\a,x}}(\cL_{X^{\a,x}} u-c_\a u )(X^{\a,x}_t) \Gamma^{\a,x}_t 
\,dt
\right],
\end{split}
\end{align}
where $\cL_{X^{\a,x}}$ is the generator of the controlled dynamics $X^{\a,x}$,
and   
$\Gamma^{\a,x}_t=\exp\big(-\int_0^t c(X^{\a,x}_s,\a_s)\,ds\big)$ 
for all $t\in [0,\tau^{\a,x}]$.
%$(\Gamma^{\a,x}_s)_{s\in [0, \tau^{\a,x}]}$ is defined as in \eqref{eq:gamma}.
The fact that  $u$ is a solution to \eqref{eq:hjb} implies that 
for $\bP$-a.s. $\om\in\Om$, and $t\in [0,\tau^{\a,x}(\om)]$, 
\begin{align}\l{eq:max_Hamiltonian}
\begin{split}
&(\cL_{X^{\a,x}} u-c_\a u)(X^{\a,x}_t (\om)) +f(X^{\a,x}_t(\om),\a_t(\om))
\\
&\le \max_{k\in \cK} 
\bigg(a^{ij}(\cdot,\ba_k)\p_{ij}u(\cdot)+b^i(\cdot,\ba_k)\p_iu(\cdot)
-c(\cdot,\ba_k) u(\cdot)+f(\cdot,\ba_k)\bigg)(X^{\a,x}_t(\om))
=0,
\end{split}
\end{align}
Then, by rearranging the terms,
 using the fact that 
$\phi(X^{\a,x}_{\tau^{\a,x}},Z^{\a,x}_{\tau^{\a,x}})=g(X^{\a,x}_{\tau^{\a,x}})\Gamma^{\a,x}_{\tau^{\a,x}}$
and taking the supremum over 
%we deduce that  $u(x)\ge J(x,\a)$
  all 
$\a\in \cA_\pi$
and $\pi \in \Pi_{\textnormal{ref}}$,  we can deduce that 
$u(x)\ge v(x)$ for all $x\in \ol{\cO}$.

We proceed to show $\a^u$ is a feedback control of \eqref{eq:value} (cf.~Definition \ref{def:feedback}).
Let $\a^u:\ol{\cO}\to \bA$ be a Borel measurable function satisfying \eqref{eq:ctrl},
and $\tilde{\a}^u:\R^n\to \bA$  be an extension of $\a^u$
such that 
 $\tilde{\a}^u=\a^u$ on $\ol{\cO}$
 and $\tilde{\a}^u=\ba_1$ on $\ol{\cO}^c$.
We shall consider the functions 
$b_\a :\R^n\to \R^n$,
$\sigma_\a :\R^n\to \bS^n_0$
such that
$b_\a(x)={b}(x,\tilde{\a}^u(x))$, $\sigma_\a(x)={\sigma}(x,\tilde{\a}^u(x))$
for all $x\in \R^n$.
The measurability of $\a^u$ and the continuity of $b$, $\sigma$ imply that $b_\a$, $\sigma_\a$ and $\tilde{\a}^u$ are Borel measurable.  
%Moreover, 
%it is shown in \cite[Exercise 8.24]{abadir2005} that 
%any matrix $A\in \bS^n_0$ has a unique square root $A^{1/2}\in \bS^n_0$,
%which together with  the assumption that $\sigma\in \bS^n_0$, 
%ensures that  $\sigma$ is the unique positive semidefinite square root of the matrix $\sigma\sigma^T\in \bS^n_0$.
%The uniform ellipticity condition \eqref{eq:elliptic} subsequently implies that 
% $\sigma_\a(x) \ge \sqrt{\nu} I_n$ 
%for all $x\in \R^n$.
Then, for any given $x\in \R^n$, 
by using the boundedness of functions $b_\a,\sigma_\a$, and 
 \cite[Theorem 1]{mishura2016}, we can deduce that there exists 
$\pi^x=(\Om^{x}, \cF^{x}, \{\cF^{x}_t\}_{t\ge 0}, \bP^{x}, W)\in \Pi_{\textnormal{ref}}$,  
and an  $\{\cF^{x}_t\}_{t\ge 0}$-progressively measurable continuous process $(X^{x}_t)_{t\ge 0}$,
 such that
$X^{x}_0=x$,
and
\bb\l{eq:sde_feedback}
dX^{x}_t= b(X^{x}_t,\tilde{\a}^u(X^{x}_t)) \,dt+ \sigma(X^{x}_t,\tilde{\a}^u(X^{x}_t))\,dW_t
\q
\textnormal{for all $t\ge 0$ and $\bP^{x}$-a.s.}
\ee
Thus 
we can obtain from the definition of $\tilde{\a}^u$ that 
$(X^{x}_t)_{t\ge 0}$ satisfies \eqref{eq:X^*} with $h=\a^u$.
Moreover, \cite[Theorem 2.2.4 on p.~54]{krylov1980} implies that 
%$\ex^{\bP^{x}}[\int_0^{\tau^x}1_\cO(X_s^x)\,ds]\le C\textrm{meas}(\cO)<\infty$, 
%from which, we obtain
%for $\bP^{x}$-a.s.~that $\tau^x<\infty$,
%and
$\ex^{\bP^{x}}[\int_0^{\tau^{\a^u,x}} \left(| b(X^{x}_s,{\a}^u(X^{x}_s))|+|\sigma(X^{x}_s,{\a}^u(X^{x}_s))|^2\right)\,ds]<\infty$,
which 
shows that 
  $\a^u$ is a feedback control of \eqref{eq:value}.

It remains to show $\a^u$ is an optimal feedback control.
If $x\in\p\cO$, we can deduce from the definition that  $\tau^{\tilde{\a}^u,x}=0$,
which shows that $g(x)=g(X^x_{\tau^{\tilde{\a}^u,x}})=J(x,{\a}^u)$,
where $J(x,{\a}^u)$ is defined as in \eqref{eq:J_h}.
Similarly, we have
for all 
$\pi\in \Pi_{\textnormal{ref}}$, 
$\a\in \cA_\pi$, 
$x\in \p{\cO}$ that
 the first exit time of $X^{\a,x}$ from $\cO$ is 0,
i.e., $\tau^{\a,x}=0$, which implies that $v(x)=g(x)$. 
Hence, we can deduce from the fact that $u$ satisfies the boundary condition of \eqref{eq:hjb} that
$u(x)=g(x)=v(x)= J(x,{\a}^u)$ for all $x\in \p\cO$.

For each $x\in \cO$, let  $X^{x}$ be a progressively measurable continuous process  satisfying the SDE \eqref{eq:sde_feedback}, 
 defined on the reference probability system $\pi^x\in \Pi_{\textnormal{ref}}$.
The assumption that $\a^u$ satisfies \eqref{eq:ctrl}
ensures that $\tilde{\a}^u(X^{x})$ and $X^{x}$ obtain the equality in \eqref{eq:max_Hamiltonian} for  
$\bP$-a.s. $\om\in\Om$, and $t\in [0,\tau^{\tilde{\a}^u,x}(\om)]$, 
from which, 
by using similar arguments as  \eqref{eq:ito},
we can obtain that
$u(x)=J(x,{\a}^u)$ (c.f.~\eqref{eq:J_h}).
On the other hand, owing to the fact that  $\tilde{\a}^u(X^x)\in \cA_{\pi^x}$, we have 
by the definition of $v$ that
$u(x) \le v(x)$ for all $x\in \cO$.
Combining this with  the fact that 
$u(x)\ge v(x)$ for all $x\in \cO$,
we can conclude that 
$u(x)=v(x)=J(x,\a^u)$ in $\cO$,
which shows that  $\a^u$ is an optimal feedback control
and $u\equiv v$ on $\ol{\cO}$.
\end{proof}

\begin{proof}[Proof of Lemma \ref{lemma:sqrt_A}]
 
The definition of $\Delta_K$ 
and $(H.\ref{assum:D})$
clearly imply 
 that the function $\tilde{b}$ is well-defined and enjoys the desired estimates.
Hence we shall focus on  establishing the properties of the function $\tilde{\sigma}$.

It has been shown in \cite[Theorem 7.14-3]{ciarlet2013} that for any given $A\in \bS^n_{>}$, there exists a unique matrix $A^{1/2}\in \bS^n_{>}$ such that $A^{1/2}(A^{1/2})^T=A$, 
$A^{1/2}\ge \sqrt{\mu} I_n$ if $A\ge \mu I_n$, 
and the mapping $\Phi:A\in \bS^n_{>}\mapsto \Phi(A)=A^{1/2}\in \bS^n_{>}$ is infinitely differentiable.
Note that  \eqref{eq:elliptic} and \eqref{eq:holder} in (H.\ref{assum:D}) ensure that
there exists a constant $C\in (0,\infty)$, such that it holds 
 for all $x\in \R^n$, $\lambda\in \Delta_K$ that
 \bb\l{eq:sum_sigma_sqrt}
  \sum_{k=1}^K {\sigma}(x,\ba_k){\sigma}(x,\ba_k)^T \lambda_k\in   
 G\coloneqq
 \bigg\{A=[a_{ij}]\in \bS^n_>\biggm\vert  A\ge \nu I_n, \sum_{i,j=1}^n|a_{ij}|\le C\bigg\}\subset \bS^n_>.
 \ee
We now define the function $\tilde{\sigma}:\R^n\t \Delta_K\to \bS^{n}_{>}$ by 
 $
 \tilde{\sigma}(x,\lambda)=\Phi\left(\sum_{k=1}^K {\sigma}(x,\ba_k){\sigma}(x,\ba_k)^T \lambda_k\right)
 $
 for all 
$x\in \R^n, \lambda\in \Delta_K$.
 The facts that $\Phi$ is a smooth function and  $G$ is  a compact subset of $\bS^n_>$
 imply that $\Phi$ is bounded and Lipschitz continuous on $G$.
Therefore, we can conclude from \eqref{eq:elliptic}, \eqref{eq:holder}, \eqref{eq:sum_sigma_sqrt} 
and the definition of $\tilde{\sigma}$ that
%there exists a constant $\Lambda'$ such that 
it holds for all  $x\in \R^n$, $\lambda\in \Delta_K$ that 
$ \tilde{\sigma}(x,\lambda)\ge \sqrt{\nu} I_n$
and 
$\sum_{i,j} |\tilde{\sigma}^{ij}(\cdot,\lambda)|_{0,1}<\infty$.
\end{proof}

\begin{proof}[Proof of Lemma \ref{lemma:rho_convex}]
We start by establishing Property \eqref{item:rho}.
Since $H:\R^K\to \R$ is a continuous convex  function, the representation of $\rho$ in (H.\ref{assum:rho}) and 
\cite[Theorem 12.2]{rockafellar1970} ensure that $\rho$ is a closed convex proper function satisfying 
\bb\l{eq:rho_*}
H(x)=\sup_{y\in \R^K} \big( y^Tx-\rho(y)\big) \q \fa x\in \R^K.
\ee
The assumption that 
$H(x)-c_0\le \max_{k\in \cK}x_k\le H(x)$ for all $x\in \R^K$ implies that 
for all $y\in \R^K$, 
\begin{align*}
\sup_{x\in \R^K}\bigg(x^Ty-
\left[\max_{k\in\cK}{x_k}+c_0\right]\bigg)\le 
\sup_{x\in \R^K}\bigg(x^Ty-H(x)\bigg)
=\rho(y)
\le 
\sup_{x\in \R^K}\bigg(x^Ty-\max_{k\in\cK}{x_k}\bigg),
\end{align*}
which together with the fact that 
$$
\sup_{x\in \R^K}\bigg(x^Ty-\max_{k\in\cK}{x_k}\bigg)
=
\begin{cases}
0, & y\in \Delta_K,\\
\infty, & y\not\in \Delta_K,
\end{cases}
$$
shows that $\rho(y)\in [-c_0,0]$ for all $y\in \Delta_K$ and $\rho(y)=\infty$ for all $y\in  (\Delta_K)^c$.
Finally, since $\rho$ is a closed convex function satisfying  $\{y\in \R^K\mid \rho(y)<\infty\}=\Delta_K$, 
we can deduce from 
\cite[Theorem 10.2]{rockafellar1970} ($\Delta_K$ is the standard simplex and hence locally simplicial)  that the restriction of $\rho$ to $\Delta_K$  is a continuous function.

We now show Property \eqref{item:H_eps}.
It is clear from (H.\ref{assum:rho}) and \eqref{eq:H_eps} that 
$H_\eps(x)- c_0\eps\le H_0(x)\le H_\eps (x)$ for all $x\in \R^K$.
Note that \eqref{eq:rho_*} and the fact that 
$\rho=\infty$ on $\Delta_K$
imply that for all $\eps>0$ we have
$$
\eps H(\eps^{-1}x)=\eps\max_{y\in \R^K}\big(y^T\eps^{-1}x-\rho(y)\big)=\max_{y\in \Delta_K}\big(y^Tx-\eps\rho(y)\big)
\q \fa x\in \R^K, 
$$
which shows 
the function  $\eps\rho$ is the convex conjugate of $H_\eps$, i.e., $(H_\eps)^*=\eps \rho$.
Hence, we can further deduce from 
\cite[Theorem 23.5]{rockafellar1970}, 
the differentiability and convexity of  $H_\eps$ that 
$$
(\nabla H_\eps)(x)=\argmax_{y\in \R^K}\big(y^Tx-( H_\eps)^*(y)\big)=\argmax_{y\in \Delta_K}\big(y^Tx-\eps\rho(y)\big)
\in \Delta_K
\q \fa x\in \R^K. 
$$
Consequently, 
we can obtain from 
the fundamental theorem of calculus and the Cauchy-Schwarz inequality that
%by using the fact that
%$\rho$ is the  conjugate of the proper convex function $H$, 
%and $\{y\in \R^K\mid \rho(y)<\infty\}=\Delta_K$, 
%we can obtain from \cite[Corollary 13.3.3]{rockafellar1970}
%that
$H_\eps$ is Lipschitz continuous with constant $L_{H_\eps}=\sup_{x\in \R^K}|(\nabla H_\eps)(x)|\le \max_{y\in \Delta_K}|y|$.
Note that $\Delta_K$ is the convex hull of $\{e_1,\ldots, e_K\}$, where $e_k$ is   the unit vector  from the
$k$-th column of the identify matrix $I_K$.
Hence
% \cite[Corollary 32.3.1]{rockafellar1970} and
%  \cite[Corollary 18.3.1]{rockafellar1970}
\cite[Theorem 32.2]{rockafellar1970}
  ensures that $\max_{y\in \Delta_K}|y|$ is attained at $\{e_1,\ldots, e_K\}$,
  which implies that $L_{H_\eps}\le 1$, 
%We then conclude from the definition of $H_\eps$ that 
%$H_\eps$ is Lipschitz continuous with constant 1,  
and finishes the proof of Lemma \ref{lemma:rho_convex}.
\end{proof}

Before establishing Proposition \ref{prop:Holder_bdd}, we first present an \textit{a priori} estimate for solutions of fully nonlinear equations involving only the second order term.
\begin{Lemma}\l{lemma:a_priori}{\cite[Theorem 7.2 on p.~125]{chen1998}}
Let $\cO$ be a bounded connected open  subset of $\R^n$,
and
$F:\cO\t \bS^n\to \R$ be a given function.
Suppose  
the function $F$ is  differentiable and convex in its second component,
and 
there exist constants 
 $\lambda,\Lambda>0$ such that 
$\lambda I_n\le \big[\f{\p F}{\p r_{ij}}(x,r)\big]\le \Lambda I_n$ for all
$(x,r)\in \cO\t \bS^n$.
Then there exists a constant $\a=\a(n,\Lambda/\lambda)\in(0,1)$ such that 
for any $\b\in (0,\a)$,
if  we have in addition that
$\p\cO\in C^{2,\b}$, $g\in C^{2,\b}(\ol{\cO})$,
and 
there exist constants $\gamma, {\mu}>0$ such that 
it holds for all $x,y\in \cO$, $r\in \bS^n$ that 
$|F(x,r)-F(y,r)|\le {\gamma}({\mu}+|r|)|x-y|^{\b}$,
then 
the  Dirichlet problem 
$$%\bb\l{eq:hjb_principal}
F(x,D^2u)=0
\q\textnormal{in $\cO$}, 
\qq
u=g \q \textnormal{on $\p\cO$,}
$$%\ee
 admits a unique solution $u\in C^{2,\b}(\ol{\cO})$ satisfying the  estimate
$[u]_{2,\b}\le C\big[|u|_0+|g|_{2,\b}+\mu\big]$,
where the constant $C$ depends only on $n$, $\Lambda/\lambda$, $\gamma$, $(\a-\b)^{-1}$ and the $C^{2,\b}$-norm of $\p\cO$.
\end{Lemma}

Now we proceed to prove the \textit{a priori} estimate for solutions to \eqref{eq:hjb_relax}.
\begin{proof}[Proof of Proposition \ref{prop:Holder_bdd}]
Throughout this proof, 
we shall denote by $C$ a generic constant, 
which   may take a different value at each occurrence.

Let $\phi\in C(\ol{\cO}) \cap C^2(\cO)$ be a given function, 
we consider  the Dirichlet problem
\bb\l{eq:F_phi}
{F}_\phi(x,D^2u(x))=0
\q\textnormal{in $\cO$}, 
\qq
u=g \q \textnormal{on $\p\cO$,}
\ee
where we define 
$D^2u(x)=[\p_{ij}u(x)]\in \bS^n$, and
the function ${F}_\phi:\cO\t \bS^n\to \R$ such that for all $x\in\cO$ and $r=[r_{ij}]\in \bS^n$,
$$
{F}_\phi(x,r)\coloneqq
H_\eps
\left(
\big(
a^{ij}_k(x)r_{ij}+b^i_k(x)\p_i\phi(x)-c_k(x) \phi(x)+f_k(x)
\big)_{k\in \cK}
\right).
$$
It follows from (H.\ref{assum:rho}) that ${F}_\phi$ is differentiable and convex in $r$.  Moreover, a straightforward computation shows 
for all $(x,r)\in \cO\t \bS^n$ that 
$\big[\tfrac{\p {F}_\phi}{\p r_{ij}}(x,r)\big]
=\sum_{k=1}^K
\eta_k(x,r)
a_k(x)$,
where we have 
$$ 
\eta_l(x,r)\coloneqq 
\p_l H_\eps
\left(
\big(
a^{ij}_k(x)r_{ij}+b^i_k(x)\p_i\phi (x)-c_k(x) \phi(x)+f_k(x)
\big)_{k\in \cK}
\right), \q l=1,\ldots, K.
$$
Note that for each $k\in \cK$, the fact that 
$a_k=\sigma\sigma^T/2$ and 
 (H.\ref{assum:D}) (see \eqref{eq:elliptic}-\eqref{eq:holder})
imply that 
there exists a constant $C$, depending only on $n$, such that 
for all $x\in \cO$, 
$$
\frac{\nu}{2}I_n\le a_k(x)\le 
\textrm{tr}\big(a_k^T(x)a_k(x)\big)I_n\le C\bigg( \sup_{i,j,k}|\sigma^{ij}_k|_{0;\ol{\cO}}\bigg)^4 I_n,
$$ 
which, along with the fact that 
$(\eta_1(x,r),\ldots, \eta_K(x,r))^T\in \Delta_K$
for all $(x,r)\in \cO\t \bS^n$ (see Lemma \ref{lemma:rho_convex}\eqref{item:H_eps}),
shows  that 
$\tfrac{\nu}{2}I_n\le \big[\tfrac{\p {F}_\phi}{\p r_{ij}}(x,r)\big]\le C I_n$,
for some constant $C$ depending only on $n$ and the constant $M$ 
defined in the statement of Proposition \ref{prop:Holder_bdd}.

The regularity of the coefficients in (H.\ref{assum:D}) 
and the Lipschitz continuity of $H_\eps$ (see Lemma \ref{lemma:rho_convex}\eqref{item:H_eps}) 
imply that,
 if the function $\phi\in C^{2,\eta}(\ol{\cO})$, $0<\eta\le \theta$, then the function $F_\phi$ satisfies for all $x,y\in \cO$, $r\in \bS^n$ that 
$$
|{F}_\phi(x,r)-{F}_\phi(y,r)|
\le C\Lambda (|r|+|\phi|_{1,\eta}+1)|x-y|^\eta,
$$
for some constant $C$ depending only on $n$.
Consequently, we can deduce from Lemma \ref{lemma:a_priori} that,
there exists a constant $\b_0=\b_0(n,\nu,M)\in (0,1)$, 
such that for all $\b\in (0,\min(\b_0,\theta)]$ and $\phi\in C^{2,\b}(\ol{\cO})$,  the Dirichlet problem \eqref{eq:F_phi} admits a unique solution  $u^\phi\in C^{2,\b}(\ol{\cO})$, and satisfies
$[u^\phi]_{2,\b}\le C\big[|u^\phi|_0+|g|_{2,\b}+|\phi|_{1,\b}+1\big]$,
where the constant $C$ depends only on $n$, $\nu$, $\Lambda$,  $\b$, and $\cO$.

Now let $u^\eps\in C^{2,\b}(\ol{\cO})$, $\b\in (0,\min(\b_0,\theta)]$ be a solution to \eqref{eq:hjb_relax}.  Then it is clear that $u^\eps$ is a solution to the Dirichlet problem:
${F}_{u^\eps}(x,D^2u(x))=0$ {in $\cO$} and  $u=g$ {on $\p\cO$}.
We can then deduce from the above arguments that,
there exists a constant  $C$, depending only on $n$, $\nu$, $\Lambda$, $\b$ and  $\cO$,
such that 
$[u^\eps]_{2,\b}\le C\big[|g|_{2,\b}+|u^\eps|_{1,\b}+1\big]$.
Hence by using the interpolation inequality (see \cite[Theorem 1.2 on p.~18]{chen1998}),
we have 
$|u^\eps|_{2,\b}\le C\big[|g|_{2,\b}+|u^\eps|_{0}+1\big]$.

It remains to estimate $|u^\eps|_{0}$. 
By using the fundamental theorem of calculus, we have
for all $x\in \cO$ that 
\begin{align*}
-H_\eps(\bf(x))&=H_\eps(\bL u^\eps(x)+\bf(x))-H_\eps(\bf(x))=
\int_0^1 (\nabla H_\eps)^T(s\bL u^\eps(x)+\bf(x))\bL u^\eps(x)\,ds,
\end{align*}
from which, by using the  classical maximum principle (see e.g.~\cite[Theorem 3.7]{gilbarg1985}) and the fact that $\nabla H_\eps \in \Delta_K$ (see Lemma \ref{lemma:rho_convex}\eqref{item:H_eps}),
we can deduce that, there exists a constant $C=C(n,\Lambda,\cO)>0$ that 
$$
|u^\eps|_0\le C\bigg(\sup_{x\in \p\cO}|g(x)|+|H_\eps(\bf)|_0\bigg)
\le C\big(|g|_{0;\ol{\cO}}+|H_0(\bf)|_0+\eps c_0\big)
\le C\big(|g|_{0;\ol{\cO}}+1+\eps c_0\big),
$$
which together with the fact that 
$|u^\eps|_{2,\b}\le C\big[|g|_{2,\b}+|u^\eps|_{0}+1\big]$ leads to the desired estimate.
\end{proof}

\begin{proof}[Proof of Proposition \ref{prop:schauder_eps}]
The well-posedness of the classical solution $w^\eps$ follows  from  
 the standard elliptic regularity theory (see \cite[Theorem 6.14]{gilbarg1985}),
 hence it suffices to  prove the \textit{a priori} estimate for a fixed $\eps>0$.

Let $\rho>0$ be a constant whose value will be specified later,
and 
$(\xi_m)_{m=1}^M$ be   a partition of unity   in a domain containing $\ol{\cO}$
such that the following properties hold: 
 (1) the support of each function $\xi_m$ is contained in a ball $B_\rho(x_m)$ for some $x_m\in \R^n$;
 (2) $\xi_m\in C^\infty(\R^n)$ satisfies 
 for all $\gamma\ge 0$ that $|\xi_m|_{\lfloor \ga \rfloor, \ga-\lfloor \ga \rfloor}\le C_\gamma \rho^{-\gamma}$,
 where $\lfloor \ga \rfloor$ is the integer part of $\ga$
 and $C_\ga$ is a constant independent of $m$ and $\ga$;
 (3) for each $x\in \ol{\cO}$, 
  $\sum_{m=1}^M \xi_m(x)= 1$ 
  and the number of intersected supports of $(\xi_m)_{m=1}^M$ at $x$
  is bounded by a constant $M_n$ depending only on the dimension $n$.
In the following, we shall denote by $w$ the  solution  $w^\eps$, 
and  by $C$ a generic constant 
 independent of $\a$, $m$ and $\eps$.
%may take a different value at each occurrence.

For each $m=1,\ldots, M$,
 we define the function $w_m=w\xi_m$,
 which satisfies $w_m=g\xi_m$ on $\p\cO$ and 
\begin{align*}
a^{ij}_\eps(x_m)\p_{ij}w_m&=(a^{ij}_\eps(x_m)-a^{ij}_\eps)\p_{ij}w_m
+ a^{ij}_\eps(\p_j w \p_i\xi_m+\p_i w \p_j\xi_m+ w \p_{ij}\xi_m)+\tilde{f}
,\q \textnormal{in $\cO$},
\end{align*}
where 
$\tilde{f}\coloneqq (-b^i_\eps\p_iw + c_\eps w- f)\xi_m$.
Hence applying the classical Schauder estimate 
% to the constant coefficient operator $a^{ij}_\eps(x_m)\p_{ij}$
yields that 
 \begin{align}\l{eq:w_m}
|w_m|_{2,\b}&\le C\big(
|(a^{ij}_\eps(x_m)-a^{ij}_\eps)\p_{ij}w_m|_\b
+ |a^{ij}_\eps(\p_j w \p_i\xi_m+\p_i w \p_j\xi_m+ w \p_{ij}\xi_m)|_\b+|\tilde{f}|_\b+|g\xi_m|_{2,\b}
\big)
\end{align}
for  some constant $C=C(n,\b,\nu,\Lambda,\cO)$. 

%We now  estimate the terms on the right-hand side of \eqref{eq:w_m}. 
Note that by choosing $\rho=(2C\Lambda \eps^{-\a})^{-1/\b}$,
we have for all $x\in B_\rho(x_m)$ and $i,j=1,\ldots,n$ that 
$$
|(a^{ij}_\eps(x_m)-a^{ij}_\eps(x)|\le [\a^{ij}_\eps]_\b |x-x_m|^\b\le [\a^{ij}_\eps]_\b \rho^\b\le 1/(2C),
$$
which together with the fact that $\p_{ij}w_m=0$ on $\ol{\cO}\setminus B_\rho(x_m)$ implies that 
 \begin{align*}
&|(a^{ij}_\eps(x_m)-a^{ij}_\eps)\p_{ij}w_m|_{\b;\ol{\cO}}
= |(a^{ij}_\eps(x_m)-a^{ij}_\eps)\p_{ij}w_m|_{\b;\ol{\cO}\cap B_\rho(x_m)} \\
&\le  |a^{ij}_\eps(x_m)-a^{ij}_\eps|_{0;\ol{\cO}\cap B_\rho(x_m)} |\p_{ij}w_m|_{\b;\ol{\cO}\cap B_\rho(x_m)}
+|a^{ij}_\eps(x_m)-a^{ij}_\eps|_{\b;\ol{\cO}\cap B_\rho(x_m)} |\p_{ij}w_m|_{0;\ol{\cO}\cap B_\rho(x_m)}\\
&\le  |w_m|_{2,\b;\ol{\cO}}/(2C)+[a^{ij}]_\b |w_m|_{2;\ol{\cO}}
\le  |w_m|_{2,\b;\ol{\cO}}/(2C)+\Lambda\eps^{-\a} |w_m|_{2;\ol{\cO}}.
\end{align*}
Then we can deduce from  the interpolation inequality (see \cite[Theorem 1.3 on p.~19]{chen1998}) 
and \eqref{eq:w_m} that 
\begin{align}\l{eq:w_m2}
|w_m|_{2,\b}&\le C\big(
\eps^{-\a (\b+2)/\b}|w_m|_{0}
+ |a^{ij}_\eps(\p_j w \p_i\xi_m+\p_i w \p_j\xi_m+ w \p_{ij}\xi_m)|_\b+|\tilde{f}|_\b+|g\xi_m|_{2,\b}
\big).
\end{align}
Note that   for all  $\gamma\ge 0$, we can obtain from  property (2) of $(\xi_m)_{m=1}^M$ that 
 $|\xi_m|_{\lfloor \ga \rfloor, \ga-\lfloor \ga \rfloor}\le C_\gamma (2C\Lambda \eps^{-\a})^{\gamma/\b}$.
Hence by  repeatedly applying interpolation inequalities, we can 
 simplify \eqref{eq:w_m2} into
\begin{align*}
|w_m|_{2,\b}&\le C\big(
\eps^{-\a (\b+2)/\b}|w_m|_{0}
+ \eps^{-\a }|f|_{0}+ [f]_{\b}+\eps^{-\a (\b+2)/\b}|g|_{2,\b}
\big),
\end{align*}
which along with properties (2) and (3) of $(\xi_m)_{m=1}^M$ leads  to the estimate that
\begin{align*}
|w|_{2,\b}\le 2M_n\max_m |w_m|_{2,\b}
\le C\big(
\eps^{-\a (\b+2)/\b}|w|_{0}
+ \eps^{-\a }|f|_{0}+ [f]_{\b}+\eps^{-\a (\b+2)/\b}|g|_{2,\b}
\big).
\end{align*}
Finally, we can conclude from the 
classical maximum principle (see e.g.~\cite[Theorem 3.7]{gilbarg1985}) that
$|w|_{0}\le C(|f|_{0}+|g|_{0})$, which finishes the proof of 
the desired \textit{a priori} estimate.
\end{proof}

\begin{proof}[Proof of Lemma \ref{lemma:S_loc}]
We first establish Property \eqref{item:composition}. 
For any given  $x=(x_1,\ldots, x_{n_2+n_3})^T\in \R^{n_2+n_3}$, we write  
$x^{(1)}=(x_1,\ldots, x_{n_2}) \in \R^{n_2}$ and $x^{(2)}=(x_{n_2+1},\ldots, x_{n_2+n_3})\in  \R^{n_3}$.

Let $x\in \R^{n_2+n_3}$ satisfy for some $k\in \{1,\ldots, n_2+n_3\}$  that $x_k\ge \max_{j\not =k}x_j+c$ with $c=\max(\vartheta_2,\vartheta_3,c_2+\vartheta_1,c_3+\vartheta_1)$. We assume without loss of generality  that $k\le n_2$. 
Then since $\phi_2$ satisfies 
($S_{\textnormal{loc}}$) with $\vartheta_2$ and $c\ge \vartheta_2$,
we have  that  $\phi_2(x^{(1)})=x_k$ and $\phi_3(x^{(2)})\le H^{(n_3)}_0(x^{(2)})+c_3$. Moreover, since $x_k\ge H^{(n_3)}_0(x^{(2)})+c$ and $c\ge c_3+\vartheta_1$, we see $\phi_2(x^{(1)})\ge \phi_3(x^{(2)})+\vartheta_1$, which, along with the assumption that  
$\phi_1$ satisfies ($S_{\textnormal{loc}}$) with $\vartheta_1$, implies $\phi(x)=\phi_2(x^{(1)})=x_k$. 
Similar arguments show that the same conclusion holds if $k\ge n_2+1$, which enables us to conclude that 
$\phi$ satisfies ($S_{\textnormal{loc}}$) with $c$.

Now let $x\in \R^{n_2+n_3}$ be an arbitrary given point. 
We have by assumptions that $\phi_2(x^{(1)})\le H^{(n_2)}_0(x^{(1)})+c_2$ and $\phi_3(x^{(2)})\le H^{(n_3)}_0(x^{(2)})+c_3$. Hence, 
by using the fact that $H_0^{(2)}$ is  component-wise increasing and subadditive on $\R^2$,
we have 
\begin{align*}
\phi(x)&= \phi_1(\phi_2(x^{(1)}),\phi_3(x^{(2)}))\le H^{(n_1)}_0(\phi_2(x^{(1)}),\phi_3(x^{(2)}))+c_1\\
&\le H^{(n_1)}_0(H^{(n_2)}_0(x^{(1)}),H^{(n_3)}_0(x^{(2)}))+c_1+\max(c_2,c_3)=H^{(n_2+n_3)}_0(x)+c_1+\max(c_2,c_3),
\end{align*}
which finishes the proof of Property \eqref{item:composition}. Property  \eqref{item:scaling} follows directly from the definition of $\phi_\eps$.
\end{proof}

%\bibliography{bi}
%\bibliographystyle{plain}

\end{document}